\documentclass[reqno,oneside,11pt]{amsart}
\usepackage{amssymb,eucal,array,enumitem,setspace,color,multirow}
 \usepackage{kpfonts}
\usepackage[T1]{fontenc}
\usepackage{mathtools} 
\usepackage{geometry}
\usepackage{subcaption}
\usepackage{microtype}

\usepackage{tikz}
\usetikzlibrary{decorations.pathmorphing,decorations.pathreplacing,cd}
\tikzset{>=to}

\definecolor{rouge}{rgb}{0.85,0.1,0.15}
\definecolor{forestgreen}{rgb}{0.13,0.54,0.13}
\definecolor{vertforet}{RGB}{237, 135,45} 
\definecolor{bleu}{rgb}{0.1,0.2,0.8}
\definecolor{g-darkgreen}{RGB}{7, 141, 112}
\definecolor{g-green}{RGB}{38,206,170}
\definecolor{g-lightgreen}{RGB}{152,232,193}
\definecolor{g-lightblue}{RGB}{127,173,226}
\definecolor{g-indigo}{RGB}{80, 73, 204}
\definecolor{g-blue}{RGB}{61,26,120}
\definecolor{gq-green}{RGB}{74,129,34}
\definecolor{gq-mauve}{RGB}{181,126,220}

\usepackage[pdfpagemode=UseNone, pdfstartview={XYZ null null null}]{hyperref}
\hypersetup{
colorlinks=true,
citecolor=g-darkgreen,
linkcolor=g-indigo,
urlcolor=g-indigo
}
\usepackage{thmtools}
\usepackage[capitalise,noabbrev]{cleveref} %

\usepackage[hyperref,
maxnames=4,
giveninits
]
{biblatex}
\renewbibmacro{in:}{} 
\AtBeginBibliography{\footnotesize} 

\setlist[itemize]{leftmargin=*}                  
\setlist[enumerate]{leftmargin=*,                
label=\textup{(\roman*)}}    

\numberwithin{equation}{section}

\newcommand{\fld}[1]{\mathbb{#1}} 
\newcommand{\ZZ}{\fld{Z}}
\newcommand{\NN}{\fld{N}}

\newcommand{\CC}{\fld{C}}

\newcommand{\eps}{\varepsilon}

\newcommand{\alg}[1]{\mathcal{#1}} 

\newcommand{\ca}[1]{\mathcal{#1}} 
\newcommand{\dih}{\mathsf{D}} 
\newcommand{\p}[1]{\rlap{\,#1}}

\newcommand{\modSB}{~\operatorname{mod}\,} 

\DeclareMathOperator{\sgn}{sgn}

\DeclareMathOperator{\End}{End}

\theoremstyle{plain}
\newtheorem{theorem}{Theorem}[section]
\newtheorem{lemma}[theorem]{Lemma}
\newtheorem{proposition}[theorem]{Proposition}
\newtheorem{example}[theorem]{Example}

\newtheorem{definition}[theorem]{Definition}

\newtheorem{notation}[theorem]{Notation}
\newtheorem*{remark}{Remark}

\newcommand{\sym}[1]{O_{#1}} 
\newcommand{\acomm}[2]{\left\lbrace #1,\, #2 \right\rbrace} 
\newcommand{\dcover}[1]{\widetilde{#1}}
\newcommand{\Lad}[1]{L_{#1}} 
\newcommand{\Clif}{\mathcal{C}}
\newcommand{\Lie}[1]{\mathfrak{#1}} 

\newcommand{\DDop}{\underline{D}} 

\newcommand{\dsig}[1]{\dcover{s}_{#1}}
\newcommand{\dtau}[1]{\dcover{t}_{#1}}
\newcommand{\Ortho}{\mathsf{O}}

\newcommand{\bbc}{\mathbb{C}}
\newcommand{\bbr}{\mathbb{R}}
\newcommand{\bbs}{\mathbb{S}}
\newcommand{\bbz}{\mathbb{Z}}
\newcommand{\ux}{\underline{x}}
\newcommand{\up}{\DDop} 
\newcommand{\fg}{\mathfrak{g}}
\newcommand{\tama}{\mathfrak{O}_\kappa}
\DeclareMathOperator{\ad}{ad}
\DeclareMathOperator{\Ad}{Ad}
\DeclareMathOperator{\id}{Id}

\newcommand{\Mab}{L_{\ib{++}{12}}}
\newcommand{\Mad}{L_{\ib{+-}{12}}}
\newcommand{\Mcb}{L_{\ib{-+}{12}}}
\newcommand{\Mcd}{L_{\ib{--}{12}}}

\newcommand{\La}{L_{\ib{+}{1}}}
\newcommand{\Lc}{L_{\ib{-}{1}}}
\newcommand{\Lb}{L_{\ib{+}{2}}}
\newcommand{\Ld}{L_{\ib{-}{2}}}

\newcommand{\dr}[1]{\dcover{r}_{#1}}
\newcommand{\df}[1]{\dcover{f}_{#1}}

\newcommand{\ib}[3][A]{\smash{\overset{#2}{#3\vphantom{#1}}}} 

\newcommand{\Tepsun}[1]{O_{\ib{#1+-}{122}}}
\newcommand{\Tepsdeux}[1]{O_{\ib{#1+-}{211}}}
\newcommand{\Tepsa}[1]{O_{\ib{#1+-}{abb}}}
\newcommand{\Tepsb}[1]{O_{\ib{#1+-}{baa}}}

\newcommand{\hwv}{\mathfrak{v}} 
\newcommand{\wv}{\mathfrak{u}} 
\newcommand{\TriSubAlg}{\mathfrak{T}}

\newcommand{\Mam}{N_{\ib{-}{1}}}
   \newcommand{\Mbm}{N_{\ib{-}{2}}}
   \newcommand{\Nmp}{N_{\ib{-+}{12}}}
   \newcommand{\Nmm}{N_{\ib{--}{12}}}

\newcommand{\linkbetween}[5] 
{
    \draw[color = #5] (#1,#2) circle(3pt);
    \draw[color = #5] (#3,#4) circle (3pt);
    \draw[very thick, color = #5, opacity = 0.6] (#1,#2) -- (#3,#4);
}

\definecolor{rouge}{rgb}{0.85,0.1,0.15}
\definecolor{forestgreen}{rgb}{0.13,0.54,0.13}
\definecolor{bleu}{rgb}{0.1,0.2,0.8}

\begin{document}
\allowdisplaybreaks

\title{The double dihedral Dunkl total angular momentum algebra}

\author[M De Martino]{Marcelo De Martino}
\address[Marcelo De Martino]{ Clifford Research Group,
Department of Electronics and Information Systems,
Faculty of Engineering and Architecture,
Ghent University,
Krijgslaan 281--S8, 9000 Gent, Belgium. \emph{Now at} Forward College, Rua das Flores 71, 1200-193 Lisboa, Portugal}
\email{Marcelo.demartino@forward-college.eu; \href{https://orcid.org/0000-0002-1851-1090}{ORCID:0000-0002-1851-1090}}
\author[A Langlois-R\'emillard]{Alexis Langlois-R\'emillard}
\address[Alexis Langlois-R\'emillard]{ Department of Applied Mathematics, Computer Science and Statistics, Faculty of Sciences, Ghent University, Krijgslaan 281--S9, 9000 Gent, Belgium. \emph{Now at} 
Hausdorff Center for Mathematics, 53115 Bonn, Germany.
}
\email{langlois@uni-bonn.de; \href{https://orcid.org/0000-0002-5919-8766}{ORCID:0000-0002-5919-8766}}
\author[R Oste]{Roy Oste}
\address[Roy Oste]{ Department of Applied Mathematics, Computer Science and Statistics, Faculty of Sciences, Ghent University, Krijgslaan 281--S9, 9000 Gent, Belgium.
}
\email{Roy.Oste@UGent.be; \href{https://orciSuperintegrability of the Dunkl–Coulomb problem in three-dimensionsd.org/0000-0002-3418-7067}{ORCID:0000-0002-3418-7067}}

\date{\today}

\keywords{Total angular momentum algebra; 
Dunkl--Dirac operator; Rational Cherednik algebra; Dihedral root system}
\subjclass{ 
16S80; 
17B10; 
20F55; 
81R12 
}
\begin{abstract}
The Dunkl total angular momentum algebra (TAMA) is realised as the dual partner of the orthosymplectic Lie superalgebra containing the Dunkl deformation of the Dirac operator. In this paper, we consider the case when the reflection group associated with the Dunkl operators is a product of two dihedral groups acting on a four-dimensional Euclidean space. We show that in this case there is a subalgebra of the total angular momentum algebra that admits a triangular decomposition. In analogy to the celebrated theory of semisimple Lie algebras, we use this triangular subalgebra to give precise necessary conditions that a finite-dimensional irreducible representation must obey, in terms of weights. In specific cases, which includes unitary representations, we construct a basis of weight vectors with explicit actions of all TAMA elements. Examples of these modules occur in the kernel of the Dunkl--Dirac operator in the context of deformations of Howe dual pairs.
\end{abstract}

\maketitle


\section{Introduction}

Total angular momentum operators naturally appear as symmetries of the Dirac operator in the study of the motion of fermions.
These operators form a realisation of the orthogonal Lie algebra associated with the $\mathrm{Spin}$- or $\mathrm{Pin}$-group, 
which fits in the framework of Howe dual pairs \cite{BDES10,Ho95}. As a consequence of this is that it allows for an algebraic description of the spinor-valued polynomials in the kernel of the Dirac operator, also called monogenic polynomials, in terms of $\mathrm{Spin}$-modules.

In this paper, we are interested in the Howe duality related with the Dunkl version of the Dirac operator.
The study of Dunkl deformations of Howe dual pairs is a recent area of research \cite{ciubotaru_deformations_2020,DBGV3d16,DBGV16,DBLROVdJ22,de_bie_algebra_2018,OSS09}. In all these cases, the main difficulty is that the full invariance for the orthogonal group is not present in the Dunkl setting. One is forced to deal with a finite reflection group and needs to compute the full centraliser algebra of the relevant Lie (super)algebra in question. This is a challenging setup as, typically, the centraliser algebra that arises has a rather complicated structure, despite its resemblance with universal enveloping algebras from Lie theory.

In general, let $V_\bbr$ be a Euclidean space (with its complexification denoted by $V$), $W\subset\Ortho(V_\bbr)$ be a real reflection group with root system $R$ and  $\kappa:R\to \CC$ be a $W$-invariant function. Attached to this data is the rational Cherednik algebra $H_{\kappa} =H_{\kappa}(V,W)$~\cite{EG02}. It has a faithful representation inside the endomorphism algebra of the polynomial representation where the partial derivatives are replaced by the Dunkl operators. It is well known \cite{He91} that the classical $\Lie{sl}(2)$-triple containing the Laplace operator is also present in $H_{\kappa}$ and the subalgebra $H_{\kappa}^{\Lie{so}(d)}\subset  H_{\kappa}$ of elements commuting with the $\Lie{sl}(2)$-triple was studied by Feigin and Hakobyan in \cite{FH15} (see also \cite{CaDM22,CDM18}). This algebra is known as the Dunkl angular momentum algebra for its relations with the angular momentum operators in the context of Calogero--Moser integrable systems (a very similar algebra, for the case $W=S_n$, has appeared in~\cite[Sect.~8, arXiv version]{Kuz96}). 

Now let $\Clif(V)$ be the Clifford algebra associated with $V$ and the bilinear symmetric form associated with its Euclidean structure. There is a realisation of the orthosymplectic Lie superalgebra $\Lie{osp}(1|2)$ inside $\Clif(V)\otimes H_{\kappa}$~\cite{OSS09}, and this algebra contains the Dunkl--Dirac operator. The relevant dual pair we are seeking to deform is ($\mathsf{Pin}(V_\bbr),\Lie{osp}(1|2)$) and the main algebra of study to us is $\tama = \tama(V,W)$, the (graded) centraliser algebra of $\Lie{osp}(1|2)$ inside $\Clif(V)\otimes H_{\kappa}$.
We shall call this algebra the \emph{Dunkl total angular momentum algebra}, or TAMA in short. In \cite{CDMO22}, the centre of this algebra was determined, and in \cite{Os22}, several structural properties of the TAMA were discussed, including a natural set of generators and commutation relations amongst them. However, at the moment, there is no known presentation of this algebra in terms of generators and relations, which both complicates and renders the theory interesting.

So far, we have discussed the rather general framework within which the algebra $\tama$ is inserted and let us now describe what is achieved in this paper. Building on previous studies for the cases in which $W=\ZZ_2^3\subset\Ortho(3)$~\cite{DBGV3d16,Huang22}, $W=\ZZ_2^d\subset \Ortho(d)$~\cite{DBGV16} (these papers were accomplished in the context of the Bannai--Ito algebra), $W=S_3\subset\Ortho(3)$~\cite{DBOVDJ18}, $W=G_2\subset\Ortho(3)$~\cite{LRO19} and $W=\dih_{2m}\times \ZZ_2\subset \Ortho(3)$~\cite{DBLROVdJ22} (this last work includes the complete classification of the finite-dimensional irreducible representations, along with the suitable restrictions on $\kappa$ for the existence of a unitarity structure), we aim to study the representation theory of the TAMA when $\dim V_\bbr = 4$ and $W=\dih_{2m_1}\times \dih_{2m_2}\subset \Ortho(4)$. 
A recent result of~\cite{Os22} states that the Dunkl total angular momentum algebra is generated by two- and three-index symmetries. Dimension four marks a change of nature of the algebra: in dimension three, the sole three-index symmetry super-anticommutes with all the elements of the algebra, but in higher dimensions the many three-index symmetries have more involved commutation relations. This is one of the motivations to study the Dunkl total angular momentum algebra for a space of dimension four specifically, as an important intermediate step toward a space of arbitrary dimension.

Apart from the change in structure that happens in dimension 4, which makes the study of this specific case interesting from a mathematical point of view, it is also important to study low-dimension examples in depth for their physical applications. Howe dual pairs are a well-studied topic in physics; see~\cite{BJMM20} for a recent review. However, we focus here on a Howe duality not covered by the aforementioned source, as it involves superalgebras~\cite{ChengWang12}. This appears in the context of physics in the study of higher-spin algebras~\cite{FN94,JM14,Ni91,Va04}. Specifically for dimension 4, recent work by Basile and Dhasmana~\cite{BD24} considered many dualities and potential candidates for deformation. Dunkl formalism in physics can be a way to provide such deformations, but much of the work studying low-dimensional cases in depth only considered $W=\mathbb Z_2^n$ and often only their realisation as polynomial representations~\cite{GMVZ14,GMVZ13,GLVZ13,Gha21,GS19,SH23}. There has been recent interest in providing worked-out examples for further, more complex, groups~\cite{DBLROVdJ22,Dun23Sigma,Dun23Jphy,FH22}. The representation theory for these groups is much richer, and we present the tools to describe it in enough generality to allow for a future extension to an arbitrary product of dihedral groups.

The main results of this article are a characterisation of the structure of finite-dimensional irreducible representations with respect to a triangular subalgebra (\cref{thm:AllWeightSpace}) and a coarse classification (see \cref{sec:CoClass}) of finite-dimensional irreducible representations of $\tama$ (Theorem~\ref{thm:ClassiLabel}). Moreover, we study properties of unitary finite-dimensional representations (\cref{prop:UnitModule}). When the parameter function is ``small'' (Definition~\ref{def:smallkappa}), we show that only specific representations can appear as finite-dimensional irreducible $\tama$-representations (\cref{thm:ClassiSmallKappa}). Such representations resemble the monogenic polynomials and we call them ``triangle-representations''  (\cref{def:mono-type}). For these representations, we are able to describe a basis of weight vectors with explicit actions (\cref{prop:MonoMaxi}).

To achieve the results described in the previous paragraph, we introduce a subalgebra $\TriSubAlg$ of $\tama$ with a ``triangular'' factorisation (\cref{prop:triangstruc}) and a weight theory to essentially capture the full action of $\tama$ (\cref{prop:WS_1D_If_HWS_1D}).  The existence of this triangular subalgebra is particular to the double dihedral case and is not guaranteed for general $W$. Nevertheless, this subalgebra is well suited to treat the case of arbitrary products of dihedral groups and hence the case studied in this article is a stepping stone to a study of the general representation theory when $W=\dih_{2m_1}\times \dots \times \dih_{2m_n}\subset \Ortho(2n)$.

We now go through the structure of the article. First, Section~\ref{sec:IniDefs} introduces the needed conventions and notions on dihedral groups and double coverings. The Dunkl TAMA for the group $W=\dih_{2m_1}\times\dih_{2m_2}$ is then presented in Section~\ref{sec:tama}. Section~\ref{sec:ladder} defines ladder operators and the triangular subalgebra $\TriSubAlg$. Finally, Section~\ref{sec:rep} contains the main results of the papers, their proofs, and examples. 

\section{Dihedral groups and double coverings}\label{sec:IniDefs}

Let $V_{\mathbb{R}}\cong \mathbb{R}^N$ be an $N$-dimensional Euclidean space endowed with the standard inner product $(\cdot,\cdot)$. From Section~\ref{ssec:Dihedral} on, we specialise to $N=4$.
We let $V\cong\mathbb{C}^N$ denote the complexification of $V_{\mathbb{R}}$. We will also denote by $(\cdot,\cdot)$ the non-degenerate complexified symmetric bilinear form on $V$ and denote by $B:V\to V^*$ the isomorphism induced by $(\cdot,\cdot)$, which is defined by $B(x)(y) = (x,y)$, for all $x,y\in V$.
Furthermore we denote by $x\mapsto \bar x$ the involution on $V = V_{\mathbb{R}}\otimes \mathbb{C}$ induced by complex conjugation.

\subsection{Clifford algebras}
We recall the definition and some properties of the Clifford algebra associated with $V_\mathbb{R}$. Let  $\ca C_{\mathbb{R}}$ denote the real Clifford algebra of the pair $(V_{\mathbb{R}},2(\cdot,\cdot))$. Specifically, this algebra is the quotient of the tensor algebra $T(V_{\mathbb{R}}) = \oplus_{j\geq 0} T^j(V_{\mathbb{R}})$ by the nonhomogeneous quadratic ideal $I$ generated by the elements of the form  $x\otimes y + y\otimes x -2(x,y)1$ for $x,y\in V_{\mathbb{R}}$.

We let $\gamma:V_{\mathbb{R}}\to\ca C_{\mathbb{R}}$ be the canonical embedding. If we denote the chosen $(\cdot,\cdot)$-orthonormal basis of $V_{\mathbb{R}}$ by $\{x_1,\ldots, x_N\}$, we let $e_j =\gamma(x_j)$ so that the Clifford algebra has the usual presentation as the unital associative algebra generated by $e_1,\ldots, e_N$ subject to the relations $\acomm{e_j}{e_k} = e_je_k + e_ke_j = 2\delta_{jk}$, for all $j,k =1,\ldots,N$. More generally, let $\bigwedge(V_\mathbb{R}) = \oplus_{p\geq 0} \bigwedge^p(V_\mathbb{R})$ be the exterior algebra on $V_\mathbb{R}$. Extend $\gamma:V_{\mathbb{R}}\to \ca C_{\mathbb{R}}$ to a linear isomorphism $\gamma: \bigwedge(V_{\mathbb{R}})\to \ca C_{\mathbb{R}}$ by declaring $\gamma(1):=1$ and, for each $p > 0$,
\begin{equation}\label{eq:quantization}
\gamma(v_1\wedge \cdots \wedge v_p) := \frac{1}{p!}\sum_{g\in S_p} \sgn(g)\gamma(v_{g(1)})\cdots \gamma(v_{g(p)}),
\end{equation}
for any $p$-tuple $(v_1,\ldots,v_p)$ of elements of $V_\mathbb{R}$. In particular, for each ordered subset $A = \{a_1,\ldots,a_p\}\subseteq \{1,\ldots,N\}$, we let $x_A:= x_{a_1}\wedge\cdots\wedge x_{a_p}\in \bigwedge^p(V_\mathbb{R})$, with $x_\emptyset := 1$ and put $e_A:=\gamma(x_A) = e_{a_1} \cdots e_{a_p}$. The set $\{e_A\mid A\subseteq \{1,\ldots, N\}\}$ forms a linear basis of $\ca C_\mathbb{R}$.

The Clifford algebra $\ca C_{\mathbb R}$ has a natural $\bbz_2$-grading $\ca C_{\mathbb{R}} = \ca C_{\mathbb{R},\bar 0} \oplus \ca C_{\mathbb{R},\bar 1}$ where $\ca C_{\mathbb{R},\bar{\jmath}}$, for $\jmath\in \{0,1\}$, is the image of $\sum_{k\geq 0} T^{2k+\jmath}(V_{\mathbb{R}})$. Finally, we let $\ca C$ denote the complexification of $\ca C_{\mathbb{R}}$ and we write $\ca C = \ca C_{\bar 0}\oplus \ca C_{\bar 1}$ for the induced $\mathbb{Z}_2$-grading. We shall also denote  by $\gamma$ the complexified isomorphism $\gamma:\bigwedge(V)\to \ca C$ defined in (\ref{eq:quantization}), and note that $\{e_A\mid A\subseteq \{1,\ldots,4\}\}$ is also a linear basis for $\ca C$.

Recall that if $\alg{A}$ is a unital, associative $\mathbb{Z}_2$-graded algebra, the graded commutator is defined as 
\begin{equation}\label{eq:Z2_products}
\llbracket a,b \rrbracket := ab -(-1)^{|a||b|}ba
\end{equation}
for homogeneous elements $a,b$ of degree $|a|$ and $|b|$, respectively, and extended accordingly. For $a,b\in \alg{A}$, 
we will denote the (ungraded) commutator
by $[a,b] := ab-ba$ and the anti-commutator by $\{a,b\}:=ab+ba$.

\subsection{Dihedral groups, double coverings and realisations}

In this section, we shall consider the double coverings of a product of dihedral groups. Before specialising to dihedral groups, we recall some general features of double coverings of real reflection groups. We follow the seminal works of Morris~\cite{Mo76,Mo80} and Schur~\cite{Sc11}. Suppose
that $W$ is a real reflection group of arbitrary rank $N$ and that $V_{\mathbb{R}}\cong \mathbb{R}^N$ is its reflection representation. There are two double coverings of $W$, a positive and a negative, reflecting the two possibilities for a definite symmetric bilinear form on $V_{\mathbb{R}}$. We shall only consider the positive double covering, which is compatible with our convention $e_j^2=1$ for the Clifford algebra. Let
\begin{equation}\label{eq:DoubCovSES}
    1 \to \mathbb{Z}_2 \to \mathsf{Pin}(N) \stackrel{\pi}{\to} \mathsf{O}(N) \to 1
\end{equation}
be the exact sequence that defines the double covering of the orthogonal group (see \cite[Theorem~2.14]{Mo76}) as a central extension. Here, $\mathsf{Pin}(N)$ (see \cite[Definition 2.12]{Mo76}) is a subgroup of $\ca C^\times$, the units of the (complexified) Clifford algebra associated with a positive-definite symmetric bilinear form of $V_{\mathbb{R}}$. The double covering of the subgroup $W \subset \mathsf{O}(N)$ is defined as $\dcover{W}:=\pi^{-1}(W)$. Because of this realisation of $\dcover{W}$ as a subgroup of the group of units in a Clifford algebra, we can use the $\mathbb{Z}_2$-grading of the latter to decompose
\begin{equation}\label{eq:DoubleCoverDecomp}
\dcover{W} = \dcover{W}_{\bar 0} \cup \dcover{W}_{\bar 1},
\end{equation}
into even and odd parts.

We let $z$ denote the generator of the kernel of $\pi$ in (\ref{eq:DoubCovSES}). Let also $e_\pm := \tfrac{1\pm z}{2}$ denote the canonical idempotents of $\mathbb{C}\dcover{W}$ and put $\mathbb{C}\dcover{W}_\pm := e_\pm(\mathbb{C}\dcover{W})$. Then
\begin{equation}\label{eq:splitting}
\mathbb{C}\dcover{W} = \mathbb{C}\dcover{W}_+ \oplus \mathbb{C}\dcover{W}_-.
\end{equation}
Moreover, $\mathbb{C}\dcover{W}_+\cong \mathbb{C}W$. Indeed, note that  if $\{\dcover{w},z\dcover{w}\} = \pi^{-1}(w)$ for $w\in W$, then $e_+\dcover w = e_+(z\dcover w)$, so the assignment $e_+\dcover w \mapsto \pi(\dcover w)$ defined on the canonical generators of $\mathbb{C}\dcover{W}$ is well-defined and induces the isomorphism. In light of (\ref{eq:splitting}), the representations of $\dcover{W}$ are split in two types: the first are the linear representations that factor through the action of $W$, and the second are the \emph{spin representations}. They are distinguished by the action of the central element $z$ of order 2: if it acts as $-1$ in the representation, then the representation is a spin representation. We shall denote by $\textup{Irr}(\dcover{W})$ the set of equivalence classes of irreducible representations of $\dcover{W}$ and by
$\textup{sIrr}(\dcover{W})\subset\textup{Irr}(\dcover{W})$
the subset of equivalence classes of irreducible spin representations. Finally, we remark that if $\rho:\bbc \dcover{W} \to \bbc W\otimes \ca C$ is the diagonal algebra homomorphism defined on the generators $\dcover w,z \in \dcover W$ by $\rho(\dcover w) = \pi(\dcover w)\otimes \dcover{w}$, $\rho(z) = 1\otimes (-1)$ and extended linearly, then
\begin{equation}\label{eq:DiagImage}
\rho(\mathbb{C} \dcover W) \cong \mathbb{C}\dcover{W}_-.
\end{equation}
This is an easy application of the Isomorphism Theorem; see \cite[Proposition~2.5]{CDMO22}.

\subsubsection{Dihedral groups}\label{ssec:Dihedral}

    From now on, we specialise our discussion to $\dim V=4$. The orthogonal group $\mathsf{O}(4)$ consists of all endomorphisms of $V_\bbr\cong\mathbb{R}^4$ that preserve the Euclidean norm. We consider a subgroup $W \cong \dih_{2m_1}\times \dih_{2m_2}\subset \Ortho(4)$, where $ \dih_{2m}$ denotes the dihedral group of order $2m$  with Coxeter presentation given by
 \begin{equation}
 \dih_{2m}= \left\langle s_1,\, s_m\ \middle|\  s_1^2 = s_m^2 =  (s_1s_m)^m =1\right\rangle.
 \end{equation}
Another presentation of the dihedral group $\dih_{2m}$ is in terms of the rotation $r := s_ms_1$ and the reflection (or flip) $f := s_m$
 \begin{equation}\label{e:dih2}
 \dih_{2m}= \left\langle r,\,f\ \middle|\  r^m = f^2 =  (rf)^2 =1\right\rangle\,.
 \end{equation}
The even elements of $\dih_{2m}$ form a cyclic group of $m$ elements, generated by $r$. The odd elements of $\dih_{2m}$ are reflections, given by $s_p = r^p f$ for $p=1,\dots, m$.

As a Weyl group, $W$ corresponds to the root system $R\subset V_{\mathbb{R}}$ of type $\mathsf{I}_2(m_1) \oplus \mathsf{I}_2(m_2)$.
In terms of an appropriate $(\cdot,\cdot)$-orthonormal basis of $V_{\mathbb{R}}$, which we will fix and denote by $\{x_1,\ldots, x_4\}$,
$R$ is realised as
\begin{equation}\label{eq:dihroots}
\begin{aligned}
\alpha_p &= (\sin(p\pi/m_1),\,-\cos(p\pi/m_1),\,0,\,0),&
\beta_q &= (0,\,0,\,\sin(q\pi/m_2),\,-\cos(q\pi/m_2)),
\end{aligned}
\end{equation}
for $p=1,\dots, 2m_1$ and $q=1,\dots, 2m_2$.
We fix the set of positive roots to be $R_+ = \left\{\alpha_1, \dots, \alpha_{m_1},\beta_1,\dots , \beta_{m_2} \right\}$. We decided to follow the same convention for the dihedral groups as Dunkl~\cite{Du89} and Humphreys~\cite{Hu90}. The associated reflections (see~\eqref{eq:refl}) ${s}_p$, $p=1,\dots, m_1$ and ${t}_q$, $q=1,\dots, m_2$, are given in matrix form by
\begin{equation}\label{eq:dihreflections}
\begin{aligned}
 {s}_p &=
 \begin{pmatrix}
 \cos(\tfrac{2p\pi}{m_1}) & \sin(\tfrac{2p\pi}{m_1}) & 0 & 0 \\
 \sin(\tfrac{2p\pi}{m_1}) & -\cos(\tfrac{2p\pi}{m_1}) & 0  & 0 \\
  0 & 0 & 1 & 0\\
  0 & 0 & 0 & 1\\
 \end{pmatrix}, &
 {t}_q &=
 \begin{pmatrix}
  1 & 0 & 0 & 0 \\
  0 & 1 & 0 & 0 \\
  0 & 0 & \cos(\tfrac{2q\pi}{m_2}) & \sin(\tfrac{2q\pi}{m_2})\\
  0 & 0 & \sin(\tfrac{2q\pi}{m_2}) & -\cos(\tfrac{2q\pi}{m_2})\\
 \end{pmatrix}.
\end{aligned}
\end{equation}

The structure of a dihedral group $\dih_{2m}$ depends on whether $m$ is even or odd. When it is odd, all reflections are in one conjugacy class; when $m$ is even, they fall into two conjugacy classes, and there is a non-trivial central element, which acts as minus the identity on the reflection representation; this is relevant for Theorem~\ref{thm:GradedCenterTAMA}.

\subsubsection{Double coverings}

The (positive) double covering of a dihedral group $G=\dih_{2m}$ of order $2m$ has the structure of a dihedral group of order $4m$: $\dcover{G} \cong \dih_{4m}$.
By~\cite[Thm. 3.4]{Mo80}, the double covering of the product of two Coxeter groups
is a quotient of the graded tensor product of the individual double coverings of the groups, where the  central extension elements are identified.

As an abstract group, $W=\dih_{2m_1}\times \dih_{2m_2}$ has the following Coxeter presentation
\begin{equation}
W = \left\langle s_1,\, s_{m_1},\, t_1,\, t_{m_2}\ \middle|\ \begin{matrix} s_1^2 = s_{m_1}^2 =  (s_1s_{m_1})^{m_1} =1,\\
 t_1^2 = t_{m_2}^2 = (t_1t_{m_2})^{m_2} =1,\end{matrix} \begin{matrix} (s_pt_q)^2=1,\\ (p=1,m_1;\;q=1,m_2)\end{matrix} \right\rangle.
\end{equation}
It follows from~\cite[Thm~3.6]{Mo76} that $\dcover{W}$ has the following presentation:
\begin{equation}\label{eq:dWgenrel}
\dcover{W} = \left\langle z, {\dcover s}_{1}, {\dcover s}_{m_1}, {\dcover t}_{1}, {\dcover t}_{m_2} \ \middle|\  \begin{matrix} {\dcover s}_{1}^{\ 2}= {\dcover s}_{m_1}^{\ 2}= 1,\   ({\dcover s}_{1}{\dcover s}_{m_1})^{m_1} =z^{m_1+1},\\
 {\dcover t}_{1}^{\ 2} = {\dcover t}_{m_2}^{\ 2} = 1,\ ({\dcover t}_{1}{\dcover t}_{m_2})^{m_2} =z^{m_2+1},\end{matrix}\ \begin{matrix}({\dcover s}_p{\dcover t}_q)^2=z,\,z^2=1,\\ (p=1,m_1;\;q=1,m_2)\end{matrix}\right\rangle,
\end{equation}
or (similar to~\eqref{e:dih2}), using ${\dcover r}_1 := z {\dcover s}_{m_1} {\dcover s}_1$, ${\dcover r}_2 := z  {\dcover t}_{m_2} {\dcover t}_1$, ${\dcover f}_1 :=  {\dcover s}_{m_1}$, ${\dcover f}_2 := {\dcover t}_{m_2}$, 
\begin{equation}\label{eq:dWgenrel2}
\dcover{W} = \left\langle z, {\dcover r}_{1}, {\dcover f}_{1}, {\dcover r}_{2}, {\dcover f}_{2} \ \middle|\
\begin{matrix} {\dcover r}_{j}^{\ m_j}= z,\ {\dcover f}_{j}^2= 1= ({\dcover r}_{j}{\dcover f}_{j})^{\ 2} ,\\
  {\dcover r}_{j}{\dcover r}_{k} = {\dcover r}_{k}{\dcover r}_{j},\
   {\dcover r}_{j}{\dcover f}_{k}= {\dcover f}_{k}{\dcover r}_{j},
  \end{matrix}\ \begin{matrix}({\dcover f}_1{\dcover f}_2)^2=z,\ z^2=1,\\ (j,k\in\{1,2\}, j\neq k)\end{matrix}\right\rangle.
\end{equation}

Let $\dcover{G}$ (resp.~$\dcover H$) denote the double cover of $\dih_{2m_1}$ (resp.~$\dih_{2m_2}$ ). In light of (\ref{eq:DoubleCoverDecomp}), we can decompose both groups into even and odd parts, so  
\begin{equation}\label{eq:DCDihDecomp}
    \dcover{G}_{\bar{0}} = \langle \dcover{r}_1 \rangle,\qquad \dcover{G}_{\bar{1}} = \dcover{G}_{\bar{0}}\dcover{f}_1 = \{(\dcover{r}_1)^j\dcover{f}_1\mid 0\leq j \leq 2m_1-1\}  = \dcover{f}_1\dcover{G}_{\bar{0}},
\end{equation}
and similarly for $\dcover{H}$. These splittings yield a decomposition of $\dcover{W}$ into four disjoint parts
\begin{equation}\label{eq:WZ2}
 \dcover{W} = \bigcup_{(\bar\imath,\bar\jmath\,)\in \mathbb{Z}_2^2}\dcover{W}_{(\bar\imath,\bar\jmath\,)}. \end{equation}
Here $\dcover{W}_{(\bar\imath,\bar\jmath\,)}$ is defined as the image of the multiplication map $\dcover{G}_{\bar\imath}\times \dcover{H}_{\bar\jmath}\to \dcover{W}$ defined by
\begin{equation}
(z_1^a\dcover{u},z_2^b\dcover{w}) \mapsto z^{a+b}\dcover{u}\dcover{w},
\end{equation}
for $z_1^a\dcover u\in \dcover G_{\bar\imath}$ and $z_2^b\dcover w\in \dcover H_{\bar\jmath}$, where we denote $z_1$ and $z_2$ to be the central extension elements in the double coverings of $G$ and $H$, respectively. It is straightforward to check that the maps $\dcover{G}_{\bar\imath}\times \dcover{H}_{\bar\jmath}\to \dcover{W}_{(\bar\imath,\bar\jmath\,)}$ are all $(2:1)$ and hence $|\dcover{W}_{(\bar\imath,\bar\jmath\,)}| = 2m_1m_2$.

As it will play a role in what follows, we mention in particular that the subgroup $ \dcover{W}_{(\bar 0,\bar 0)}$ is isomorphic to the quotient of the abelian group $C_{2m_1} \times C_{2m_2}$ by the group $C_2$ coming from the identification $\dr{1}^{\ m_1}=z=\dr{2}^{\ m_2}$:
\begin{equation}\label{eq:GddSubGrpdcW}
 \dcover{W}_{(\bar 0,\bar 0)} = \langle \dcover{r}_1,\dcover{r}_2\rangle \cong (C_{2m_1}\times C_{2m_2})/C_2.
\end{equation}

\subsubsection{Representation theory}
We now include a short review of the representation theory of the group $\dcover{W}$. Our focus here is on the spin representations, as these will be relevant for the representation theory of~$\tama$, and for that, the presentation~\eqref{eq:dWgenrel2} will be useful.
Note that the classification results below depend on the parity of the dihedral parameters $m_1$ and $m_2$.

A finite-dimensional $\dcover{W}$-module decomposes into a direct sum of one-dimensional simple modules for the abelian subgroup $\dcover W_{(\bar 0,\bar0)}$, see~\eqref{eq:GddSubGrpdcW}. The $\dcover W_{(\bar 0,\bar0)}$-modules are given by the $C_{2m_1}\times C_{2m_2}$-modules where $r_1^{m_1}$ and $r_2^{m_2}$ have the same action.
Specifically, these are given by $u(\ell_1,\ell_2)$, with $\ell_1\in \{0,\dots, 2m_1-1\}$, $\ell_2\in \{0,\dots, 2m_2-1\}$ and $\ell_1 \equiv \ell_2 \modSB 2$, endowed with the following action on $u\in u(\ell_1,\ell_2)$:
\begin{align}\label{e:actu}
	\dr{1}\cdot u &= \zeta_1^{\ell_1} u, & \dr{2}\cdot u&= \zeta_2^{\ell_2}u; & \zeta_1 &:= e^{i\pi/m_1}, & \zeta_2:=e^{i\pi /m_2}.
\end{align}
For this to be a spin representation, $z=\dr{1}^{\ m_1}=\dr{2}^{\ m_2}$ has to act by $-1$, which restricts the values of $\ell_1$ and $\ell_2$ to be odd integers. Next, we consider the induced $\dcover{W}$-representation
\begin{equation}\label{eq:Indlk}
    U(\ell_1,\ell_2) :=\mathrm{Ind}_{\dcover{W}_{(\bar 0, \bar 0)}}^{\dcover{W}}(u(\ell_1,\ell_2)).
\end{equation}
If $u\in u(\ell_1,\ell_2)$, then $ U(\ell_1,\ell_2) = \mathrm{span}\{ u,\df{1}u,\df{2}u,\df{1}\df{2}u\}$.
This leads to the classification of irreducible spin representations of $\dcover{W}$. The  complete details of the proof can be found in~\cite[Section 6.2.2]{LR23}.

\begin{theorem}\label{thm:spinirrepsW}
Let $\dcover{W}$ be the positive double covering of $W=\dih_{2m_1}\times \dih_{2m_2}$.
A list of all irreducible spin representations is given as follows.
\begin{itemize}
\item
For positive odd integers $\ell_1$ and $\ell_2$, with $\ell_1 < m_1$ and $\ell_2<m_2$, $U(\ell_1,\ell_2)$ is a four-dimensional irreducible spin representation with the actions of the elements given in matrix form by
	\begin{align*}
        \df{1} &= \left(\begin{smallmatrix}
                        0&1&0&0\\
                        1&0&0&0\\
                        0&0&0&1\\
                        0&0&1&0\\
                        \end{smallmatrix}\right), &
        \df{2} &= \left(\begin{smallmatrix}
                        0&0&1&0\\
                        0&0&0&-1\\
                        1&0&0&0\\
                        0&-1&0&0\\
                        \end{smallmatrix}\right) &
            \dr{1} &= \left(\begin{smallmatrix}
                        \zeta_1^{\ell_1}&0&0&0\\
                        0&\zeta_1^{-\ell_1}&0&0\\
                        0&0&\zeta_1^{\ell_1}&0\\
                        0&0&0&\zeta_1^{-\ell_1}\\
                        \end{smallmatrix}\right),
            & \dr{2} &= \left(\begin{smallmatrix}
                        \zeta_2^{\ell_2}&0&0&0\\
                        0&\zeta_2^{\ell_2}&0&0\\
                        0&0&\zeta_2^{-\ell_2}&0\\
                        0&0&0&\zeta_2^{-\ell_2}\\
                        \end{smallmatrix}\right).
	\end{align*}
	Note that $U(2m_1-\ell_1,\ell_2)$ or $U(\ell_1,2m_2-\ell_2)$ would result in a representation equivalent to $U(\ell_1,\ell_2)$.
This gives $\lfloor m_1/2 \rfloor \times \lfloor m_2/2\rfloor$ non-isomorphic irreducible spin representations of dimension 4.

\item If at least one of the dihedral parameters $m_1$ and $m_2$ is odd, there are also irreducible spin representations of dimension 2.
\begin{enumerate}
	\item If $m_1$ is odd, there are $m_2 $ representations $U(m_1,\ell_2)$ indexed by a positive odd integer $\ell_2$ with $\ell_2<2m_2$.
	The actions of the elements are given in matrix form by
	\begin{align*}
        \df{1} &= \left(\begin{smallmatrix}
                       	1&0\\0&-1
                       \end{smallmatrix}\right),&
        \df{2} &= \left(\begin{smallmatrix}
                       	0&1\\1&0
                       \end{smallmatrix}\right),&
        \dr{1}  &= \left(\begin{smallmatrix}
                       	-1&0\\0&-1
                       \end{smallmatrix}\right), &
        \dr{2}&= \left(\begin{smallmatrix}
                 	\zeta_2^{\ell_2}&0\\0&\zeta_2^{-\ell_2}
                 \end{smallmatrix}\right).
	\end{align*}
	    \item If $m_2$ is odd, there are $m_1 $ representations $U(\ell_1,m_2)$ indexed by a positive odd integer $\ell_1$ with $\ell_1 < 2m_1$.
	    The actions of the elements are given in matrix form by
	\begin{align*}
        \df{1} &=  \left(\begin{smallmatrix}
                       	0&1\\1&0
                       \end{smallmatrix}\right),&
        \df{2} &=\left(\begin{smallmatrix}
                       	1&0\\0&-1
                       \end{smallmatrix}\right),&
        \dr{1}  &= \left(\begin{smallmatrix}
                 	\zeta_1^{\ell_1}&0\\0&\zeta_1^{-\ell_1}
                 \end{smallmatrix}\right), &
        \dr{2}&= \left(\begin{smallmatrix}
                       	-1&0\\0&-1
                       \end{smallmatrix}\right).
	\end{align*}
 \end{enumerate}
If both $m_1$ and $m_2$ are odd, the representation $U(m_1,m_2)$ is in both of the two previous cases.
\end{itemize}
   \end{theorem}

\subsubsection{Clifford algebra spinors} In our setting, as $\dim V_\bbr =4$ is even, there is a unique, up to equivalence, simple module for $\Clif$, denoted $\bbs$. Upon restriction, $\bbs$ will determine an irreducible spin $\dcover{W}$-representation. We shall end this section by determining the $\dcover{W}$-action on $\bbs$. For $a\in\{1,2\}$,  write
\begin{equation}\label{eq:thetas}
\theta_a^+ := \theta_a = \tfrac{1}{2}(e_{2a-1} + ie_{2a}),\qquad \theta_a^- := \bar{\theta}_a =\tfrac{1}{2}(e_{2a-1} - ie_{2a}).
\end{equation}
The elements $\{\theta_1,\bar\theta_1,\theta_2,\bar\theta_2\}$ also generate $\Clif$ and they satisfy the Grassmann relations
\begin{equation}\label{eq:GrassmRels}
\theta_a\theta_b + \theta_b\theta_a = 0 =
\bar\theta_a\bar\theta_b + \bar\theta_b\bar\theta_a,\qquad
\theta_a\bar\theta_b + \bar\theta_b\theta_a = \delta_{ab},
\end{equation}
for $a,b\in\{1, 2\}$. The spin module $\bbs$ can be realised as the four-dimensional vector space $\bigwedge(\bar\theta_1,\bar\theta_2)$. In terms of the generators (\ref{eq:thetas}), $\bar\theta_a$ acts as a multiplication operator, while $\theta_a$ acts as an odd derivation with
$\theta_a(\theta_b)=0$ and $\theta_a(\bar\theta_b) = \delta_{ab}$.

\begin{proposition}
    In terms of Theorem \ref{thm:spinirrepsW}, as an irreducible spin $\dcover{W}$-representation, we have $\bbs\cong U(1,1)$.
\end{proposition}

\begin{proof}
    We have that $\{1,\bar\theta_1,\bar\theta_2,\bar\theta_1\wedge\bar\theta_2\}$ is a basis of $\bbs$. Furthermore, the generators of $\dcover{W}$ are realised as $\gamma(\alpha_p)$ and $\gamma(\beta_q)$. One checks that $\gamma(\alpha_p)=i(\zeta_1^{-p}\theta_1-\zeta_1^p\bar\theta_1)$, while $\gamma(\beta_q)=i(\zeta_2^{-q}\theta_2-\zeta_2^q\bar\theta_2)$. It is then straightforward to obtain that $\gamma(\alpha_p)$, with $p=1,\ldots, m_1$, acts via the assignments
    \[
1\mapsto -i\zeta_1^{p}\bar\theta_1,\quad
\bar\theta_1\mapsto i\zeta_1^{-p}1,\quad
\bar\theta_2\mapsto -i\zeta_1^{p}\bar\theta_1\wedge\bar\theta_2,\quad
\bar\theta_1\wedge\bar\theta_2\mapsto i\zeta_1^{-p}\bar\theta_2,
    \]
    while $\gamma(\beta_q)$, with $q=1,\ldots, m_2$, acts via the assignments
    \[
1\mapsto -i\zeta_2^{q}\bar\theta_2,\quad
\bar\theta_1\mapsto i\zeta_2^{q}\bar\theta_1\wedge\bar\theta_2,\quad
\bar\theta_2\mapsto i\zeta_2^{-q}1,\quad
\bar\theta_1\wedge\bar\theta_2\mapsto -i\zeta_2^{-q}\bar\theta_1.
    \]
It follows that $\bbs$ is indeed irreducible and the $(\dr1,\dr2)$-eigenvalues of the elements $1$, $ \bar\theta_1$, $\bar\theta_2$ and $\bar\theta_1\wedge\bar\theta_2$ are given by $(\zeta_1, \zeta_2)$, $(\zeta_1^{-1},\zeta_2)$, $(\zeta_1,\zeta_2^{-1})$ and $(\zeta_1^{-1},\zeta_2^{-1})$,  respectively, thus finishing the proof.
\end{proof}

\section{Total angular momentum algebra}\label{sec:tama}

\subsection{Rational Cherednik algebra}
Recall the pair $(V,W)$ defined above in \cref{sec:IniDefs}. Fix $\kappa: R \to \bbc$ a $W$-invariant parameter function, that is, an assignment $\alpha\mapsto \kappa_\alpha\in \bbc$ such that $\kappa_{w(\alpha)} = \kappa_\alpha$, for all $\alpha \in R$ and $w\in W$.
For  $W= \dih_{2m_1}\times \dih_{2m_2}$, there can be at most 4 distinct values of $\kappa$, in the case when $m_1$ and $m_2$ are both even. We will denote these as $\{\kappa_j\}_{j=1}^4$, for which we have $\kappa_{\alpha_{2p-1}} = \kappa_1$, $\kappa_{\alpha_{2p}} = \kappa_2$, $\kappa_{\beta_{2q -1}} = \kappa_3$, $\kappa_{\beta_{2q}} = \kappa_4$, for the appropriate integer values of $p,q$. When $m_1$ (resp.~$m_2$) is odd,  we denote $\kappa_{\alpha_j} = \kappa_1$ (resp.~$\kappa_{\beta_j} = \kappa_3$) for all values of $j$.
For any $\alpha\in R_+$, the reflection ${s}_\alpha$ is given by
\begin{equation}\label{eq:refl}
    {s}_\alpha(x) = x - 2\frac{(x,\alpha)}{(\alpha,\alpha)}\alpha,
\end{equation}
for all $x\in V$. With $\alpha$ given as in (\ref{eq:dihroots}), the associated reflection ${s}_\alpha$ has matrix form as in (\ref{eq:dihreflections}). Note that, for $p=1,\ldots,m_1$ (resp. $q=1,\ldots,m_2$), ${s}_\alpha = {s}_p$ if $\alpha\in \{\alpha_p,\alpha_{p+m_1}\}$ (resp. ${s}_\beta = {t}_q$ if $\beta \in \{\beta_q,\beta_{q+m_2}\})$.

\begin{definition}[\cite{EG02}]
The rational Cherednik algebra $H_\kappa = H_\kappa(V,W)$ is the quotient of the smash-product algebra $T(V\oplus V^*)\rtimes \bbc W$ by the relations $[\xi,\eta]=0=[x,y]$, for all $\xi,\eta\in V^*,x,y\in V$ and
\begin{equation}\label{eq:rcarel}
[\xi,x] = \xi(x) + \psi_\kappa(B^{-1}(\xi),x),
\end{equation}
for all $\xi \in V^*$ and $x\in V$, where $\psi_\kappa(x,y) = \sum_{\alpha\in R_+} \frac{2\kappa_\alpha}{(\alpha,\alpha)}(x,\alpha)(y,\alpha){s}_\alpha$ and $B:V\to V^*$ is the linear isomorphism induced by the Euclidean structure.
\end{definition}
\begin{proposition}
In $H_\kappa$, we have $[B(x),y] = [B(y),x]$, for all $x,y\in V$.
\end{proposition}
\begin{proof}
Follows from the fact that $\psi_\kappa$ is symmetric and $B(x)(y) = (x,y) = B(y)(x)$.
\end{proof}

As a vector space, we have $H_\kappa = S(V)\otimes \bbc W \otimes S(V^*)$ where $S(V)$ is the symmetric algebra on $V$.
Given $\tau:W \to GL(V_\tau)$, an irreducible representation of $W$ (with representation space $V_\tau$), we denote by $M_\kappa(\tau) = H_\kappa \otimes_{\bbc W \otimes S(V^*)} V(\tau) = S(V) \otimes V_\tau$ the standard module of $H_\kappa$ where $V^*$ acts by zero on $V_\tau$.
When $\tau$ is the trivial representation, $M_\kappa(\tau)$ corresponds to the space of polynomial functions on $V^*$.
In this case, an explicit (faithful~\cite{EG02}) realisation of $H_\kappa$ is given by means of Dunkl operators~\cite{Du89}.

Let $\textup{Irr}(W)$ denote the set of (equivalence classes of) irreducible representations of $W$,
it is well known that the set $\textup{Irr}(W)$ parametrises the set of (equivalence classes of) simple modules of $H_\kappa$ and that there is a surjection $M_\kappa(\tau)\to L_\kappa(\tau)$ from the standard module
to the corresponding simple module $L_\kappa(\tau)$, for each $\tau \in \textup{Irr}(W)$.

\subsection{Dual pair}
We will now work in the tensor product $H_\kappa\otimes \ca C$.
The $\mathbb{Z}_2$-grading of the Clifford algebra $ \ca C$ naturally induces a $\mathbb{Z}_2$-grading on $H_\kappa\otimes \ca C$ by declaring
$(H_\kappa\otimes \ca C)_{\bar{\jmath}} = H_\kappa\otimes \ca C_{\bar{\jmath}},$
for $j\in \{0,1\}$.
To ease notation, we will identify $H_\kappa \cong H_\kappa\otimes \bbc\subset H_\kappa\otimes \ca C $
and $\ca C \cong \bbc\otimes \ca C\subset H_\kappa\otimes \ca C $.

Since $\bbc W\hookrightarrow H_\kappa$, we can define $\rho:\bbc\dcover W\to H_\kappa\otimes\ca C$ as in~(\ref{eq:DiagImage}).
The restriction of $\rho$ to $\dcover W$ is a group isomorphism, and we will sometimes identify $\rho(\dcover W)$ with $\dcover W$ and omit the notation $\rho$ when it is clear from the context that one is working in $H_\kappa\otimes\ca C$.

\begin{notation}\label{not:refl}
In the Clifford algebra realisation, $\dcover{W}$ is generated by~$\gamma(R)$, the image of the (normalised) root system~\eqref{eq:dihroots}. For a root $\alpha\in R$, we will denote $\dcover{s}_\alpha  := \rho(\gamma(\alpha)) =  s_\alpha\otimes \gamma(\alpha)\in \rho(\mathbb{C}\dcover W)$.  Similarly to the notations of~(\ref{eq:dihroots}), we shall also write $\dsig{p} :=  \dcover{s}_{\alpha_p}$ for $1\leq p \leq m_1$ and $\dtau{q} :=  \dcover{s}_{\beta_q}$, for $1\leq q \leq m_2$.
\end{notation}

Given the standard orthonormal basis $\{x_1,\ldots, x_4\}$ of $V_{\mathbb{R}}$, we let $\{\xi_1,\ldots,\xi_4\}$ be the dual, orthonormal basis of $V_{\mathbb{R}}^*$ defined by $\xi_j = B(x_j)$, for all $j =1,\ldots, 4$. With respect to these bases, consider the elements
\begin{equation}\label{eq:EvenPart}
\Delta_{\kappa} = \sum_{j=1}^4 \xi_j^2\otimes 1,\qquad E= \frac{1}{2}\sum_{j=1}^4 (x_j\xi_j + \xi_j x_j)\otimes 1,\qquad |x|^2 = \sum_{j=1}^4 x_j^2\otimes 1
\end{equation}
and
\begin{equation}\label{eq:OddPart}
\up = \sum_{j=1}^4 \xi_j\otimes e_j,\qquad \ux = \sum_{j=1}^4 x_j\otimes e_j
\end{equation}
of $H_\kappa\otimes \ca C$. Let $\fg_{\bar 0} = \textup{span}\{\Delta_{\kappa},E,|x|^2\}$ and $\fg_{\bar 1} = \textup{span}\{\up,\ux\}$. The following is well  known (see~\cite{OSS09}):

\begin{theorem}\label{thm:osp}
The vector subspace $\fg = \fg_{\bar 0} \oplus \fg_{\bar 1}$ of $H_\kappa\otimes \ca C$ spanned by $\{\Delta_{\kappa}, \up, E, \ux, |x|^2\}$ has the structure of a Lie superalgebra isomorphic to $\mathfrak{osp}(1|2,\bbc)$.
\end{theorem}

By definition of the elements of $\fg$, we have $\llbracket \fg, \rho(\mathbb{C}\dcover W)\rrbracket = 0$, using the graded commutator.

\begin{definition}
The Dunkl total angular momentum algebra (TAMA) $\tama = \tama(V,W)$ is defined as the graded centraliser in $H_\kappa\otimes \ca C$ of $\fg\cong\mathfrak{osp}(1|2,\bbc)$.
\end{definition}

\subsection{Structural properties of the TAMA}
We shall now recall some structural properties of the algebra $\tama\subseteq H_\kappa\otimes \ca C$ (see \cite{Os22}), specialised to our context of $(V,W)$. To that end, let $P\in \End(H_\kappa\otimes\ca C)$ be the element defined by
\begin{equation}\label{eq:Proj}
P = \id - \tfrac{1}{2}\ad(\up)\ad(\ux) = \id + \tfrac{1}{2}\ad(\ux)\ad(\up).
\end{equation}
Here, the adjoint action $\ad \colon H_\kappa\otimes\ca C \to \End(H_\kappa\otimes\ca C) $  is given by $ \ad(x)(y) = \llbracket x,y \rrbracket$ where we use the graded commutator (see~\eqref{eq:Z2_products}). The relevance of $P\in \End(H_\kappa\otimes\ca C)$ is manifest by the following result:
\begin{theorem}[\cite{Os22}]\label{thm:Projector}
When restricted to $\ca A = \textup{Cent}_{H_\kappa\otimes\ca C}(\fg_{\bar 0})$, we have $(P|_{\ca A})^2 = P|_{\ca A}$ and $\tama = P(\ca A)$.
\end{theorem}
Note that the Clifford algebra $\ca C \cong \bbc\otimes \ca C\subset H_\kappa\otimes \ca C $ trivially commutes with $\fg_{\bar 0}\subset H_\kappa\otimes \bbc\subset H_\kappa\otimes \ca C$, so $\ca C \subset \ca A$.
Recalling the linear isomorphism $\gamma:\bigwedge(V)\to \ca C$ of (\ref{eq:quantization}), there is a $W$-equivariant linear map $O:\bigwedge(V)\to \tama$, defined for $v\in \bigwedge(V)$ as
\begin{equation}\label{eq:SymmOps}
    O(v) = -\tfrac{1}{2}P(
    \gamma(v)).
\end{equation}
Because of the $W$-equivariance, we have the following interaction with $\rho(\mathbb{C}\dcover W)$:
\begin{equation}\label{eq:WComm}
     \rho(\dcover w)O(v) = (-1)^{k|\dcover w|}O(\pi(\dcover w)(v))\rho(\dcover w),
\end{equation}
for $v\in \bigwedge^k(V)$ and $\dcover w \in \dcover W$, where~\eqref{eq:DoubleCoverDecomp} gives the
$\mathbb{Z}_2$-grading of $\dcover W$. Furthermore, the following result was proved for a general reflection group $W$.
\begin{theorem}[\cite{Os22}]\label{thm:TAMAgen}
As an associative subalgebra of $H_\kappa\otimes\ca C$, the centraliser algebra $\tama$ is generated by
\[
\textstyle \{O(v)\;\mid\; v \in \bigwedge^2(V)  \textup{ or }v \in \bigwedge^3(V) \} \cup \{\rho(\dcover w)\mid \dcover w\in \dcover W\}.
\]
\end{theorem}

For completeness, we spell out the elements of the form~\eqref{eq:SymmOps}.
For $x\in V$, we have
    \begin{equation}\label{eq:Ox}
    O(x) = \sum_{\alpha \in R_+} \kappa_\alpha(x,\alpha)\dsig{\alpha} = \left(\sum_{p=1}^{m_1} \kappa_{\alpha_p}(x,\alpha_p)\dsig{p} +\sum_{q=1}^{m_2} \kappa_{\beta_q}(x,\beta_q)\dtau{q} \right)\in \rho(\mathbb{C}\dcover W).
  \end{equation}
Given vectors $x,y,z\in V$, we have:
\begin{equation}\label{eq:2indFormula}
O(x\wedge y) = (xB(y) - yB(x)) + \tfrac{1}{2}\gamma(x\wedge y)+ (O(x)\gamma(y) - O(y)\gamma(x)) ,
\end{equation}
which reduces to a total angular momentum operator when $\kappa=0$, and
\begin{equation}
\begin{aligned}\label{eq:3indFormula}
O(x\wedge y \wedge z) &= -\tfrac{1}{2}\gamma(x\wedge y\wedge z) +O(x\wedge y)\gamma(z) -  O(x\wedge z)\gamma(y) + O(y\wedge z)\gamma(x)\\
&\quad-O(x)\gamma(y\wedge z) + O(y)\gamma(x\wedge z)  - O(z)\gamma(x\wedge y) .
\end{aligned}
\end{equation}
We will describe the image of an element of the one-dimensional space $\bigwedge^4(V)$ in~\eqref{eq:Z-symmetry} below.
We note that there is a similar expression as~\eqref{eq:3indFormula}, but it
can also be expressed in terms of other elements following Theorem~\ref{thm:TAMAgen}, see~\eqref{eq:Z-symmetry2}.

\begin{notation}
For elements of the linear basis $\{x_A \mid A\subseteq \{1,\ldots,4\}\}$ of $\bigwedge(V)$, we shall use the notations $O_A = O(x_A)$ or even $O_{a_1a_2\cdots a_p}$, if $A=\{a_1,a_2,\ldots, a_p\}\subseteq \{1,\ldots,4\}$.
\end{notation}

We extract now some of the commutation relations from~\cite{Os22}.

\begin{theorem}\label{thm:TAMArel}
Let $a,b,c,p,q\in\{1,2,3,4\}$. The generators satisfy the following commutation relations:
\begin{equation}\label{eq:2ind2indComm}
    \llbracket O_{ab},O_{pq}\rrbracket = \delta_{bp}A_{aq} - \delta_{bq}A_{ap} - \delta_{ap}A_{bq} + \delta_{aq}A_{bp} + \{ O_a,O_{bpq} \} - \{ O_b,O_{apq} \},
\end{equation}
where $A_{ab} = O_{ab} + O_aO_b-O_bO_a$ is an anti-symmetric expression on the indices. Moreover,
\begin{subequations}
\begin{alignat}{2}\label{eq:2ind3indComm1}
    \llbracket O_{ab},O_{apq}\rrbracket  &= -(O_{bpq} + \{O_b,O_{pq}\} + [O_a,O_{abpq}]),\\\label{eq:2ind3indComm2}
    \llbracket O_{ab},O_{abq}\rrbracket  &= -(\{O_a,O_{aq} \} + \{O_b,O_{bq}\}),
\end{alignat}
\end{subequations}
and also
\begin{subequations}
\begin{alignat}{2}\label{eq:3ind3indComm1}
    \llbracket O_{abp},O_{abq}\rrbracket &= \{O_p,O_q\} + \{O_{ap}, O_{aq}\} + \{O_{bp}, O_{bq}\},\\\label{eq:3ind3indComm2}
    \llbracket O_{abc},O_{abc}\rrbracket &= 2(O_a^2 + O_b^2 + O_c^2 + O_{ab}^2 + O_{ac}^2 + O_{bc}^2 - \tfrac{1}{4}).
\end{alignat}
\end{subequations}
\end{theorem}

\begin{remark}
The relations of Theorem~\ref{thm:TAMArel} can also be formulated for arbitrary vectors, which can be found in~\cite{Os22} together with, for instance, expressions for the product of two 3-index symmetries.
\end{remark}

\begin{remark}
The results described in Theorem~\ref{thm:TAMAgen} together with the relations in Theorem~\ref{thm:TAMArel} and \eqref{eq:WComm} do not yield a full presentation of $\tama$. Only when $\dim V = 3$ will the relations above describe all the relations that the generators satisfy.
\end{remark}

It will also be useful to consider an isotropic (with respect to~$(\cdot,\cdot)$) basis for the complexified space $V$ besides the orthonormal basis $\{x_1,x_2,x_3,x_4\}$ of $V_{\mathbb{R}}$.
We decompose $V = V_1\oplus V_2$, with $V_a$ the span of $\{x_{2a-1},x_{2a}\}$ for $a \in \{1,2\}$.
Recall that the space $V_1$ (resp.~$V_2$) is fixed under the action of the dihedral group $\dih_{2m_2}$ (resp.~$\dih_{2m_1}$).
For each 2-dimensional space $V_a$, we let
\begin{equation}\label{eq:z}
z_a^+ := z_a = x_{2a-1} + ix_{2a},\qquad z_a^- := \bar{z}_a =x_{2a-1} - ix_{2a}.
\end{equation}
We note that $(z_a^+, z_b^-) = 2 \delta_{a,b}$ while $(z_a^\pm, z_b^\pm) = 0$ for $a,b \in \{1,2\}$.

While the action of $W$ on the orthonormal basis is given by~\eqref{eq:dihreflections}, for the isotropic basis we have
\begin{equation}
\label{eq:actionWisotropic}
   s_p (z^\pm_1) = \zeta_1^{\pm2p} z^\mp_1, \qquad
   s_p (z^\pm_2) =  z^\pm_2, \qquad
   t_q (z^\pm_{2}) = \zeta_2^{\pm2q} z^\mp_{2}, \qquad
   t_q (z^\pm_1) =  z^\pm_{1},
\end{equation}
for $p=1,\dotsc,m_1$ and $q=1,\dotsc,m_2$, with $\zeta_1:=e^{i\pi/m_1}$ and  $\zeta_2:=e^{i\pi/m_2}$.

 \begin{notation}\label{not:isotropic}
 It will be convenient to succinctly express the symmetry operators with respect to the isotropic basis $\{z_1^+,z_1^-,z_2^+,z_2^-\}$. Accordingly, with $\delta,\varepsilon,\nu\in\{+,-\}$ and $a,b,c\in\{1,2\}$, we shall write,
    \begin{equation}\label{eq:notationisotropic}
    O_{\ib{\delta}{a}} := O(z_a^\delta),\quad O_{\ib{\delta\varepsilon}{ab}}:= O(z_a^\delta\wedge z_b^\varepsilon) \quad\text{and}\quad O_{\ib{\delta\varepsilon\nu}{abc}} := O(z_a^\delta\wedge z_b^\varepsilon\wedge z_c^\nu).
    \end{equation}
 In formulas, the signs  $\{+,-\}$ will be interpreted as $\{+1,-1\}$.  
 \end{notation}

Denote by $\Omega$ the Casimir and by $\mathcal{S}$ the Scasimir~\cite[Example 2, p.9]{FSS00} of the $\mathfrak{osp}(1|2,\bbc)$ realisation of Theorem~\ref{thm:osp}.
The Casimir is an element of $\tama$, but the Scasimir is not.
However, $\mathcal{S}$ is closely related to the symmetry corresponding to the volume form  $x_1\wedge x_2\wedge x_3\wedge x_4 \in \bigwedge^4(V)$.
Let us denote
\begin{equation}\label{eq:Z-symmetry}
Z := O_{1234} =  -O(z_1\wedge\bar{z}_1\wedge z_2\wedge\bar{z}_2)/4.
\end{equation}
 One has  $Z = \mathcal{S}e_{1234} = e_{1234}\mathcal{S}$. From the properties of the Clifford algebra pseudo-scalar $e_{1234}$, it follows that, for $v\in \bigwedge^k(V)$, we have
\begin{equation}\label{eq:Zcommutation}
ZO(v) = (-1)^{k} O(v) Z\p.
\end{equation}
Furthermore, $Z$ admits non-trivial relations with the other symmetries of $\tama$, see~\cite{Os22}.
\begin{proposition}\label{prop:Z}
    The element $Z$ can be expressed in terms of the other symmetries as
\begin{equation}\label{eq:Z-symmetry2}
Z = \{O_{12},O_{34}\}
- \{O_{13},O_{24}\}
+ \{O_{14},O_{23}\}
-2  (O_{123}O_4 + O_{124}O_3 + O_{134}O_2+ O_{234}O_1 ).
\end{equation}
or, with respect to the isotropic basis, as
\begin{equation}\label{eq:Z-symmetry3}
-4Z  = \{O_{\ib{+-}{11}},O_{\ib{+-}{22}}\}
- \{O_{\ib{++}{12}},O_{\ib{--}{12}}\}
+ \{O_{\ib{+-}{12}},O_{\ib{-+}{12}}\}
+ [O_{\ib{-}{2}}, O_{\ib{+-+}{112}}]
 -[O_{\ib{+}{2}} , O_{\ib{+--}{112}}]
+ [O_{\ib{-}{1}},O_{\ib{++-}{122}}]
 - [O_{\ib{+}{1}},O_{\ib{-+-}{122}}]
 \rlap{\,.}
\end{equation}
The square $Z^2$ has the following expression
\begin{equation}\label{eq:CasimirZ}
    Z^2 = \Omega + \tfrac{1}{4} = \tfrac{3}{4} - 2\sum_{j=1}^4O_j^2 -\!\!\sum_{1\leq j<k\leq 4}O_{jk}^2,
\end{equation}
or, with respect to the isotropic basis,
\begin{equation}\label{eq:CasimirZ2}
    Z^2 =  \tfrac{3}{4} - \{O_1^+,O_1^-\}  - \{O_2^+,O_2^-\}  +\tfrac14 (O_{\ib{+-}{11}})^2 +\tfrac14 (O_{\ib{+-}{22}})^2  -\tfrac14\{O_{\ib{+-}{12}},O_{\ib{-+}{12}}\} -\tfrac14 \{O_{\ib{++}{12}},O_{\ib{--}{12}}\}.
\end{equation}
\end{proposition}

The (graded) centre of $\tama$ was computed in \cite{CDMO22} for a general reflection group $W$.
It is always a univariate polynomial ring, but its generator is different depending on whether or not the longest element of $W$, which we denote by $w_0$, is a scalar multiple of the identity $(-1)_V$ in $\mathsf{GL}(V)$.
Note that for $W = \dih_{2m_1}\times\dih_{2m_2}$, this is the case if and only if both parameters of the dihedral groups are even.

\begin{theorem}[\cite{CDMO22}]\label{thm:GradedCenterTAMA}
The graded centre of $\tama$ is the polynomial ring $\mathbb{C}[\mathsf{S}]$ with generator
\[
\mathsf{S} = \left\{
\begin{array}{rl}
     \mathcal{S}w_0,& \textup{ if } w_0=(-1)_V,  \\
    \Omega, & \textup{ if } w_0\neq(-1)_V.
\end{array}
\right.
\]
\end{theorem}

\subsection{Star structures}\label{sec:unitarity}
Both $H_\kappa$ and $\ca C$ are equipped with natural $\ast$-structures whose product defines such a structure on $H_\kappa\otimes \ca C$. We shall denote all these structures simply by $\ast$.

For $H_\kappa$, we use the identification $B:V\to V^*$ induced by the Euclidean structure to define an involutive endomorphism of the space of generators $(V\oplus V^*)\otimes \mathbb{C} W$ via
\begin{equation}\label{eq:starRCA}
x^*=B(x), \qquad\xi^*=B^{-1}(\xi),\qquad w^* = w^{-1},
\end{equation}
for $x\in V_\mathbb{R}$, $\xi \in V^*_\mathbb{R}$ and $w\in W$. This is then extended to a conjugate-linear anti-involution on $H_\kappa$. This extension is only well-defined when the parameter function $\kappa$ is real-valued. Recall that for parameters $\kappa$ sufficiently close to $0$ (in the sense of~\cite{ES09}), the standard modules $M_\kappa(\tau)$ are unitarisable (with respect to the star structure just mentioned) and simple.

For $\ca C$, the transpose anti-involution on the tensor algebra $v_{i_1}\otimes \cdots \otimes v_{i_p} \mapsto
v_{i_p}\otimes \cdots \otimes v_{i_1}$
descends to a conjugate-linear anti-involution on the Clifford algebra (and likewise on the exterior algebra $\bigwedge V$) defined on the generators by $e_i^*=e_i$ and extended to a conjugate-linear anti-involution on $\ca C$. The map $O:\bigwedge V\to \tama$ of (\ref{eq:SymmOps}) behaves well with respect to the star structures.

\begin{lemma}\label{lem:starO}
    For any $v \in \bigwedge V$, we have $O(v)^* = O(v^*)$.
\end{lemma}
\begin{proof}
    Note that the $\ast$-structure on $H_\kappa\otimes\ca C$ described in this section interchanges
    $\up$ and $\underline{x}$. Hence, the map $P$ of (\ref{eq:Proj}) satisfies $P(X^\ast) = P(X)^*$ for any $X$ in the domain of $P$. This implies the result.
\end{proof}
\begin{proposition}
    When the parameter function $\kappa$ is real-valued, both $\mathfrak{g}\cong \mathfrak{osp}(1|2)$ and its centraliser algebra $\mathfrak{O}_k$ are $\ast$-invariant subalgebras of $H_\kappa\otimes \ca C$.
\end{proposition}

\begin{proof}
    With the assumptions on $\kappa$, in terms of (\ref{eq:EvenPart}) and (\ref{eq:OddPart}), the $\ast$-structure interchanges $\Delta_{\kappa}\leftrightarrow |x|^2$ and $\up\leftrightarrow \underline{x}$ while $E^*=E$. As for $\mathfrak{O}_k$, using Lemma \ref{lem:starO} the star structure preserves the generators.
\end{proof}

For parameter functions $\kappa$ such that $M_\kappa(\tau)$ is irreducible and unitarisable, the module $M_\kappa(\tau)\otimes \mathbb{S}$ (with $\mathbb{S}$ the simple spin module of $\ca C$) is a unitarisable module for $H_\kappa\otimes \ca C$ as well as for $\mathfrak{g}$ and for $\tama$.

\section{Ladder operators and the triangular subalgebra}\label{sec:ladder}
 We remind the reader that $[\cdot,\cdot]$ and $\{\cdot,\cdot\}$ denote the usual ungraded commutator and anti-commuta\-tor product, while $\llbracket \cdot,\cdot \rrbracket$ is the $\bbz_2$-graded commutator.

\subsection{Definitions and notations}

We start by describing an abelian Lie subalgebra of $\tama$, which is specific to the geometrical constraints imposed by the  double dihedral reflection group.

\begin{definition}\label{def:Tzero}
For $a\in \{1,2\}$, define
\begin{equation}\label{eq:Hs}
H_a := \tfrac{1}{2}O(z_a\wedge\bar{z}_a)  = -iO(x_{2a-1}\wedge x_{2a})
\end{equation}
and let $\mathfrak{h} = \textup{span}\{H_1,H_2\}$. Furthermore, define $\mathfrak{a} = \textup{span}\{H_1,H_2,Z\}$ and $\mathfrak{t}_0 = \mathfrak{a}\oplus \textup{span}\{\dcover r_1,\dcover r_2\}$.
\end{definition}

\begin{proposition}\label{prop:HsFirst}
The subspace $\mathfrak{t}_0$ is an abelian Lie subalgebra of $\tama$.
\end{proposition}
\begin{proof}
Using (\ref{eq:2ind2indComm}) we get $\llbracket H_1,H_2 \rrbracket = -\llbracket O_{12},O_{34} \rrbracket= - \{ O_1,O_{234} \} + \{ O_2,O_{134} \} = 0,$
since $\rho(\dcover \sigma)O(x\wedge x_3\wedge x_4) = -O(x\wedge x_3\wedge x_4) \rho(\dcover \sigma)$, whenever $x\in V_1$ and $\pi(\widetilde{\sigma})\in D_{2m_2}$, in view of relations~(\ref{eq:WComm}) and~\eqref{eq:Ox}.
That $Z$ commutes with $\mathfrak{h}$ follows from~(\ref{eq:Zcommutation}). Finally, it is straightforward to check that $\operatorname{span}\{\dr{1},\dr{2}\}$ commutes with $\mathfrak{a}$ (see Lemma~\ref{lem:actionWonT}, below).
\end{proof}

Let $\varpi_1,\varpi_2\in \mathfrak{h}^*$ denote the dual functionals defined by $\varpi_a(H_b) = \delta_{ab}$. Consider the subset
$\Phi = \Phi_{\bar 0} \cup \Phi_{\bar 1}$ of $\mathfrak{h}^*$ defined by
\begin{equation}\label{eq:roots}
\Phi_{\bar 0} = \{\pm(\varpi_1 - \varpi_2),\pm(\varpi_1 + \varpi_2)\}, \qquad
\Phi_{\bar 1} = \{\pm\varpi_1 ,\pm\varpi_2\}.
\end{equation}
Whenever $\alpha\in\Phi_{\bar 0}$, we write $\alpha = \delta\varpi_1 + \varepsilon\varpi_2$, with $\delta,\varepsilon\in\{+,-\}$. Similarly, we write $\beta = \varepsilon\varpi_a\in\Phi_{\bar 1}$.
It will be useful below to partition $\Phi$ in positive and negative roots as follows: $\Phi= \Phi_+ \cup \Phi_-$ with
\begin{equation}\label{eq:rootsPosNeg}
\Phi_+:= \{(\varpi_1 - \varpi_2),(\varpi_1 + \varpi_2),\varpi_1,\varpi_2\}, \quad \Phi_- := -\Phi_+.
\end{equation}

\begin{definition}
Define the ladder elements $\{L_\alpha\mid\alpha\in\Phi\}$ via the following: if $\alpha = \delta\varpi_1 + \varepsilon\varpi_2\in\Phi_{\bar 0}$, with $\delta,\varepsilon\in \{+,-\}$ then
\begin{equation}\label{eq:EvLadder}
L_\alpha := \{H_1,\{H_2,O(z_1^\delta\wedge z_2^\varepsilon)\}\}
\end{equation}
and if $\beta = \varepsilon\varpi_b \in \Phi_{\bar 1}$ with $\varepsilon\in\{+,-\}$ and $b\in \{1,2\}$, then
\begin{equation}\label{eq:OddLadder}
L_\beta := \{H_b,O(z_b^{\varepsilon}\wedge z_a \wedge \bar{z}_a)\}
\end{equation}
where $a\in \{1,2\}\setminus\{b\}$.
For computations, it will be convenient to also denote the ladder elements as
\begin{equation}\label{eq:LadderOB}
L_{\ib{\delta\varepsilon}{12}} := L_{\delta\varpi_1 + \varepsilon\varpi_2},\qquad L_{\ib{\varepsilon}{b}} := L_{\varepsilon\varpi_b}.
\end{equation}
\end{definition}

Recall that the diagonal homomorphism $\rho:\mathbb{C}\dcover W\to \tama$ is not injective. In fact, its kernel is $\mathbb{C}\dcover W_+\cong \mathbb{C} W$ and $\rho(\mathbb{C}\dcover W)\cong \mathbb{C}\dcover W_-$. However, restricting $\rho$ to $\dcover W$ induces an injective group homomorphism into the units of $\tama$. We denote the conjugation action of $\dcover W$ on $\tama$ by
\begin{equation}\label{eq:conjact}
\Ad(\dcover w)(X) = \rho(\dcover w) X \rho(\dcover w)^{-1},
\end{equation}
for any $\dcover w\in \dcover W$ and $
X\in \tama$. Clearly, $\Ad(z)$ acts trivially on $\tama$. Also, note that since the volume element $x_1\wedge x_2\wedge x_3\wedge x_4 \in \bigwedge(V)$ is invariant under any orthogonal transformation and has even parity, the element $Z\in \tama$ is invariant for the conjugation action of $\dcover W$.

The ladder elements $\{L_\alpha\mid\alpha\in\Phi\}$ and the abelian algebra $\mathfrak{a}$ determine an $11$-dimensional subspace
\begin{equation}\label{eq:T-space}
    \mathfrak{l} := \mathfrak{a}\oplus \textup{span}\{L_\alpha\mid \alpha\in\Phi\}\subset\tama.
\end{equation}
This subspace is a $\dcover W$-module for the conjugation action~\eqref{eq:conjact}.
We have the following description of this action, with $\zeta_1 := e^{i\pi /m_1}$ and $\zeta_2 := e^{i\pi /m_2}$.
   \begin{lemma}\label{lem:actionWonT}
  The $\dcover{W}$-conjugation action on $\mathfrak{l}$ is given as follows.
  First $\Ad(z)L = L$ for any $L\in \mathfrak{l}$ and $\Ad(\dcover w)Z = Z$ for all $\dcover w\in\dcover W$.
  For $1\leq p \leq m_1$ and $1\leq q \leq m_2$, we have
   \begin{equation}
   \begin{aligned}
    \Ad(\dcover{s}_p) H_1 &= - H_1, &
   		\Ad(\dcover{t}_{q}) H_1 &= H_1,&
   		\Ad(\dcover{s}_{p}) H_2 &= H_2,  &
   		\Ad(\dcover{t}_{q}) H_2 &= -H_2.
   \end{aligned}
    \end{equation}
    Furthermore, we have
    \begin{equation}
        \begin{aligned}
        \Ad(\dcover{s}_p) L_{\ib{\pm}{1}} &= \zeta_1^{\pm2p}L_{\ib{\mp}{1}}, &
   		\Ad(\dcover{t}_{q}) L_{\ib{\pm}{1}} &= L_{\ib{\pm}{1}}, &
   		\Ad(\dcover{s}_p) L_{\ib{\pm}{2}} &= L_{\ib{\pm}{2}}, &
   		\Ad(\dcover{t}_{q}) L_{\ib{\pm}{2}} &=\zeta_2^{\pm2q}L_{\ib{\mp}{2}},
        \end{aligned}
    \end{equation}
    and
    \begin{equation}
	\Ad(\dcover{s}_p)L_{\ib{\pm\varepsilon}{12}} =
	-\zeta_1^{\pm2 p} L_{\ib{\mp\varepsilon}{12}}, \qquad
   	\Ad(\dcover{t}_{q})L_{\ib{\delta\pm}{12}} =
   	-\zeta_2^{\pm2q} L_{\ib{\delta\mp}{12}}.
   	\end{equation}
   	for all $\delta,\varepsilon\in\{+,-\}$.  In particular, for $\dr{1}:= z \dcover{s}_{m_1}\dcover{s}_1$ and $\dr{2}:=z \dcover{t}_{m_2}\dcover{t}_1$ we have
   	\begin{equation}
   	\begin{aligned}
   	\Ad(\dr{1}) L_{1}^{\pm} &= \zeta_1^{\pm 2}L_{1}^{\pm}, & \Ad(\dr{2}) L_{1}^{\pm} &= L_{1}^{\pm},\\
   	\Ad(\dr{1}) L_{2}^{\pm} &= L_{2}^{\pm}, &
    \Ad(\dr{2}) L_{2}^{\pm} &= \zeta_2^{\pm 2}L_{2}^{\pm},\\
   		\Ad(\dr{1})L_{\ib{\pm\varepsilon}{12}} &=
	\zeta_1^{\pm2} L_{\ib{\pm\varepsilon}{12}}, &
   	\Ad(\dr{2})L_{\ib{\delta\pm}{12}} &=
   	\zeta_2^{\pm2} L_{\ib{\delta\pm}{12}}.
   	\end{aligned}
   	\end{equation}

   \end{lemma}
   \begin{proof}
   Straightforward computations using~\eqref{eq:WComm} and~\eqref{eq:actionWisotropic}.
   \end{proof}

\subsection{Properties of the ladder operators}

Recall the abbreviations in Notation~\ref{not:isotropic}.

\begin{lemma}\label{lem:1index}
For  $\delta,\varepsilon\in\{+,-\}$ and distinct elements $a,b\in \{1,2\}$, we have
\[
O_{\ib{\delta}{1}}O_{\ib{\varepsilon}{2}} = -O_{\ib{\varepsilon}{2}}O_{\ib{\delta}{1}}, \qquad
\Tepsa{\delta}O_{\ib{\bar\delta}{a}} = O_{\ib{\delta}{a}}\Tepsa{\bar\delta}, \qquad
\Tepsa{\delta}O_{\ib{\varepsilon}{b}} = O_{\ib{\varepsilon}{b}}\Tepsa{\delta}, \qquad \text{where $\bar \delta := -\delta$.}
\]
\end{lemma}
\begin{proof}
These relations follow by~\eqref{eq:WComm} and~\eqref{eq:actionWisotropic}.
\end{proof}

\begin{lemma}\label{lem:H2index}
   In the associative algebra $H_\kappa\otimes\ca C$, for $\delta,\varepsilon\in\{+,-\}$, the following hold:
    \begin{align*}
[H_1,O_{\ib{\delta\varepsilon}{12}}] &= \delta(O_{\ib{\delta\varepsilon}{12}} + 2O_{\ib{\delta}{1}}O_{\ib{\varepsilon}{2}}) - O_{\ib{\delta}{1}}O_{\ib{+-\varepsilon}{112}}, && \text{and} &
[H_2,O_{\ib{\delta\varepsilon}{12}}] &= \varepsilon(O_{\ib{\delta\varepsilon}{12}} + 2O_{\ib{\delta}{1}}O_{\ib{\varepsilon}{2}}) + O_{\ib{\varepsilon}{2}}O_{\ib{\delta+-}{122}}.
    \end{align*}
\end{lemma}
\begin{proof}
From Theorem~\ref{thm:TAMArel}, which can be found for arbitrary vectors in~\cite{Os22}, we have
 \begin{align*}
[O_{\ib{+-}{11}},O_{\ib{\delta\varepsilon}{12}}] &= \delta2(O_{\ib{\delta\varepsilon}{12}} + [O_{\ib{\delta}{1}},O_{\ib{\varepsilon}{2}}]) -  \{O_{\ib{\delta}{1}},O_{\ib{+-\varepsilon}{112}}\}, &
[O_{\ib{+-}{22}},O_{\ib{\delta\varepsilon}{12}}] &= \varepsilon2(O_{\ib{\delta\varepsilon}{12}} + [O_{\ib{\delta}{1}},O_{\ib{\varepsilon}{2}}]) + \{O_{\ib{\varepsilon}{2}},O_{\ib{\delta+-}{122}}\}.
    \end{align*}
The results now follow from Lemma~\ref{lem:1index} and the definitions of $H_1$ and $H_2$.
\end{proof}

\begin{lemma}\label{lem:H3index}
For $\varepsilon\in \{+,-\}$ and distinct elements $a,b\in \{1,2\}$, we have
\begin{align*}
[H_a,O_{\ib{\varepsilon+-}{abb}}] &= \varepsilon O_{\ib{\varepsilon+-}{abb}} + 4 O_{\ib{\varepsilon}{a}}(\varepsilon  H_b+Z  ), && \text{and} &
[H_b,O_{\ib{\varepsilon+-}{abb}}] &= 0.
\end{align*}
\end{lemma}

\begin{proof}
For the first equation, from~\cite{Os22} we have
\begin{align*}
[O_{\ib{+-}{aa}},O_{\ib{\varepsilon+-}{abb}}] & = \varepsilon2 (O_{\ib{\varepsilon+-}{abb}}
+ \{O_{\ib{\varepsilon}{a}}, O_{\ib{+-}{bb}}\})
- [O_{\ib{\varepsilon}{a}}, O_{\ib{+-+-}{aabb}}] .
\end{align*}
The result follows from Lemma~\ref{lem:actionWonT} and the definitions of $H_1=\tfrac{1}{2}O_{\ib{+-}{11}}$, $H_2=\tfrac{1}{2}O_{\ib{+-}{22}}$ and $Z=\tfrac{-1}{4}O_{\ib{+-+-}{aabb}}$.

For the second equation, we shall proceed in a different way using the element $P$ defined in~\eqref{eq:Proj}. First, note that if $S,T\in H_\kappa\otimes \ca C$ commute with the $\mathfrak{sl}(2)$-triple of~(\ref{eq:EvenPart}), then $P(P(S)T) = P(S)P(T)$. It follows that for any $v\in \bigwedge(V)$, if $\llbracket O_{\ib{+-}{bb}},\gamma(v)\rrbracket = 0$ then $\llbracket O_{\ib{+-}{bb}},O(v)\rrbracket = 0$. We apply this to $\gamma(z_a^\varepsilon\wedge z_b\wedge\bar z_b)=\gamma(z_a^\varepsilon)\gamma(z_b\wedge\bar z_b)$ and $2H_b = O_{\ib{+-}{bb}}$ to obtain
\begin{align*}
\llbracket O_{\ib{+-}{bb}},\gamma(z_a^\varepsilon\wedge z_b\wedge\bar z_b)\rrbracket &= \llbracket O_{\ib{+-}{bb}},\gamma(z_a^\varepsilon)\rrbracket \gamma(z_b\wedge\bar z_b)+\gamma(z_a^\varepsilon) \llbracket O_{\ib{+-}{bb}},\gamma(z_b\wedge\bar z_b)\rrbracket=0. \qedhere
\end{align*}
\end{proof}

\begin{proposition}\label{prop:LadderProp}
For any $H \in \mathfrak{h}$ and $\alpha\in\Phi$, we have
\begin{equation}\label{eq:LadEq}
\llbracket H,L_\alpha \rrbracket = \alpha(H)L_\alpha.
\end{equation}
\end{proposition}
\begin{proof}
Let $H = cH_1 + dH_2$ with $c,d\in\mathbb{C}$. Suppose first that $\alpha = \delta\varpi_1+\varepsilon\varpi_2 \in \Phi_{\bar 0}$ with $\delta,\varepsilon\in\{+,-\}$. Note that $\alpha(H) = \delta c+\varepsilon d$. Using Lemma~\ref{lem:H2index}, we compute
\begin{align*}
\llbracket H,L_\alpha \rrbracket &= [H,\{H_1,\{H_2,O_{\ib{\delta\varepsilon}{12}}\}\}]=\{H_1,\{H_2,[H,O_{\ib{\delta\varepsilon}{12}}]\}\}\\
&=c\{H_1,\{H_2,e(O_{\ib{\delta\varepsilon}{12}} + [O_{\ib{\delta}{1}},O_{\ib{\varepsilon}{2}}]) - \{O_{\ib{\delta}{1}},O_{\ib{\varepsilon+-}{211}}\}\}\}
+d\{H_1,\{H_2,f(O_{\ib{\delta\varepsilon}{12}} - [O_{\ib{\delta}{1}},O_{\ib{\varepsilon}{2}}]) + \{O_{\ib{\varepsilon}{2}},O_{\ib{\delta+-}{122}}\}\}\}\\
&= (\delta c+\varepsilon d) L_\alpha,
\end{align*}
where we used $\{H_1,\{H_2,X\}\} = \{H_2,\{H_1,X\}\},$ for any $X\in H_\kappa\otimes \ca C$; $[O_{\ib{\delta}{1}},O_{\ib{\varepsilon}{2}}] = 2O_{\ib{\delta}{1}}O_{\ib{\varepsilon}{2}}$; $\acomm{O_{\ib{\delta}{1}}}{\Tepsdeux{\varepsilon}} = 2O_{\ib{\delta}{1}}\Tepsdeux{\varepsilon}$;
$\{O_{\ib{\varepsilon}{2}},\Tepsun{\delta}\} = 2O_{\ib{\varepsilon}{2}}\Tepsun{\delta}$; $\{H_a,O_{\ib{\delta}{1}}O_{\ib{\varepsilon}{2}}\} = 0$ and, by Lemma~\ref{lem:H3index},
\begin{align*}
\{H_2, \{H_1, O_{\ib{\delta}{1}}\Tepsdeux{\varepsilon}\}\} &=  -\{H_2, O_{\ib{\delta}{1}}[H_1,\Tepsdeux{\varepsilon}]\} = 0,&
\{H_1, \{H_2, O_{\ib{\varepsilon}{2}}\Tepsun{\delta}\}\} &=  -\{H_1, O_{\ib{\varepsilon}{2}}[H_2,\Tepsun{\delta}]\} = 0.
\end{align*}
Next, suppose that $\beta = \varepsilon\varpi_a\in\Phi_{\bar 1}$. Then, using Lemma~\ref{lem:H3index}
\begin{align*}
[H,L_\beta] &= [H,\{H_a,O_{\ib{\varepsilon+-}{abb}}\}] = \{H_a,[H,O_{\ib{\varepsilon+-}{abb}}]\}\\
&=\{H_a,\varepsilon\varpi_a(H)(O_{\ib{\varepsilon+-}{abb}} + 2\{O_{\ib{\varepsilon}{a}}, H_b\}) + 2[O_{\ib{\varepsilon}{a}}, Z]\}= \beta(H)L_\beta,
\end{align*}
since $\{O_{\ib{\varepsilon}{a}}, H_b\} = 2O_{\ib{\varepsilon}{a}}H_b$, $[O_{\ib{\varepsilon}{a}}, Z] = 2O_{\ib{\varepsilon}{a}}Z$, and $H_a$ anti-commutes with $O_{\ib{\varepsilon}{a}}$ and commutes with $H_b$ and $Z$. This finishes the proof.
\end{proof}


Despite Proposition~\ref{prop:LadderProp}, the ladder elements do not behave as nicely as root vectors in a Lie (super)algebra. In general, they do not satisfy the property $\llbracket L_\alpha, L_\beta\rrbracket \in \bbc L_{\alpha + \beta}$; see \cref{prop:ExtraTriangRules}. Certain products of two ladder elements have a useful factorisation, which we describe below. Before that, we describe useful formulas involving the generators of $\tama$.
\begin{lemma}\label{lem:LadderAlt}
    The ladder elements can be written alternatively as
   \begin{subequations}
    \begin{alignat}{2}\label{eq:oddLadderAlt}
        L_{\ib{\delta}{a}} &= \Tepsa{\delta}(2H_a+\delta) + 4O_a^{\delta}(\delta H_b+Z),\\\label{eq:evenLadderAlt}
        L_{\ib{\delta\varepsilon}{12}} &= O_{\ib{\delta\varepsilon}{12}}(2H_1+\delta)(2H_2+\varepsilon)  - O_{\ib{\delta}{1}} \Tepsdeux{\varepsilon}(2H_2+\varepsilon)+ \Tepsun{\delta} O_{\ib{\varepsilon}{2}}(2H_1+\delta) + 2O_{\ib{\delta}{1}} O_{\ib{\varepsilon}{2}}(\varepsilon\delta - 2Z)
    \end{alignat}
    \end{subequations}
    where $a,b$ are distinct elements of $\{1,2\}$ and $\delta,\varepsilon \in \{+,-\}$.
\end{lemma}

\begin{proof}
For (\ref{eq:oddLadderAlt}), we compute directly using Lemma~\ref{lem:H3index}:
\[
L_{\ib{\delta}{a}} = \{H_a,O_{\ib{\delta+-}{abb}}\} = 2O_{\ib{\delta+-}{abb}}H_a + [H_a,O_{\ib{\delta+-}{abb}}] = O_{\ib{\delta+-}{abb}}(2H_a+\delta) + 4O_a^{\delta}(\delta H_b+Z).
\]
The expression (\ref{eq:evenLadderAlt}) is slightly more complicated. From the definition of $L_{\ib{\delta\varepsilon}{12}}$ we compute
\[
L_{\ib{\delta\varepsilon}{12}} = \{H_1,[H_2,O_{\ib{\delta\varepsilon}{12}}]\} + 2[H_1,O_{\ib{\delta\varepsilon}{12}}]H_2 + 4O_{\ib{\delta\varepsilon}{12}}H_2H_1.
\]
Substituting the identities, which are computed using Lemma~\ref{lem:H2index},
\begin{align*}
    \{H_1,[H_2,O_{\ib{\delta\varepsilon}{12}}]\} &=
    O_{\ib{\delta\varepsilon}{12}}(2\varepsilon H_1 + \varepsilon\delta) + O_{\ib{\delta}{1}} O_{\ib{\varepsilon}{2}}(2\varepsilon\delta - 4Z -4\delta H_2) - \varepsilon O_{\ib{\delta}{1}} O_{\ib{\varepsilon}{2}}+ \Tepsun{\delta}O_{\ib{\varepsilon}{2}}(2H_1+\delta),\\
2[H_1,O_{\ib{\delta\varepsilon}{12}}]H_2 &= 2\delta O_{\ib{\delta\varepsilon}{12}}  H_2 + 4 \delta O_{\ib{\delta}{1}} O_{\ib{\varepsilon}{2}}  H_2 - 2O_{\ib{\delta}{1}} \Tepsun{\delta} H_2,
\end{align*}
the claim follows.
\end{proof}

\begin{lemma}\label{lem:producttwoindex}
    The following factorisations hold
    \begin{subequations}
    \begin{alignat}{2}
O_{\ib{\pm\pm}{12}}O_{\ib{\mp\mp}{12}}
& = (H_1+ H_2)^2 \mp 2(H_1+H_2)- Z^2 + Z   +\tfrac34
- 2 (O_{\ib{\pm}{1}}O_{\ib{\mp}{1}}+O_{\ib{\pm}{2}}O_{\ib{\mp}{2}}) 
\mp (O_{\ib{\pm}{1}}O_{\ib{\mp+-}{122}}+O_{\ib{\pm}{2}}O_{\ib{+-\mp}{112}}),  \\
O_{\ib{\pm\mp}{12}} O_{\ib{\mp\pm}{12}}
& = (H_1- H_2)^2 \mp 2(H_1 -H_2)- Z^2 - Z +\tfrac34
- 2 (O_{\ib{\pm}{1}}O_{\ib{\mp}{1}}+O_{\ib{\mp}{2}}O_{\ib{\pm}{2}}) 
\pm(O_{\ib{\pm}{1}}O_{\ib{\mp+-}{122}}-O_{\ib{\mp}{2}}O_{\ib{+-\pm}{112}})   .
\end{alignat}
\end{subequations}
\end{lemma}
\begin{proof}
Note that $O_{\ib{\pm\pm}{12}}O_{\ib{\mp\mp}{12}} = \tfrac{1}{2}(
\{O_{\ib{++}{12}},O_{\ib{--}{12}}\}\pm [O_{\ib{++}{12}},O_{\ib{--}{12}}])$. From
Theorem~\ref{thm:TAMArel}, (whose commutation relations for arbitrary vectors can be found in~\cite{Os22}), we get
\[
\tfrac{1}{2}[O_{\ib{++}{12}},O_{\ib{--}{12}}]=-(2H_1+2H_2+[O_1^+,O_1^-]+[O_2^+,O_2^-])-\tfrac{1}{2}O_{\ib{+}{1}}O_{\ib{-+-}{122}} -\tfrac{1}{2}O_{\ib{-}{1}}O_{\ib{++-}{122}}
-\tfrac{1}{2}O_{\ib{+}{2}}O_{\ib{+--}{112}}
-\tfrac{1}{2}O_{\ib{-}{2}}O_{\ib{+-+}{112}}.
\]
Computing $Z-Z^2$ using \eqref{eq:Z-symmetry3} and \eqref{eq:CasimirZ2} of Proposition~\ref{prop:Z}, we obtain an expression for
$
\tfrac{1}{2}\{O_{\ib{++}{12}},O_{\ib{--}{12}}\}
$ and settle the first equation. The second equation is similarly treated from
$O_{\ib{\pm\mp}{12}}O_{\ib{\mp\pm}{12}} = \tfrac{1}{2}(
\{O_{\ib{+-}{12}},O_{\ib{-+}{12}}\}\pm [O_{\ib{+-}{12}},O_{\ib{-+}{12}}])$, where we use $Z+Z^2$ to compute $\tfrac{1}{2}\{O_{\ib{+-}{12}},O_{\ib{-+}{12}}\}$.
\end{proof}

\begin{lemma}\label{lem:TT}
For distinct elements $a,b\in \{1,2\}$, we have
\begin{equation}\label{eq:TT1}
    O_{\ib{\pm+-}{abb}}O_{\ib{\mp+-}{abb}} = 4((Z \mp H_b)^2 - (H_a \mp \tfrac12)^2 +  O_{\ib{\pm}{a}} O_{\ib{\mp}{a}}),
\end{equation}
 and for  $\delta,\varepsilon\in\{+,-\}$
 \begin{equation}\label{eq:TT2}
    \Tepsa{\delta} \Tepsb{\varepsilon} = 4O_{\ib{\delta\varepsilon}{ab}}( \varepsilon H_a- \delta H_b +Z -\tfrac{1}{2}\delta\varepsilon ) - 2\varepsilon O_{\ib{\delta}{a}} \Tepsb{\varepsilon} - 2\delta \Tepsa{\delta} O_{\ib{\varepsilon}{b}}  -  4\delta\varepsilon O_{\ib{\delta}{a}} O_{\ib{\varepsilon}{b}}.
\end{equation}
Furthermore,
\begin{equation}\label{eq:LongShortO-comm}
    [O_{\ib{\delta\pm}{12}},O_{\ib{\delta+-}{122}}] = \pm 2\{O_{\ib{\delta}{1}},O_{\ib{\delta\pm}{12}}\}
    \qquad\textup{and}\qquad
    [O_{\ib{\pm\varepsilon}{12}},O_{\ib{\varepsilon+-}{211}}] = \pm 2\{O_{\ib{\varepsilon}{2}},O_{\ib{\pm\varepsilon}{12}}\}.
\end{equation}
\end{lemma}
\begin{proof}
Following the strategy in \cite{Os22}, using the expansion (\ref{eq:3indFormula}) of the $3$-index symmetries and the projection $P$ of (\ref{eq:Proj}) we compute the product
$O(u\wedge v\wedge w)O(x\wedge y\wedge z)=-\tfrac{1}{2}P(O(u\wedge v\wedge w)\gamma(x\wedge y\wedge z))$ for vectors $u,v,w,x,y,z\in V$.
The choice $u=z_a^\pm,v=z_b,w=\bar z_b$ and $x=z_a^\mp,y=z_b,z=\bar z_b$ yields (\ref{eq:TT1}), while choosing  $u=z_a^\delta,v=z_b,w=\bar z_b$ and $x=z_b^\varepsilon,y=z_a,z=\bar z_a$ yields (\ref{eq:TT2}). The last equations follow from \eqref{eq:2ind3indComm1} and \eqref{eq:2ind3indComm2}. They also follow from \cite[Proposition 4.13]{Os22}, which is formulated in a more general fashion.
\end{proof}

\begin{lemma}\label{lem:LT}
For $\delta,\varepsilon\in\{+,-\}$ and distinct elements $a,b\in \{1,2\}$, we have
\begin{equation}\label{eq:TL}
   O_{\ib{\delta+-}{abb}}L_{\ib{\eps}{b}}  
   = 4(O_{\ib{\delta\varepsilon}{ab}}(2H_b + \eps)  + O_b^{\eps}O_{\ib{\delta+-}{abb}})
  ( \varepsilon H_a- \delta H_b + Z-\tfrac{1}{2}\delta \varepsilon) - 2\varepsilon( O_{\ib{\delta}{a}} \Tepsb{\varepsilon}  +2\delta O_{\ib{\delta}{a}} O_{\ib{\varepsilon}{b}})(2H_b + \eps) ,
\end{equation}
 and 
 \begin{equation}\label{eq:LT}
 \begin{aligned}
   L_{\ib{\eps}{b}}O_{\ib{\delta+-}{abb}}     &=
4(O_{\ib{\delta\varepsilon}{ab}}(2H_b + \eps) +O_{\ib{\eps}{b}} O_{\ib{\delta+-}{abb}})
( \varepsilon H_a- \delta H_b - Z+\tfrac{1}{2}\delta \varepsilon) \\
&\quad- 2\varepsilon (  O_{\ib{\delta}{a}}\Tepsb{\varepsilon}
-2\delta O_{\ib{\delta}{a}}O_{\ib{\varepsilon}{b}} )(2H_b + \eps) 
  -16 O_{\ib{\delta}{a}}O_{\ib{\eps}{b}}  (\delta H_b + Z).
\end{aligned}
\end{equation}
\end{lemma}
\begin{proof}
Using 
\eqref{eq:oddLadderAlt} of \cref{lem:LadderAlt}
and 
\eqref{eq:TT2} of \cref{lem:TT}
we have 
  \begin{align*}
O_{\ib{\delta+-}{abb}}L_{\ib{\eps}{b}}  &= 2(2O_{\ib{\delta\varepsilon}{ab}}( \varepsilon H_a- \delta H_b + Z-\tfrac{1}{2}\delta \varepsilon) - \varepsilon O_{\ib{\delta}{a}} \Tepsb{\varepsilon} - \delta \Tepsa{\delta} O_{\ib{\varepsilon}{b}}  -  4\delta\varepsilon O_{\ib{\delta}{a}} O_{\ib{\varepsilon}{b}})(2H_b + \eps) + 4O_{\ib{\delta+-}{abb}}O_b^{\eps}(\eps H_a+Z)\\
 & = 2(2O_{\ib{\delta\varepsilon}{ab}}( \varepsilon H_a- \delta H_b + Z-\tfrac{1}{2}\delta \varepsilon) - \varepsilon O_{\ib{\delta}{a}} \Tepsb{\varepsilon}  -  2\delta\varepsilon O_{\ib{\delta}{a}} O_{\ib{\varepsilon}{b}})(2H_b + \eps) + 4O_b^{\eps}O_{\ib{\delta+-}{abb}}(\eps H_a  -\delta H_b + Z -\tfrac{\delta\varepsilon}{2})\\
 & = 4(O_{\ib{\delta\varepsilon}{ab}}(2H_b + \eps)  + O_b^{\eps}O_{\ib{\delta+-}{abb}})
  ( \varepsilon H_a- \delta H_b + Z-\tfrac{1}{2}\delta \varepsilon) - 2\varepsilon( O_{\ib{\delta}{a}} \Tepsb{\varepsilon}  +2\delta O_{\ib{\delta}{a}} O_{\ib{\varepsilon}{b}})(2H_b + \eps)  .
  \end{align*}
  
Using Lemma~\ref{lem:H3index}, 
\eqref{eq:oddLadderAlt} of \cref{lem:LadderAlt}
and 
\eqref{eq:TT2} of \cref{lem:TT},
we have 
\begin{align*}   
L_{\ib{\eps}{b}}O_{\ib{\delta+-}{abb}}  &=
(O_{\ib{\eps+-}{baa}}(2H_b + \eps) + 4  O_{\ib{\eps}{b}}(\eps H_a+Z))O_{\ib{\delta+-}{abb}}
\\
&=
O_{\ib{\eps+-}{baa}}O_{\ib{\delta+-}{abb}}(2H_b + \eps)
+ 
4  O_{\ib{\eps}{b}} O_{\ib{\delta+-}{abb}}(\eps H_a-Z +\delta\eps )  +16 O_{\ib{\eps}{b}} O_{\ib{\delta}{a}} (\delta H_b + Z)\\  
  &=
(4O_{\ib{\varepsilon\delta}{ba}}( \delta H_b -\varepsilon H_a + Z-\tfrac{1}{2}\delta \varepsilon) 
- 2\delta  O_{\ib{\varepsilon}{b}}  \Tepsa{\delta}
- 2\varepsilon  \Tepsb{\varepsilon} O_{\ib{\delta}{a}}
-  4\delta\varepsilon O_{\ib{\varepsilon}{b}}O_{\ib{\delta}{a}} )(2H_b + \eps) \\
&\quad  + 
4  O_{\ib{\eps}{b}} O_{\ib{\delta+-}{abb}}(\eps H_a-Z +\delta\eps ) 
+16 O_{\ib{\eps}{b}} O_{\ib{\delta}{a}} (\delta H_b + Z)\\  
& =
4(O_{\ib{\delta\varepsilon}{ab}}(2H_b + \eps) +O_{\ib{\eps}{b}} O_{\ib{\delta+-}{abb}})
( \varepsilon H_a- \delta H_b - Z+\tfrac{1}{2}\delta \varepsilon) 
- 2\varepsilon ( \Tepsb{\varepsilon} O_{\ib{\delta}{a}}
+2\delta O_{\ib{\varepsilon}{b}}O_{\ib{\delta}{a}} )(2H_b + \eps) \\
&\quad
  +16 O_{\ib{\eps}{b}} O_{\ib{\delta}{a}} (\delta H_b + Z),
  \end{align*} 
  finishing the proof.
\end{proof}

\begin{proposition}\label{prop:OddOppLadders}
For $\beta = \delta \varpi_a \in \Phi_{\bar 1}$, with $\delta\in\{+,-\}$, $a\in \{1,2\}$ and $\{b\} = \{1,2\}\setminus \{a\}$, we have 
\begin{equation}
L_\beta L_{-\beta} = 16((H_a - \tfrac{1}{2}\delta)^2 - O_a^\delta O_a^{\bar\delta})((Z - \delta H_b)^2-(H_a - \tfrac{1}{2}\delta)^2),
\end{equation}
where $\bar\delta := -\delta$.
\end{proposition}

\begin{proof}
From (\ref{eq:oddLadderAlt}) we have
$L_{\ib{\delta}{a}} =  (2H_a - \delta)\Tepsa{\delta}  + 4(Z -\delta H_b)O_{\ib{\delta}{a}}.$
Using $\Tepsa{\delta}O_{\ib{\bar\delta}{a}} = O_{\ib{\delta}{a}}\Tepsa{\bar\delta}$ and Lemma~\ref{lem:H3index}, we get
\begin{equation}\label{eq:OddLadderEq2}
    L_{\ib{\delta}{a}}L_{\ib{\bar\delta}{a}}
    =  4(H_a - \tfrac{1}{2}\delta)^2\Tepsa{\delta} \Tepsa{\bar\delta} - 16(Z -\delta H_b)^2O_{\ib{\delta}{a}}O_{\ib{\bar\delta}{a}}.
\end{equation}
The desired claim follows by Lemma~\ref{lem:TT}.
\end{proof}

\begin{proposition}\label{prop:EvenEvenLadders}
    For $\alpha = \delta \varpi_1 + \varepsilon\varpi_2 \in \Phi_{\bar 0}$ with $\delta,\varepsilon \in \{+,-\}$, we have the factorisation
    \[
    L_\alpha L_{-\alpha} = 16((H_1 - \tfrac{1}{2}\delta)^2 - O_1^\delta O_1^{\bar \delta})
                             ((H_2 - \tfrac{1}{2}\varepsilon)^2 - O_2^\varepsilon O_2^{\bar\varepsilon})((\delta H_1 + \varepsilon H_2 -1)^2 - (Z - \tfrac{1}{2}\delta\varepsilon)^2 ),
    \]
    where $\bar\delta := -\delta$ and $\bar\varepsilon := -\varepsilon$.
\end{proposition}


\begin{proof}
For $\delta,\varepsilon\in\{+,-\}$ and distinct elements $a,b\in \{1,2\}$, consider the following elements of $\tama$
\begin{align*}
p_a^\delta &:= (H_a - \tfrac{1}{2}\delta)^2 - O_a^\delta O_a^{\bar\delta}, &
\Xi^{\delta,\varepsilon} &:= 4\varepsilon H_1 Z + 2\varepsilon\delta Z + 2\delta H_1 + 1,\\
q_a^\delta &:= (Z -\delta H_b)^2 - (H_a - \tfrac{1}{2}\delta)^2,
&
\Upsilon^{\delta,\varepsilon} &:= Z+\varepsilon H_1 -\delta H_2 - \tfrac{1}{2}\delta\varepsilon.
\end{align*}
These elements all live in the associative subalgebra of $\tama$ generated by $\mathfrak{t}_0$ (see Definition~\ref{def:Tzero}). Now fix $\delta,\eps$ from the root $\alpha = \delta \varpi_1 + \varepsilon\varpi_2 \in \Phi_{\bar 0}$ of the statement.
It is straightforward to check that $\Upsilon^{\delta,\varepsilon}$ divides both $q_1^\delta$ and $q_2^\varepsilon + \Xi^{\bar\delta,\varepsilon}$ and we have
\[
q_1^\delta = (Z-\varepsilon H_1 -\delta H_2 + \tfrac{1}{2}\delta\varepsilon)\Upsilon^{\delta,\varepsilon},
\quad
q_2^\varepsilon + \Xi^{\bar\delta,\varepsilon} = (Z+\varepsilon H_1 +\delta H_2 - \tfrac{3}{2}\delta\varepsilon)\Upsilon^{\delta,\varepsilon}.
\]
Now, on the one hand (as $[\Upsilon^{\delta,\varepsilon},L_{-\alpha}]=0$), we have
$L_1^\delta L_2^\varepsilon L_2^{\bar\varepsilon} L_1^{\bar\delta} = -16L_\alpha L_{-\alpha} (\Upsilon^{\delta,\varepsilon})^2$
and on the other hand, using $[q_2^\varepsilon, L_1^{\bar\delta}] = L_1^{ \bar\delta}\Xi^{\bar\delta,\varepsilon}$ and $[p_2^\varepsilon,L_1^{\bar\delta}]=0$,
we have
\begin{align*}
L_1^\delta L_2^\varepsilon L_2^{\bar\varepsilon} L_1^{\bar\delta} &=L_1^\delta (16p_2^\varepsilon q_2^\varepsilon) L_1^{\bar\delta} 
=16 L_1^\delta L_1^{\bar\delta} p_2^\varepsilon q_2^\varepsilon + 16 L_1^\delta [p_2^\varepsilon q_2^\varepsilon, L_1^{\bar\delta}]\\
&=16 L_1^\delta L_1^{\bar\delta} p_2^\varepsilon (q_2^\varepsilon + \Xi^{\bar\delta,\varepsilon})
= 16^2 p_1^\delta p_2^\varepsilon q_1^\delta (q_2^\varepsilon + \Xi^{\bar\delta,\varepsilon}).
\end{align*}
Factoring out $(\Upsilon^{\delta,\varepsilon})^2$ (since the associative algebra generated by $\mathfrak{t}_0$ is isomorphic to a polynomial ring, thus a domain), we get
$L_\alpha L_{-\alpha}= 16(p_1^\delta p_2^\varepsilon)((\delta H_1 + \varepsilon H_2 - 1)^2 - (Z-\tfrac{1}{2}\delta\varepsilon)^2),
$ where
\[
(Z-\varepsilon H_1 -\delta H_2 + \tfrac{1}{2}\delta\varepsilon)(Z+\varepsilon H_1 +\delta H_2 - \tfrac{3}{2}\delta\varepsilon)) = (Z-\tfrac{1}{2}\delta\varepsilon)^2 - (\delta H_1 + \varepsilon H_2 - 1)^2
\]
is obtained by completing squares. This finishes the proof.
\end{proof}

\begin{proposition}\label{prop:OddOddLaddersIsEven}
The ladder operators associated with odd roots admit the following factorisations
\begin{equation}
L_a^\delta L_b^\varepsilon = 4L_{\ib{\delta\varepsilon}{ab}}(Z+\varepsilon H_a -\delta H_b - \tfrac{1}{2}\delta\varepsilon),
\end{equation}
for $\delta,\varepsilon \in \{+,-\}$ and distinct elements $a,b\in \{1,2\}$.
\end{proposition}
\begin{proof}
Using (\ref{eq:oddLadderAlt}), the commutation relations of Lemma~\ref{lem:H3index} and the action described in Lemma~\ref{lem:actionWonT}, we compute that
\begin{multline}\label{eq:PrododdLadders2}
     L_{\ib{\delta}{a}} L_{\ib{\varepsilon}{b}} = \Tepsa{\delta}  \Tepsb{\varepsilon}(2H_a+\delta) (2H_b+\varepsilon)+ 4 \Tepsa{\delta}O_{\ib{\varepsilon}{b}} (2H_a+\delta)(Z+\varepsilon H_b) \\
- 4O_{\ib{\delta}{a}} \Tepsb{\varepsilon}( Z-\delta\varepsilon -\delta H_b)(2H_b+\varepsilon) - 16 O_{\ib{\delta}{a}} O_{\ib{\varepsilon}{b}}( Z-\delta\varepsilon -\delta H_b)(Z+\varepsilon H_a).
\end{multline}
Substituting (\ref{eq:TT2}) in (\ref{eq:PrododdLadders2}) and using the expression (\ref{eq:evenLadderAlt}) yield the desired result.
\end{proof}

\subsection{The triangular subalgebra}

Recall the partition $\Phi = \Phi_+\cup \Phi_-$ of~(\ref{eq:rootsPosNeg})
and the vector space $\mathfrak{t}_0$ of Definition~\ref{def:Tzero}.
Given any vector subspace $U\subset \tama$, let $\alg A(U)$ denote the associative subalgebra of $\tama$ generated by $U$.

\begin{definition}
    We let $\mathfrak{t} = \mathfrak{t}_0 \oplus \mathfrak{t}_+\oplus \mathfrak{t}_-$   with  $\mathfrak{t}_{\pm} = \mathrm{span}\{L_\alpha\mid \alpha \in \Phi_{\pm}\}$. The subalgebra  $\mathfrak{T}= \alg A(\mathfrak t )$ will be called the triangular subalgebra of $\tama$. Moreover, let $\mathfrak{T}_\pm = \alg A(\mathfrak{t}_\pm)$, $\mathfrak{T}_0 = \alg A(\mathfrak{t}_{0})$.
\end{definition}

This nomenclature is justified by Proposition~\ref{prop:triangstruc}, proved below. The computations in \cref{prop:OddOppLadders,prop:OddOddLaddersIsEven,prop:EvenEvenLadders} indicate that the product of ladder elements is again a ladder element, modulo right multiplication by a polynomial expression on $\mathfrak{t}_0$.

\begin{proposition}\label{prop:ExtraTriangRules}
    Let $(\alpha,\beta)$ be a pair of non-opposite roots. Then, the following assertions hold.
    \begin{enumerate}
        \item If $(\alpha,\beta)\in\Phi_{\bar 0}\times\Phi_{\bar 1} \cup \Phi_{\bar 1}\times\Phi_{\bar 0}$ and the angle between these roots is obtuse, then $L_\alpha L_\beta = (L_{\alpha+\beta}) p$, for some $p \in \mathfrak{T}_0$. \label{prop:ExtraTriangRulesItemi}
        \item If $(\alpha,\beta)\in\Phi_{\bar 0}\times\Phi_{\bar 1} \cup \Phi_{\bar 1}\times\Phi_{\bar 0}$ and the angle between these roots is acute, then  $[L_\alpha, L_\beta] = 0$. \label{prop:ExtraTriangRulesItemii}
        \item If $(\alpha,\beta)\in\Phi_{\bar 0}\times\Phi_{\bar 0}$, then $L_\alpha L_\beta = (L_{(\alpha+\beta)/2})^2 p$,
        for some $p \in \mathfrak{T}_0$. \label{prop:ExtraTriangRulesItemiii}
    \end{enumerate}
\end{proposition}

\begin{proof}
In \ref{prop:ExtraTriangRulesItemi}, note that $\alpha + \beta \in \Phi_{\bar 1}$. We can thus assume $\alpha = \delta\varpi_a + \eps\varpi_b$ and $\beta = -\varepsilon\varpi_b$ for $\delta,\eps\in\{+,-\}$ and $a,b\in\{1,2\}$. Then there exist polynomials $q_1,q_2\in \mathfrak{T}_0$ such that $q_2$ commutes with $L_{\ib{\bar\varepsilon}{b}}$ and
\[
    L_{\ib{\delta}{a}}(L_{\ib{\varepsilon}{b}}L_{\ib{\bar\varepsilon}{b}}) = L_{\ib{\delta}{a}}q_1
    =L_{\ib{\delta\varepsilon}{ab}}L_{\ib{\bar\varepsilon}{b}}q_2 = (L_{\ib{\delta}{a}}L_{\ib{\varepsilon}{b}})L_{\ib{\bar\varepsilon}{b}},
\]
where we recall $\bar\varepsilon := -\varepsilon$.
Moreover, from the precise expressions for $q_1$ and $q_2$ coming from \cref{prop:OddOppLadders,prop:OddOddLaddersIsEven}, one can compute that $q_2$ divides $q_1$. For \ref{prop:ExtraTriangRulesItemii}, we can assume $\alpha = \delta\varpi_a + \varepsilon\varpi_b$ and $\beta = \varepsilon\varpi_b$ so that, using Proposition~\ref{prop:OddOddLaddersIsEven}, we can find $q_1,q_2\in \mathfrak{T}_0$ such that
\[
L_{\ib{\varepsilon}{b}}(L_{\ib{\delta}{a}}L_{\ib{\varepsilon}{b}}) = L_{\ib{\varepsilon}{b}}L_{\ib{\delta\varepsilon}{ab}}q_1 =
L_{\ib{\delta\varepsilon}{ab}}q_2L_{\ib{\varepsilon}{b}} =(L_{\ib{\varepsilon}{b}}L_{\ib{\delta}{a}})L_{\ib{\varepsilon}{b}}.
\]
In addition, one can compute that $q_2L_{\ib{\varepsilon}{b}} = L_{\ib{\varepsilon}{b}}q_1$.

Finally, for \ref{prop:ExtraTriangRulesItemiii}, we can assume $\alpha = \delta\varpi_a + \varepsilon\varpi_b$ and $\beta = \delta\varpi_a -\varepsilon\varpi_b$. Using again \cref{prop:OddOppLadders,prop:OddOddLaddersIsEven}, we find polynomials $q_1,q_2,q_3\in \mathfrak{T}_0$ such that
\[
L_{\ib{\delta}{a}}(L_{\ib{\varepsilon}{b}}L_{\ib{\bar\varepsilon}{b}})L_{\ib{\delta}{a}} =
L_{\ib{\delta}{a}} L_{\ib{\delta}{a}}q_1 = L_{\ib{\delta\varepsilon}{ab}} L_{\ib{\delta\bar\varepsilon}{ab}}q_2q_3=
(L_{\ib{\delta}{a}}L_{\ib{\varepsilon}{b}})(L_{\ib{\bar\varepsilon}{b}}L_{\ib{\delta}{a}})
\]
with the property that $q_2,q_3$ are coprime of degree one and both divide $q_1$; hence
$q_2q_3$ divides $q_1$.
\end{proof}

\begin{proposition}\label{prop:triangstruc}
 The triangular subalgebra $\mathfrak{T}$ admits a decomposition
\begin{equation}
    \mathfrak{T} =
    \mathfrak{T}_-\mathfrak{T}_+\mathfrak{T}_0.
\end{equation}
\end{proposition}

\begin{proof}
Due to Proposition~\ref{prop:LadderProp}, Lemma~\ref{lem:actionWonT} and $ZL_\alpha = (-1)^{|\alpha|}L_\alpha Z$, it is clear that any expression in $\mathfrak{T}$ is written as sums of elements in $\mathfrak{T}_{\pm}\mathfrak{T}_0$. Now let $M = T_1T_2\cdots T_nA$ be any monomial expression in $\mathfrak{T}$, with $T_i\in\mathfrak{T}_\pm$ and $A\in \mathfrak{T}_0$. We show that we can rearrange $M$ and write it as an element in $\mathfrak{T}_-\mathfrak{T}_+\mathfrak{T}_0$, that is, a finite sum $M = \sum_j N_jP_jA_j$, with $N_j\in \mathfrak{T}_-$, $P_j\in \mathfrak{T}_+$ and $A_j\in\mathfrak{T}_0$. Define the inversion set of a monomial $M = T_1T_2\cdots T_nA$ to be
\[
\mathcal{I}(M) = \{(i,j)\mid 1\leq i,j \leq n, i<j \textup{ and } T_i\in \mathfrak{T}_+, T_j\in \mathfrak{T}_-\}.
\]
If $|\mathcal{I}(M)| = 0$, we are done. Else, there exist ladder elements $T_i,T_{i+1}$ with $T_i = L_\alpha\in \mathfrak{T}_+$ and $T_{i+1}=L_\beta\in \mathfrak{T}_-$. Moreover, we can assume that the index $i$ is maximal for this property, which implies that the remaining factors $T_{i+2}\cdots T_n$ are correctly ordered. There are four cases we need to analyse: (a) $(\alpha,\beta)\in \Phi_{\bar 1}\times \Phi_{\bar 1}$, (b) $(\alpha,\beta)\in \Phi_{\bar 0}\times \Phi_{\bar 1}$, (c) $(\alpha,\beta)\in \Phi_{\bar 1}\times \Phi_{\bar 0}$ and (d) $(\alpha,\beta)\in \Phi_{\bar 0}\times \Phi_{\bar 0}$. We claim that, in each case, when we evaluate $T_{i}T_{i+1}$, we get $M = \sum_j M_j$ with each $M_j\in \mathfrak{T}_{\pm}\mathfrak{T}_0$ monomial with $|\mathcal{I}(M_j)|<|\mathcal{I}(M)|$. So after finitely many steps, we shall write $M\in \mathfrak{T}_-\mathfrak{T}_+\mathfrak{T}_0$.

Note that we can assume that $(\alpha,\beta)$ are not opposite roots, since otherwise $L_\alpha L_{-\alpha} \in \mathfrak{T}_0$ would trivially decrease the cardinality of the inversion set. In case (a), we then use Proposition~\ref{prop:OddOddLaddersIsEven} and in (b), (c) we use Lemma~\ref{prop:ExtraTriangRules} items~\ref{prop:ExtraTriangRulesItemi} and~\ref{prop:ExtraTriangRulesItemii}. In all these cases, we can either commute $T_iT_{i+1} = T_{i+1}T_i$ or $T_iT_{i+1} = TA$ for some $T\in \mathfrak{T}$ and $A \in \mathfrak{T}_0$. After sending $A$ to the right, we have a linear combination of monomials with smaller inversion sets.

Case (d) is a bit more complicated, since, by Lemma~\ref{prop:ExtraTriangRules} item~\ref{prop:ExtraTriangRulesItemiii}, the product $T_iT_{i+1} = T^2A$, for some odd ladder element $T$. If $T\in \mathfrak{T}_-$, after sending $A$ to the right, this would produce correctly ordered monomials $M'$ so that in each $T_1\cdots T_{i-1} M'$, the maximal index where an inversion occurs becomes $i_1<i$, and we can repeat the process for that $i_1$. So assume $T\in \mathfrak{T}_+$. Note that this could increase the inversion sets of the produced monomials $M'$ after we send $A$ to the right. However, from the maximality assumption on $i$, we necessarily need to deal with the configuration $T^2L_1\cdots L_m$ where  each $L_j\in \mathfrak{T}_-$. Since $T$ is an odd ladder element, when we deal with this configuration, case (iv) will not occur in the first step, and hence we can, after finitely many steps, reorder $T(L_1\cdots L_m) = \sum_jM_jP_jA_j$ in $\mathfrak{T}_-\mathfrak{T}_+\mathfrak{T}_0$. We then obtain $T^2L_1\cdots L_m = \sum_j TM_jP_jA_j$. But note that $|\mathcal{I}(TL_1\cdots L_m)|=m$ and also $|\mathcal{I}(TM_j)|\leq m$ for each $j$, since the expressions $M_j\in\mathfrak{T}_-$ are a product of negative ladder elements with at most $m$ factors. After reordering each $TM_j$ (using the same principle that allowed us to reorder $TL_1\cdots L_m)$, the original monomial is replaced by a linear combination
\[
M = T_1T_2\cdots T_nA = T_1T_2\cdots T_{i-1}\left(\sum_k M_kP_kA_k\right).
\]
with $|\mathcal{I}(T_1T_2\cdots T_{i-1}M_k)|<|\mathcal{I}(M)|$, for each $k$.  This finishes the proof.
\end{proof}
\begin{remark}
Note that the decomposition of \cref{prop:triangstruc} is not a full PBW-type decomposition, because the multiplication map $m:\mathfrak T_-\otimes\mathfrak T_+ \otimes \mathfrak T_0\to \mathfrak T$ is not injective. Indeed, $1\otimes \La\Lb \otimes 1 - 4\otimes \Mab \otimes (Z + H_1 -H_2 - \frac12 ) $ lies in the kernel, using \cref{prop:OddOddLaddersIsEven}. Nevertheless, the existence of such a triangular reordering of the monomials in $\mathfrak{T}$ will prove to be very useful when dealing with the representation theory of the $\tama$.
\end{remark}

To conclude this section, we state and prove the following result related to the $\ast$-structure defined in \cref{sec:unitarity}.

\begin{proposition}\label{prop:StarStrucTriangAlg}
   For $\kappa$ real-valued, the triangular subalgebra $\mathfrak{T}$ is a $*$-subalgebra of $\mathfrak{O}_\kappa$, with $\mathfrak{T}_0^* = \mathfrak{T}_0$ and $\mathfrak{T}_\pm^* = \mathfrak{T}_\mp$.
    In particular, for $a\in \{1,2\}$ and $\alpha \in \Phi = \Phi_{\bar 0}\cup \Phi_{\bar 1}$, we have
    \begin{equation}
    H_a^* = H_a, \qquad L_\alpha^* = -(-1)^{|\alpha|}L_{-\alpha},\qquad Z^*=Z,\qquad (\dr1)^*=(\dr1)^{-1},\qquad (\dr2)^*=(\dr2)^{-1}.
    \end{equation}
\end{proposition}
\begin{proof}
The equations above follow from straightforward computations using Lemma~\ref{lem:starO} and the  identities $\{u,v\}^* = \{u^*,v^*\}$ and $\{u,\{v,x\}\}^* = \{u^*,\{v^*,x^*\}\}$, which hold for any associative algebra endowed with an anti-involution. It is also clear that $\dcover{W}_{(\bar 0,\bar 0)}$ is closed for the $\ast$-operation, finishing the proof.
\end{proof}

\section{Finite representation theory of the symmetry algebra}  \label{sec:rep}

In this section, we investigate the finite-dimensional irreducible representations of $\tama$. In order to discuss $\ast$-unitary representations of $\tama$, we will assume the parameter function $\kappa$ to be real-valued. 

The main tool we will use for our investigation is the triangular subalgebra $\TriSubAlg$.
We start by defining a weight theory using the fact that $\mathfrak{t}_0$ is abelian (see Proposition \ref{prop:HsFirst}), and we 
will later use the highest weight as a label for a finite-dimensional irreducible representation. In Theorem~\ref{thm:ClassiLabel}, we will give a set of equations these labels must satisfy.  

\begin{definition}\label{def:Weights}
For $M$ a $\tama$-module (or $\mathfrak{T}_0$-module) and $\mu\in\mathfrak{t}_0^*$, $M_\mu = \{v\in M\mid Av = \mu(A)v,\textup{ for all } A\in\mathfrak{t}_0\}$ and we denote by $\textup{Wt}_{\mathfrak{t}_0}(M) = \{\mu\in\mathfrak{t}_0^*\mid M_\mu\neq 0\}$ the set of $\mathfrak{t}_0$-weights of $M$. The set of $\mathfrak{h}$-weights $\textup{Wt}_\mathfrak{h}(M)$ is defined similarly. 
Whenever convenient, we shall see a $\mathfrak{t}_0$-weight as a $5$-tuple, consisting of the values of the weight with respect to the elements $H_1,H_2,Z,\dr1,\dr2 \in \mathfrak{t}_0$, and an $\mathfrak{h}$-weight as a $2$-tuple consisting of the values of the weight with respect to the elements $H_1$, $H_2$.

For any $\lambda\in \textup{Wt}_{\mathfrak{t}_0}(M)$, a nonzero vector $\hwv\in M_\lambda$ will be called a weight vector of the weight space $ M_\lambda$.
\end{definition}

We will now look at how the ladder operators and the group interact with the weight spaces. 
Recall that, 
in view of Theorem~\ref{thm:TAMAgen} and \cref{eq:DiagImage}, the restriction of any $\tama$-module $M$ to the subalgebra $\rho(\mathbb{C}\dcover W)\subset \tama$ decomposes, by Maschke's Theorem, as a direct sum of irreducible spin representations. The eigenvalues of $\dr1$ and $\dr2$ are then of the form~\eqref{e:actu} with  $\ell_1$ and $\ell_2$ odd integers. 

We define actions of (the lattice generated by) the root system $ \Phi$ and the group $\dcover W$ on $\mathfrak{t}_0^*$ as follows. For $\alpha = \alpha_1 \varpi_1 + \alpha_2 \varpi_2 \in \Phi$ and $\mu = (\lambda_1,\lambda_2,\Lambda,\zeta_1^{\ell_1},\zeta_2^{\ell_2}) \in \mathfrak{t}_0^*$, 
    \begin{equation}\label{eq:rootactweight}
\alpha\cdot\mu  = (\lambda_1 + \alpha_1,\lambda_2 +\alpha_2,(-1)^{|\alpha|}\Lambda,\zeta_1^{\ell_1+2\alpha_1 },\zeta_2^{\ell_2+2\alpha_2 }),
\end{equation}
while, recalling the decomposition~\eqref{eq:WZ2} of $\dcover{W}$, for $\dcover w \in\dcover{W}_{(\bar\imath,\bar\jmath\,)}$ with $(\bar\imath,\bar\jmath\,)\in \mathbb{Z}_2^2$, we let
 \begin{equation}\label{eq:rgroupactweight}
 \dcover w \cdot \mu = ((-1)^{\bar\imath}\lambda_1 ,(-1)^{\bar\jmath}\lambda_2 ,(-1)^{\bar\imath+\bar\jmath}\Lambda,\zeta_1^{(-1)^{\bar\imath}\ell_1 },\zeta_2^{(-1)^{\bar\jmath}\ell_2 }).
\end{equation}
One readily verifies that these actions are compatible with the group operations:  for $\mu\in \mathfrak{t}_0^*$ one has $(\alpha + \beta) \cdot \mu = \alpha \cdot( \beta \cdot \mu)$ for $\alpha,\beta \in \Phi$ and $(\dcover w_1\dcover w_2) \cdot \mu = \dcover w_1 \cdot( \dcover w_2 \cdot \mu)$ 
for $\dcover w_1,\dcover w_2\in \dcover{W}$.
\begin{lemma}\label{lem:actRootsWeights}
Let $M$ be a $\tama$-module, $\mu  \in \textup{Wt}_{\mathfrak{t}_0}(M)$, $\alpha \in \Phi$, $w \in\dcover{W}$. Then 
$L_{\alpha} M_\mu \subset M_{\alpha\cdot\mu}$ 
and $\dcover w M_\mu \subset M_{\dcover{w}\cdot\mu}$. 
\end{lemma}
\begin{proof}
  This follows from Proposition~\ref{prop:LadderProp}, Lemma~\ref{lem:actionWonT}, and finally \eqref{eq:Zcommutation} for the $Z$-weight.
\end{proof}

We introduce a lexicographic-like ordering on the $\mathfrak{t}_0$-weights for the sake of defining the notion of highest weight, which we will use to label finite-dimensional irreducible representations. The definition below will have the consequence that the highest weight is annihilated by all positive ladder operators. In general this condition alone will not yield a single weight. Moreover, in order to make a coherent choice, regardless of the value of the parameter function $\kappa$, we define the order at the restriction of the lattice at $\kappa=0$; see \cref{fig:ExKappaToZero} for an example.

\begin{definition} 
Let $M$ be a $\tama$-module. 
For $\lambda,\mu\in \textup{Wt}_{\mathfrak{t}_0}(M) $, we denote by 
$\lambda_0$ (resp.~$\mu_0$) the evaluation of $\lambda$ (resp.~$\mu$) at $\kappa=0$.
For $\lambda,\mu\in \textup{Wt}_{\mathfrak{t}_0}(M) $, we say that $\lambda$ is higher than $\mu$ when
$\lambda_0(H_1) \geq \mu_0(H_1)$, and if $\lambda_0(H_1) = \mu_0(H_1)$, then $\lambda_0(H_2)\geq \mu_0(H_2)$, where if $\lambda_0(H_2) = \mu_0(H_2)$, then $\lambda_0(Z)\geq \mu_0(Z)$, where if $\lambda_0(Z)= \mu_0(Z)$, then $\ell_1 \leq \ell_1'$ where  $\ell_1,\ell_1'\in \{1,\dots, 2m_1-1\}$ are odd integers such that $\lambda(\dr1)=\zeta_1^{\ell_1} $ and $\mu(\dr1)=\zeta_1^{\ell_1'} $ and finally if  $\ell_1 = \ell_1'$ then $\ell_2 \leq \ell_2'$ where  $\ell_2,\ell_2'\in \{1,\dots, 2m_2-1\}$ are odd integers such that $\lambda(\dr2)=\zeta_2^{\ell_2} $ and $\mu(\dr2)=\zeta_2^{\ell_2'} $. 

A $\mathfrak{t}_0$-weight $\lambda \in \textup{Wt}_{\mathfrak{t}_0}(M)$ will be called a highest weight (of $M$)  
if it is higher than all other $\mu \in \textup{Wt}_{\mathfrak{t}_0}(M)$, and any nonzero vector $\hwv\in M_\lambda$ will be called a highest-weight vector.
\end{definition}

\newcommand{\DIMex}{2}
\newcommand{\pkapa}{3/4}
\newcommand{\pkapb}{9/4}
\begin{figure}[h]
    \[
    \begin{tikzpicture}[scale=.8,baseline={(current bounding box.center)}]
 \foreach \i in {2,...,\DIMex}
    {
    \foreach \j in {\i,...,\DIMex}
        {
    	\foreach \x/\y in {1/1, 1/-1, -1/1, -1/-1}
             {
            \linkbetween{{\x*(\DIMex-\j+1-.5+\pkapa)}}{{\y*(\i-.5-\pkapb}}{{\x*(\DIMex-\j+1-.5+\pkapa)}}{{\y*(\i-1-.5-\pkapb)}}{orange}
            \linkbetween{{\x*(\i-.5+\pkapa)}}{{\y*(1-.5-\pkapb)}}{{\x*(\i-1-.5+\pkapa)}}{{\y*(1-.5-\pkapb)}}{teal} }
        }
    }
    \foreach \i in {1,...,\DIMex}
        {
        \foreach \j in {\i,...,\DIMex}
        {
        \foreach \x/\y in {1/1, 1/-1, -1/1, -1/-1}
        {
    \draw[color = black,fill=gray, opacity=.7]
    ({\x*(\i-.5+\pkapa)},{\y*(\DIMex-\j+1-.5-\pkapb)}) circle(3pt);
    }
    }}
    \draw (\DIMex-.5+\pkapa,1-\pkapb) node {$\hwv$};
    \draw[dotted,->] (-\DIMex,0) -- (\DIMex,0) node[right] {$\varpi_1$};
    \draw[dotted,->] (0,-\DIMex) -- (0,\DIMex) node[above] {$\varpi_2$} ;
\end{tikzpicture}
\qquad \xrightarrow{\kappa =0}
\qquad 
  \begin{tikzpicture}[scale=.8, baseline={(current bounding box.center)}]
 \foreach \i in {2,...,\DIMex}
    {
    \foreach \j in {\i,...,\DIMex}
        {
    	\foreach \x/\y in {1/1, 1/-1, -1/1, -1/-1}
             {
            \linkbetween{{\x*(\DIMex-\j+1-.5)}}{{\y*(\i-.5}}{{\x*(\DIMex-\j+1-.5)}}{{\y*(\i-1-.5)}}{orange}
            \linkbetween{{\x*(\i-.5)}}{{\y*(1-.5)}}{{\x*(\i-1-.5)}}{{\y*(1-.5)}}{teal} }
        }
    }
    \foreach \i in {1,...,\DIMex}
        {
        \foreach \j in {\i,...,\DIMex}
        {
        \foreach \x/\y in {1/1, 1/-1, -1/1, -1/-1}
        {
    \draw[color = black,fill=gray, opacity=.7]
    ({\x*(\i-.5)},{\y*(\DIMex-\j+1-.5)}) circle(3pt);
    }
    }}
    \draw (\DIMex-.5,1) node {$\hwv$};
    \draw[dotted,->] (-\DIMex,0) -- (\DIMex,0) node[right] {$\varpi_1$};
    \draw[dotted,->] (0,-\DIMex) -- (0,\DIMex) node[above] {$\varpi_2$} ;
\end{tikzpicture}
\]
\caption{
Depiction of the $\mathfrak{h}$-weights in the $\tama$-module given by the polynomial Dunkl monogenics of degree 1 associated with the group $W=D_6\times D_6$.
On the left is the case where $\kappa_1 = 1/4, \kappa_3 = -3/4$, and on the right is how they are placed at $\kappa=0$. 
In order for the vector $\hwv$ on the left to reduce to the one on the right, we define the order on $\mathfrak{t}_0$-weights at $\kappa=0$.
Each cluster of three connected nodes represents a  $\TriSubAlg$-submodule.
}\label{fig:ExKappaToZero}
\end{figure}
We note that when restricted to $\mathfrak{h}$-weights and on an simple $\TriSubAlg$-submodule of a $\tama$-module, the lexicographical order used in the previous definition amounts to an $\mathfrak{h}$-weight being bigger than another $\mathfrak{h}$-weight if their difference is a linear combination of positive roots with positive integer coefficients, in analogy with the theory of semisimple Lie algebras. 
In general this does not hold on the full $\tama$-module because of the action of $\dcover W$, see~\eqref{eq:rgroupactweight}.

\begin{lemma}
Let $M$ be a $\tama$-module with a highest-weight vector $\hwv \in M$. We have
 $L_\alpha\hwv=0$ for all $\alpha\in\Phi_+$.
\end{lemma}
\begin{proof}
This follows from \cref{lem:actRootsWeights} and~\eqref{eq:rootactweight}. Indeed, for $\alpha\in\Phi_+$, we must have $L_\alpha\hwv=0$, otherwise $L_\alpha\hwv$ would have a higher weight than $\hwv$. 
\end{proof}

\subsection{Coarse classification}\label{sec:CoClass}
Using the factorisation results of \cref{prop:OddOppLadders,prop:EvenEvenLadders}, we can determine conditions in the form of a set of equations that  
an element $\mu\in\mathfrak{t}_0^*$ must respect if it appears as the highest weight of a finite-dimensional $\tama$-representation. 
Such necessary conditions imposed on a finite-dimensional irreducible representation is what we will refer to as a \emph{coarse classification}.

In order to do so, we first consider the action of the symmetries of the form $O(x)$ for $x\in V$. 
Recall that these are elements of $ \rho(\mathbb{C}\dcover W)$; see~\eqref{eq:Ox}.

\begin{lemma}
\label{lem:actionOOu}
For odd integers $\ell_a\in \{1,\dots, 2m_a-1\}$, with $a\in\{1,2\}$, let $u(\ell_1,\ell_2)$ be a simple module for the quotient group $\dcover{W}_{(\bar 0, \bar 0)}$ with the action~\eqref{e:actu} on $ \mathfrak{u}\in u(\ell_1,\ell_2)$. Then, for $a\in\{1,2\}$
\begin{equation}\label{eq:ActionOwithQa}
O_{\ib{\pm}{a}}  \mathfrak{u}
=  \pm i \, Q_a(\ell_a\pm 1) \,\dcover f_a \mathfrak{u} , \quad \text{and} \quad O_{\ib{\pm}{a}}O_{\ib{\mp}{a}}  \mathfrak{u}
=  Q_a(\ell_a\mp 1)^2 \mathfrak{u},
  \end{equation}
where
  \begin{equation}\label{eq:Qa}
\!Q_a(j):=
\begin{cases}
    m_a \kappa_{2a-1}  ,
    & \text{if $m_a$ is odd, and  $j\equiv 0 \modSB 2m_a$;}\\
     \tfrac{m_a}{2}(\kappa_{2a}\! + \kappa_{2a-1}),
    & \text{if $m_a$ is even, and $j\equiv 0 \modSB 2m_a$;}\\
    \tfrac{m_a}{2}(\kappa_{2a}-\kappa_{2a-1}),
    & \text{if $m_a$ is even, and $j\equiv m_a \modSB 2m_a$;}\\
   0, & \text{otherwise.}
\end{cases}
\end{equation}
\end{lemma}
\begin{proof} 
Note that from~\eqref{eq:O1f1}, if
$\mathfrak{u}\in u(\ell_1,\ell_2)$ then
$\dr{a} (O_{\ib{\pm}{a}}\mathfrak{u}) = \zeta_a^{-\ell_a}  (O_{\ib{\pm}{a}}\mathfrak{u})$ for $a\in\{1,2\}$, where $\zeta_a= e^{i\pi/m_a}$. On the other hand, using~\eqref{eq:WComm} and~\eqref{eq:actionWisotropic} we have $\dr{a}(O_{\ib{\pm}{a}}\mathfrak{u}) = \zeta_a^{\pm 2}O_{\ib{\pm}{a}}\dr{a}\mathfrak{u} = \zeta_a^{\ell_a\pm 2}(O_{\ib{\pm}{a}}\mathfrak{u})$. Thus, if  $O_{\ib{\pm}{a}}\mathfrak{u}\neq 0$ then $\ell_a\pm 2 \equiv -\ell_a\modSB 2m_a$ or, equivalently, $2(\ell_a\pm 1) \equiv 0\modSB 2m_a$.

We assume $a=1$ and prove the lemma for the first dihedral group $\dih_{2m_1}$; the results for the group $\dih_{2m_2}$ follow in the same fashion. 
Using the realisation of Notation~\ref{not:refl}, it follows that $z \dcover r_1^{\ p} \dcover f_1 = \dcover s_p $ for $p=1,\dotsc,m_1$.  From~\eqref{eq:Ox}, using~\eqref{eq:dihroots}, \eqref{eq:z}  and $\rho(z) = -1$, we obtain
\begin{equation}\label{eq:O1f1}
O(z_{\ib{\pm}{1}})
= \mp i \sum_{p=1}^{m_1}  \kappa_{\alpha_p} \zeta_1^{\pm p}
\dcover s_p
= \pm i \sum_{p=1}^{m_1}  \kappa_{\alpha_p}  \dcover f_1(\zeta_1^{\pm 1} \dcover r_1)^p.
\end{equation}
Denoting $\omega_\pm := \zeta_1^{\ell_1\pm 1 }$, we have $\zeta_1^{\pm 1} \dcover r_1 \mathfrak{u} = \omega_\pm  \mathfrak{u}$ for $\mathfrak{u}\in u(\ell_1,\ell_2)$, and
\begin{equation}
 \sum_{p=1}^{m_1}  \kappa_{\alpha_p} \omega_\pm ^p =\begin{cases}
    \kappa_1   \sum_{p=1}^{m_1}  \omega_\pm ^p,
    & \text{if $m_1$ is odd;}\\
  \sum_{p=1}^{m_1/2} (\kappa_1    \omega_\pm^{2p-1}
    +  \kappa_2 \omega_\pm^{2p} ),
    & \text{if $m_1$ is even.}
    \end{cases}
\end{equation}
The first equation of~\eqref{eq:ActionOwithQa} now follows from  $\omega_\pm $ being an $m_1$\textsuperscript{th} root of unity, while the second equation follows from $ \dcover f_1^2 =1$ and $Q_1(-j) = Q_1(j)$.
\end{proof}

For a $\tama$-module $M$,
we define, for  $\mu = (\lambda_1,\lambda_2,\Lambda,\zeta_1^{\ell_1},\zeta_2^{\ell_2}) \in \textup{Wt}_{\mathfrak{t}_0}(M)$,
 a nonnegative integer $K\in\NN$, 
 an odd root $\alpha= \delta \varpi_a\in \Phi_{\bar 1}$, and an even root $\beta  = \eps_1 \varpi_1 + \varepsilon_2 \varpi_2\in\Phi_{\bar 0}$  
 with  $\delta,\varepsilon_1,\eps_2 \in \{+,-\}$ 
 and $a\in\{1,2\}$, the two quantities
   \begin{subequations}\label{eq:A}
   \begin{alignat}{2}
    A(\mu,\alpha,K) &:= \big((\lambda_a + \delta  (K + \tfrac{1}{2}))^2 - Q_a(\ell_a + \delta ( 2K +1))^2\big)\big(((-1)^K\Lambda + \delta \lambda_b)^2 - (\lambda_a + \delta  (K + \tfrac{1}{2}))^2\big),\label{eq:Ashort}\\
    A(\mu,\beta,K) &:= \big((\lambda_1 + \varepsilon_1( K +\tfrac12))^2 - Q_1(\ell_1 + \varepsilon_1(2K+1))^2\big)
        \big((\lambda_2 + \varepsilon_2(K+\tfrac12) )^2 - Q_2(\ell_2 + \varepsilon_2(2K+1))^2\big) \nonumber\\
        &\quad \times
        \big((\varepsilon_1 \lambda_1 + \varepsilon_2 \lambda_2 + 1 +2K)^2 - (\Lambda - \tfrac{\varepsilon_1\varepsilon_2}{2})^2\big), \label{eq:Along}
   \end{alignat}
\end{subequations}
so that, by \cref{lem:actionWonT}, for $\hwv \in M_\mu$,
\begin{align}
    L_{-\alpha}L_{\alpha}L_{\alpha}^K\hwv &= 16A(\mu,\alpha,K) L_{\alpha}^K\hwv, & L_{-\beta}L_{\beta} L_{\beta}^K  \hwv &= 16A(\mu,\beta,K) L_{\beta}^K\hwv.
\end{align}
We will use this to state the main theorem of the section. 

   \begin{theorem}\label{thm:ClassiLabel}
Let $\mathcal{V}$ be a finite-dimensional irreducible $\tama$-representation. For a highest weight $\mu$, there exist nonnegative integers $N_{-\alpha}\in \NN$ for all $\alpha\in\Phi_+$ such that the following system of equations is satisfied:
\begin{equation}\label{eq:thmclassrel}
	    A(\mu,\alpha,0) =0,\quad
		A(\mu,-\alpha,N_{-\alpha}) = 0,\quad  A(\mu,-\alpha,j)\neq 0, \qquad 0\leq j < N_{-\alpha}, \quad \text{for all } \alpha\in\Phi_+.
\end{equation}
   \end{theorem}
   \begin{proof}
   Let $\mathcal{V}$ be a finite-dimensional irreducible representation of $\tama$.  By \cref{prop:HsFirst}, we know $\mathcal{V}$ admits a basis of common eigenvectors of $\mathfrak{t}_0$. We take a highest-weight vector and denote it by $\hwv$. 

Since $\hwv$ is a highest-weight vector, we have $L_\alpha \hwv = 0$ for all $\alpha\in\Phi_+$, or else this would result in a vector with a higher weight than $\hwv$ by \cref{lem:actRootsWeights}. 
Using the factorisations of \cref{prop:OddOppLadders,prop:EvenEvenLadders}, this implies the following four relations:  $L_{-\alpha}L_{\alpha}\hwv  = 0 = A(\mu, \alpha, 0)$,  $\alpha \in \Phi_+$.

Now, let $\alpha\in \Phi_+$, then for every value $N\in \mathbb{N}$ 
the vector $L_{-\alpha}^N\hwv$ will have a different weight by \cref{lem:actRootsWeights}. 
Since $\mathcal V$ is finite-dimensional, we must have $L_{-\alpha}^N\hwv =0$ for some  $N\in \mathbb{N}$. 
Let $N_{-\alpha}$ for $\alpha\in \Phi_+$ denote the minimal value 
such that $L_{-\alpha}^{N_{-\alpha}}\hwv \neq0$, while $L_{-\alpha}^{N_{-\alpha}+1}\hwv = 0$. 

The equations $(L_{-\alpha})^{N_{-\alpha}+1}\hwv = 0 $ then imply that $L_{\alpha}L_{-\alpha} (L_{-\alpha})^{N_{-\alpha}} =0$, and thus that $A(\mu, -\alpha, N_{-\alpha})=0$ follows from the factorisations of \cref{prop:OddOppLadders,prop:EvenEvenLadders}. The minimality of the integers $N_{-\alpha}$ then implies $A(\mu, -\alpha, j)\neq 0$ for $0\leq j < N_{-\alpha}$.
   \end{proof}

We express pictorially in Figure~\ref{fig:generalpicture} what \cref{thm:ClassiLabel} entails. Each finite-dimensional irreducible representation of $\tama$ has four chains of elements in $\TriSubAlg$ of respective length $N_{-\alpha}$ for each highest-weight vector $\hwv$. 
  In this picture and from now on, we will follow the conventions of Notation~\ref{not:isotropic} and when convenient may write $N_{\ib{\delta}{a}}$ or $N_{\ib{\delta\varepsilon}{12}}$, with $a\in \{1,2\}$ and $\delta,\varepsilon\in\{+,-\}$, instead of $N_{-\alpha}$ for $\alpha\in\Phi_+$.

 \begin{figure}[h]
\centering
\begin{tikzpicture}[baseline={(current bounding box.center)},scale=0.5]
\foreach \x in {0,...,4}{
    \foreach \y in {-1,...,4}{
        \fill (2*\x +1,2*\y +1) circle(1pt);
    }}
    \linkbetween{9}{3}{7}{5}{gq-green}
    \linkbetween{5}{7}{7}{5}{gq-green}
    \linkbetween{9}{3}{7}{3}{gq-mauve}
    \linkbetween{7}{3}{5}{3}{gq-mauve}
    \linkbetween{5}{3}{3}{3}{gq-mauve}
    \linkbetween{3}{3}{1}{3}{gq-mauve}
    \linkbetween{9}{3}{9}{1}{orange}
    \linkbetween{9}{-1}{9}{1}{orange}
    \linkbetween{9}{3}{7}{1}{blue}
    \linkbetween{5}{-1}{7}{1}{blue}
	\draw[dotted, gray] (0,-2) -- (0,9); 
	\draw[dotted, gray] (0,-2) -- (9,-2); 
	\draw [decorate, decoration = {brace}]  (5.25,7.5)--  (9.25,3.5);
	\draw (8,6) node{$\Nmp$};
	\draw [decorate, decoration = {brace}]  (8.75,2.5)--  (5.25,-1.25);
	\draw (7.75,0.25) node{$\Nmm$};
	\draw [decorate, decoration = {brace}]  (8.75,2.5)--  (1.25,2.5);
	\draw (5,1.5) node{$\Mam$};
	\draw [decorate, decoration = {brace}]  (9.5,3)--  (9.5,-1);
	\draw (10.5,1) node{$\Mbm$};
	\draw (9,3) node[above right]{$\hwv$};
\end{tikzpicture}
\caption{Four chains of vectors from \cref{thm:ClassiLabel} in any finite-dimensional $\tama$-representation. Diagonal North-West links mean moving with $\Mcb$; diagonal South-West links mean moving with $\Mcd$; horizontal West with $\Lc$, and vertical South with $\Ld$. The integers $\Nmp$, $\Nmm$, $\Mam$ and $\Mbm$ are associated with their respective root.}\label{fig:generalpicture}
\end{figure}

The first equations of the  system of equations~\eqref{eq:thmclassrel} come from $\mu$ being a highest weight; the second set comes from the finite-dimensionality of the representation, and the third set comes from its irreducibility.

\subsection{Restriction to the triangular subalgebra}
In this section, we explore conditions under which the restriction to the triangular subalgebra $\TriSubAlg$ and the group $\dcover W$ is enough to determine the whole action of $\tama$ on a finite-dimensional irreducible representation.

\begin{definition}
We define $\dcover{\TriSubAlg}$ as the subalgebra of $\tama$ generated by $\rho(\dcover W)$ and $\mathfrak{T}$. 
Given a $\tama$-module $M$, for $v\in M$, we define $\dcover{\mathfrak{T}}(v)$ as the vector subspace of $M$ spanned by all translates of $v$ by the actions of $\mathbb{C}\dcover W$ and $\mathfrak{T}$, that is,
 \begin{equation}
\dcover{\mathfrak{T}}(v) := \textup{span}\{\rho(\dcover w)T v\mid \dcover w\in \dcover W,T\in \mathfrak T\}.
  \end{equation} 
\end{definition}

Outside specific values of the $\mathfrak h$-weight, the action of the two- and three-index symmetries on $v$ lies in $\dcover{\mathfrak T}(v)$ and can be given explicitly.
\begin{proposition}\label{prop:ActionTwoThreeSymOutsideHalf}
Let $M$ be a finite-dimensional $\tama$-module and let $\mu\in \textup{Wt}_{\mathfrak{t}_0}(M)$. 
If for $\delta \in \{+,-\}$ and $a\in\{1,2\}$, one has $\mu(H_a) \neq-\delta\tfrac{1}{2}$,
then for $\wv\in M_\mu$, 
\begin{equation}\label{eq:actOijk}
	    O_{\ib{\delta+-}{abb}}(\wv) = \frac{(L_{\ib{\delta}{a}} - 4O_a^{\delta}(\delta H_b+Z))}{2\mu(H_a) + \delta}\wv,
	     \end{equation}
and, if for $\delta,\varepsilon \in \{+,-\}$, one has $\mu(H_1) \neq-\delta\tfrac{1}{2}$ and $\mu(H_2) \neq-\varepsilon\tfrac{1}{2}$,  then  for $\wv\in M_\mu$,
	    \begin{equation}\label{eq:actOij}
   O_{\ib{\delta\eps}{12}}(\wv) = \frac{ \big(L_{\ib{\delta\eps}{12}}
+ O_{\ib{\delta}{1}} O_{\ib{\eps+-}{211}}(2H_2+\eps)
- O_{\ib{\delta+-}{122}} O_{\ib{\eps}{2}}(2H_1+\delta)
- 2O_{\ib{\delta}{1}} O_{\ib{\eps}{2}}(\eps\delta - 2 Z)\big)}{(2\mu(H_1)+\delta)(2\mu(H_2)+\eps)}\wv.
	    \end{equation}
\end{proposition}
\begin{proof}
 This is a direct application of \cref{lem:LadderAlt}, where $O_{\ib{\delta+-}{abb}}(\wv)$ and $O_{\ib{\delta\eps}{12}}(\wv)$ can be isolated since the conditions prevent a division by zero.
\end{proof}

If a $\tama$-module has weights attaining the values excluded by \cref{prop:ActionTwoThreeSymOutsideHalf}, one cannot recover the action of the two- and three-index symmetries in this way. 
Nevertheless, we can still obtain information about their action in another way.

\begin{lemma}\label{lem:LadDontCrossAxis}
  Let $M$ be a finite-dimensional $\tama$-module and  $a,b$ be distinct elements of $ \{1,2\}$. If there is a weight $\mu \in \textup{Wt}_{\mathfrak{t}_0}(M)$ such that $\mu(H_a) = -\delta\tfrac{1}{2}$ for some $\delta \in \{+,-\}$, then for $ \wv\in M_\mu$ and  $\eps \in \{+,-\}$ the following statements hold. Denoting  $\alpha = \delta \varpi_a$ and $\beta = \delta \varpi_a + \eps \varpi_b$,
    \begin{enumerate}
       \item $O_{\ib{\delta+-}{abb}}(\wv) \in  M_{\alpha\cdot\mu}$;
       \label{enum:lemLad1}
        \item $L_{\ib{\delta}{a}}(\wv)=0$;\label{enum:lemLad2}
        \item  $O_{\ib{\delta\eps}{ab}}(\wv) \in  
        \begin{cases}
          M_{\beta\cdot \mu}\oplus M_{\df{b}\cdot \alpha \cdot \mu}, 
            & \text{if }\mu(H_b) \neq - \varepsilon \tfrac12 \text{ and }
            \mu(\dr{b}) = \zeta_b^{\ell_b}$ with $\ell_b \equiv -\varepsilon \modSB m_b;
            \\
            M_{\beta\cdot\mu}, &
            \text{else;} 
        \end{cases}$
        \label{item:2indWS1}
      \item $\Lad{\ib{\delta\eps}{ab}}(\wv) =0$. \label{item:2indWS2}
    \end{enumerate}
\end{lemma}
\begin{proof}
Assume $a = 1$ and $\wv\in M_\mu$ with $\mu = (-\delta\tfrac{1}{2},\mu_2,\mu_Z,\zeta_1^{\ell_1},\zeta_2^{\ell_2})$ for $\delta \in \{+,-\}$; the case $a=2$ is similar. We write 
\begin{equation}\label{eq:3indexhalfaction}
O_{\ib{\delta+-}{122}}(\wv) = \sum_{\nu\in \textup{Wt}_{\mathfrak{t}_0}(M)} \wv_\nu,
\end{equation}
 for some $\wv_\nu\in M_\nu$. The relations $[H_2,O_{\ib{\delta+-}{122}}]=0=[\dr2,O_{\ib{\delta+-}{122}}]=\{Z,O_{\ib{\delta+-}{122}}\}$ fix the respective components of the weights that can occur in the decomposition~\eqref{eq:3indexhalfaction}. For the $\dr{1}$-component, using~\eqref{eq:WComm} and~\eqref{eq:actionWisotropic} we have $\dr{1}(O_{\ib{\delta+-}{122}}\wv) = \zeta_1^{\delta2}O_{\ib{\delta+-}{122}}\dr{1}\wv = \zeta_1^{\ell_1+ \delta 2}O_{\ib{\delta+-}{122}}\wv$, giving the claimed $\dr{1}$-component. Next, we look at the $H_1$-component.
 Using  \cref{lem:H3index}, we get
\begin{align}
H_1O_{\ib{\delta+-}{122}}(\wv)-O_{\ib{\delta+-}{122}}H_1(\wv) & = \delta O_{\ib{\delta+-}{122}}(\wv)  + 4O_{\ib{\delta}{1}}(\delta H_2 + Z)(\wv),
\end{align}
and thus
\begin{align}\label{eq:ProofLadDontCrossEq2}
(H_1 - \delta\tfrac{1}{2})O_{\ib{\delta+-}{122}}(\wv) = 4(\delta \mu_2+\mu_Z)O_{\ib{\delta}{1}}(\wv).\end{align}
We will now show that~\eqref{eq:ProofLadDontCrossEq2} must equal zero. By \cref{lem:actionOOu} we know the weight space of the right-hand side: $O_{\ib{\delta}{1}}(\wv) \in M_{(\delta/2,\mu_2,-\mu_Z,\zeta_1^{-\ell_1},\zeta_2^{\ell_2})}$. As the left-hand side must be in the same weight space, either $(H_1 - \delta\tfrac{1}{2})$ will act by zero, or $O_{\ib{\delta+-}{122}}(\wv) =0$; in both cases, \eqref{eq:ProofLadDontCrossEq2} is zero and we can conclude the proof of item~\ref{enum:lemLad1}.

 Now, using~\eqref{eq:oddLadderAlt}, item~\ref{enum:lemLad2} follows: 
\begin{equation}
    L_{\ib{\delta}{a}}(\wv) = \Tepsa{\delta}(2H_a+\delta)(\wv) + 4O_a^{\delta}(\delta H_b+Z)(\wv) = 0.
\end{equation}

Next, we prove item~\ref{item:2indWS1} so let  $\eps \in \{+,-\}$. The values of the $\dr 1$- and $\dr2$-components of $O_{\ib{\delta\varepsilon}{12}}(\wv)$ are found using again~\eqref{eq:WComm} and~\eqref{eq:actionWisotropic}. 
Moreover, $Z$ commutes with $O_{\ib{\delta\varepsilon}{12}}$. 
By \cref{lem:H2index}, 
\begin{equation}
    H_1O_{\ib{\delta\eps}{12}}(\wv) -O_{\ib{\delta\eps}{12}}H_1(\wv) = \delta O_{\ib{\delta\eps}{12}}(\wv) + \delta 2O_{\ib{\delta}{1}}O_{\ib{\eps}{2}}(\wv) - O_{\ib{\delta}{1}}O_{\ib{+-\eps}{112}}(\wv),
\end{equation}
and thus 
\begin{equation}\label{eq:WhereGoO12atCrit1}
    (H_1 - \delta\tfrac{1}{2})O_{\ib{\delta\eps}{12}}(\wv) =   \delta 2O_{\ib{\delta}{1}}O_{\ib{\eps}{2}}(\wv) - O_{\ib{\delta}{1}}O_{\ib{+-\eps}{112}}(\wv) .
\end{equation}
Equation~\eqref{eq:WhereGoO12atCrit1} must equal zero because the terms in the right-hand side all have $H_1$-weight $\delta\tfrac12$, so either $(H_1 - \delta\tfrac{1}{2})$ will act by zero, or $O_{\ib{\delta\eps}{12}}(\wv) =0$. 

All that remains is to determine the possible $H_2$-weights that can occur in a decomposition similar to~\eqref{eq:3indexhalfaction} for $O_{\ib{\delta\eps}{12}}(\wv)$.
If $\mu (H_2) = -\eps \tfrac12$,  then we can follow the same procedure for the $H_2$-weight, which shows that in this case $O_{\ib{\delta\eps}{12}}(\wv) \in  M_{\beta\cdot\mu}$ with  $\beta = \delta \varpi_1 + \eps \varpi_2$. 

Assume now $\mu (H_2) \neq -\eps \tfrac12 $, then using again \cref{lem:H2index}, one has
\begin{equation}\label{eq:WhereGoO12atCrit2}
    (H_2-(\mu(H_2)+\eps))O_{\ib{\delta\eps}{12}}(\wv) =  \eps 2O_{\ib{\delta}{1}}O_{\ib{\eps}{2}}(\wv) + O_{\ib{\eps}{2}}O_{\ib{\delta+-}{122}}(\wv) 
    =  O_{\ib{\eps}{2}}(  O_{\ib{\delta+-}{122}}-\eps 2O_{\ib{\delta}{1}})(\wv).
\end{equation}
Any term of $O_{\ib{\delta\eps}{12}}(\wv)$ that is in the weight space $M_{\beta\cdot\mu}$ with  $\beta = \delta \varpi_1 + \eps \varpi_2$ will be killed by the action of  $(H_2-(\mu(H_b)+\eps))$ in the left-hand side of \eqref{eq:WhereGoO12atCrit2}. Denote $\alpha = \delta \varpi_1$, the terms in the right-hand side go to the space $M_{\df{2}\cdot \alpha \cdot \mu}$, 
where we note that, when $\mu(H_1) = -\delta\tfrac{1}{2}$, we have $\df{1}\cdot\mu=\alpha \cdot\mu$ whenever $O_{\ib{\delta}{1}}$ has a nonzero action on $\wv \in M_\mu$.
Indeed, as seen in the proof of \cref{lem:actionOOu}, a one-index symmetry $O_{\ib{\delta}{a}}$ will have a nonzero action on $\wv \in M_\mu$ only if $\df{a}\cdot\mu(\dr{a}) = \zeta_a^{-\ell_a} = \zeta_a^{\ell_a+\delta 2} =\alpha \cdot\mu(\dr{a})$.
 
The final item is then given by expanding $L_{\ib{\delta\eps}{12}}(\wv)$. Denote $\mu(H_1) = \lambda_1$ and $\mu(H_2)  = \lambda_2$.
\begin{align*}
    L_{\ib{\delta\eps}{12}}u &= \acomm{H_1}{\acomm{H_2}{O_{\ib{\delta\eps}{12}}}}(\wv) = (H_1H_2O_{\ib{\delta\eps}{12}} + H_1O_{\ib{\delta\eps}{12}}H_2 + H_2 O_{\ib{\delta\eps}{12}}H_1 +O_{\ib{\delta\eps}{12}} H_1H_2)(\wv)\\
    &= H_2(H_1+\lambda_1)O_{\ib{\delta\eps}{12}}(\wv) + \lambda_2(H_1+\lambda_1)O_{\ib{\delta\eps}{12}}(\wv)
    = (H_1+\lambda_1)(H_2+\lambda_2)O_{\ib{\delta\eps}{12}}(\wv).
\end{align*}
This proves that  $L_{\ib{\delta\eps}{12}}u=0$ by item~\ref{item:2indWS1}.
\end{proof}
\begin{theorem}\label{thm:AllWeightSpace}
A simple finite-dimensional $\tama$-module $M$ can be expressed as a direct sum of $\dcover{\mathfrak T}$-modules according to the following situations:
\begin{enumerate}
    \item \label{enum:AllWS4T}
    if there is a weight $\mu$ such that $\mu(H_1)= \tfrac12=\mu(H_2)$, and $\mu(\dr 1) = \zeta_1^{\ell_1}$  with $\ell_1\not\equiv 1 \modSB m_1$, and $\mu(\dr 2) = \zeta_2^{\ell_2}$ with $\ell_2\not\equiv 1 \modSB m_2$; then, for any nonzero $\wv\in M_{\mu}$,
    \begin{equation}\label{eq:M=4T}
M =\dcover{\mathfrak{T}}(\wv) \oplus \dcover{\mathfrak{T}}(O_{\ib{-+-}{122}}(\wv)) \oplus \dcover{\mathfrak{T}}(O_{\ib{+--}{112}}(\wv))\oplus \dcover{\mathfrak{T}}(O_{\ib{--}{12}}(\wv));
 \end{equation}
 \item \label{enum:AllWS2T}
 if there is a weight $\mu$ such that,  for $a\in\{1,2\}$, $\mu(H_a) =\tfrac12$  and $\mu(\dr{a})=\zeta_a^{\ell_a}$ with $\ell_a\not\equiv 1 \modSB m_a$ while for 
  $b\in \{1,2\}\setminus\{a\}$, for every  weight $\lambda \in \textup{Wt}_{\mathfrak{t}_0}$, either $\lambda(H_b) \not\in \{-\tfrac12,\tfrac12\}$, or  $\lambda(H_b)= \pm\frac12$ and $\lambda(\dr b) = \zeta_b^{\ell_b}$ with $\ell_b \equiv \mp1 \modSB m_b$; then, for any nonzero $\wv\in M_{\mu}$,
 \begin{equation}\label{eq:M=2T}
M =\dcover{\mathfrak{T}}(\wv) \oplus\dcover{\mathfrak{T}}(O_{\ib{-+-}{abb}}(\wv));
 \end{equation}
 \item \label{enum:AllWS1T}
 or in all other cases, for any nonzero $\wv \in M$, then
 \begin{equation}\label{eq:M=1T}
     M = \dcover{\mathfrak T}(\wv).
 \end{equation}
\end{enumerate}
\end{theorem}
\begin{proof}
First, if there are no weights $\mu \in \textup{Wt}_{\mathfrak{t}_0}(M)$ such that $\mu(H_a)\in \{-\tfrac12,\frac12\}$ for $a\in\{1,2\}$, 
then, by \cref{prop:ActionTwoThreeSymOutsideHalf},  it follows that $M =\dcover{\mathfrak{T}}(\wv)$  for any nonzero $\wv\in M$.

Now, assume there is a weight $\mu \in \textup{Wt}_{\mathfrak{t}_0}(M)$ 
satisfying the condition of item~\ref{enum:AllWS4T} and let $\wv\in M_\mu$.
We will show that~\eqref{eq:M=4T} holds. 
Denote by $M'$ the subspace of $M$ in the right-hand side of equation~\eqref{eq:M=4T}. It suffices to show that for any generator $O(v)$ of $\tama$ (with $v\in \bigwedge^2(V) \cup \bigwedge^3(V)$), we have $O(v)(M')\subseteq M'$, since this shows that $M'$ is a $\tama$-module and hence it must be all of $M$ by irreducibility. Furthermore, from \cref{lem:actRootsWeights,lem:LadDontCrossAxis}, it follows that all possible $\mathfrak{t}_0$-weights occurring in the right-hand side of~\eqref{eq:M=4T} are distinct and hence the sum must be indeed  direct.  

From the factorisation of Proposition~\ref{prop:triangstruc}, and since elements of $\TriSubAlg_0$ act as a scalar on a weight vector, a general element of $M'$ is written as a sum of vectors of the type 
\begin{equation}\label{eq:GenVecT4}
    \hwv' = \dcover{w}L_\Upsilon (\wv'),
\end{equation}
where $L_\Upsilon=L_{\Upsilon_1}\cdots L_{\Upsilon_m}$
for $\Upsilon=(\Upsilon_1,\dots, \Upsilon_m)$ a sequence with $\Upsilon_j\in\Phi$, $\wv'\in\{\wv,O_{\ib{-+-}{122}}(\wv),O_{\ib{+--}{112}}(\wv),O_{\ib{--}{12}}(\wv)\}$ and $\dcover{w}\in\dcover{W}$.  
We show by induction on $m=|\Upsilon|$, the size of the sequence $\Upsilon$, that $O(v)(\hwv')\in M'$ for any $\hwv'$ as in \eqref{eq:GenVecT4} and any $v \in \bigwedge^2(V)$ or $v \in \bigwedge^3(V)$. 

We do the base case, that is, when $m=0$. We only need to consider $\wv'=\wv$ since we can use \cref{thm:TAMArel,lem:TT} to reduce to this case. The results of the action of $\sym{\ib{--}{12}}$, $\sym{\ib{-+-}{122}}$ and $\sym{\ib{+--}{112}}$ on $\wv$ are in $M'$ by construction, and the results of the action of $\sym{\ib{++}{12}}$, $\sym{\ib{++-}{122}}$ and $\sym{\ib{+-+}{112}}$ on $\wv$ are in $M'$ by \cref{prop:ActionTwoThreeSymOutsideHalf}. Thus, it only remains to show that $\sym{\ib{-+}{12}}\wv, \sym{\ib{+-}{12}}\wv\in M'$. To this end, we consider the anticommutation relations coming from the equality \eqref{eq:TT2} when acting on $\wv$:
\begin{subequations}
\label{eq:BaseCase}
    \begin{alignat}{2}
\acomm{\sym{\ib{++-}{211}}}{\sym{\ib{-+-}{122}}} &= 8\sym{\ib{-+}{12}} + 2 \acomm{\sym{\ib{+}{2}}}{\sym{\ib{-+-}{122}}} - 2\acomm{\sym{\ib{-}{1}}}{\sym{\ib{++-}{211}}} + 4\acomm{\sym{\ib{-}{1}}}{\sym{\ib{+}{2}}}, \label{eq:BaseCaseOmp}\\
\acomm{\sym{\ib{-+-}{211}}}{\sym{\ib{++-}{122}}} &= -8\sym{\ib{+-}{12}} + 2 \acomm{\sym{\ib{+}{1}}}{\sym{\ib{-+-}{211}}} - 2\acomm{\sym{\ib{-}{2}}}{\sym{\ib{++-}{122}}} + 4\acomm{\sym{\ib{+}{1}}}{\sym{\ib{-}{2}}}. \label{eq:BaseCaseOpm}
     \end{alignat}
\end{subequations}
Then, \eqref{eq:BaseCaseOmp} shows that $\sym{\ib{-+}{12}}\wv\in M'$, and  \eqref{eq:BaseCaseOpm} shows that $\sym{\ib{+-}{12}}\wv\in M'$.

Now assume by induction that for all $v\in\bigwedge^2(V)\cup \bigwedge^3(V)$ we have $O(v)(\hwv')\in M'$ for any $\hwv'$ as in \eqref{eq:GenVecT4}, whenever $|\Upsilon| \leq m$. Without loss of generality we can assume $\dcover{w} = 1$. Let $\Upsilon=(\Upsilon_0,\Upsilon_1,\dots, \Upsilon_m)$ and  $\Upsilon'=(\Upsilon_1,\dots, \Upsilon_m)$. 
Note that if $m>0$ and $\Upsilon_m\in\Phi_{\bar 0}$, then using Lemma \ref{lem:LadDontCrossAxis}\ref{item:2indWS2}, only one out of the four possible options for $L_{\Upsilon_m}$ satisfies $L_{\Upsilon_m}(\wv')\neq 0$. Suppose $\Upsilon_m$ is that option. 
In the case where $\Upsilon_j = \Upsilon_m$ for all $1\leq j <m$,
the weight space where $L_{\Upsilon}(\wv')$ lands has a well-defined action of $O(v)$ by \cref{prop:ActionTwoThreeSymOutsideHalf}. So we can assume that $\Upsilon$ is not of the type $(\Upsilon_m,\ldots,\Upsilon_m)$.

In addition, we claim that we can always assume that $\Upsilon_0 \in\Phi_{\bar{1}}$. Indeed, suppose that $\Upsilon_0\in\Phi_0$. There is a minimal index $1\leq j \leq m$ such that $\Upsilon_j\neq \Upsilon_0$, for otherwise we would have $L_\Upsilon(\wv')=0$, or $\Upsilon = (\Upsilon_0,\ldots,\Upsilon_0)$ for $\Upsilon_0$ the unique root in $\Phi_0$ for which $L_{\Upsilon_0}(\wv')\neq 0$, which is a situation we already treated above. Suppose that $\Upsilon_j\in\Phi_0$. If $\Upsilon_j=-\Upsilon_0$, then by Proposition \ref{prop:EvenEvenLadders}, $L_{\Upsilon_0}L_{\Upsilon_j}\in\TriSubAlg_0$ and hence after sending that abelian term to the right, $L_\Upsilon$ can be expressed as a sum of monomials all with fewer ladder elements and we would be done by the inductive hypothesis. If, on the other hand, $\Upsilon_j\notin\{\Upsilon_0,-\Upsilon_0\}$, then using Proposition \ref{prop:ExtraTriangRules}\ref{prop:ExtraTriangRulesItemiii} we have
$L_{\Upsilon_0}L_{\Upsilon_j} = L_\beta^2p$, where $\beta\in\Phi_{\bar 1}$ forms an obtuse angle with $\Upsilon_0$ and $p\in\TriSubAlg$. But in this situation, using Proposition \ref{prop:ExtraTriangRules}\ref{prop:ExtraTriangRulesItemii}, we can commute $L_\beta$ with all the $\Upsilon_0$ so that $L_\Upsilon$ can be expressed as a sum of monomials all starting with $L_\beta$. 

All that said, fix $v\in\bigwedge^2(V)\cup\bigwedge^3(V)$. Assume both that $L_{\Upsilon_0}\in\Phi_{\bar 1}$ and that our inductive hypothesis hold.
Then,
\begin{equation}\label{eq:commindT4}
 O(v) L_\Upsilon(\wv') = L_{\Upsilon_0}O(v) L_{\Upsilon'}(\wv') + [O(v), L_{\Upsilon_0}]L_{\Upsilon'}(\wv').   
\end{equation}
Since $L_{\Upsilon_0}(M')\subseteq M'$,
the term $L_{\Upsilon_0}O(v) L_{\Upsilon'}(\wv')\in M'$ by the inductive hypothesis. As for the term $[O(v), L_{\Upsilon_0}]L_{\Upsilon'}(\wv')$,
if $v\in\bigwedge^3(V)$, using Lemma \ref{lem:LT}, this commutator is expressed as a linear combination of monomials $\dcover{w}O(u) p$, with $\dcover{w}\in\dcover{W}$, $u\in\bigwedge^3(V)$ and $p\in\TriSubAlg_0$, for single symmetry operators $O(u)$. Hence, by induction, $[O(v), L_{\Upsilon_0}]L_{\Upsilon'}(\wv')$ is also in $M'$. Thus, $O(v)(M')\subseteq M'$ whenever $v\in\bigwedge^3(V)$.

Suppose next that $v\in\bigwedge^2(V)$ and let $\Upsilon_0 = \delta\varpi_a$, so that $L_{\Upsilon_0} = L_{\ib{\delta}{a}}$. Using Lemma \ref{lem:LadderAlt}, we have
\begin{align*}
[O(v),L_{\ib{\delta}{a}}] &= [O(v),\Tepsa{\delta}(2H_a+\delta) + 4O_a^{\delta}(\delta H_b+Z)]  \\
&=[O(v),\Tepsa{\delta}](2H_a+\delta) + 
2\Tepsa{\delta}[O(v),H_a] + 
4[O(v),O_a^{\delta}(\delta H_b+Z)].
\end{align*}
By the inductive hypothesis $XL_{\Upsilon'}(\wv')\in M'$ whenever $X$ is $[O(v),\Tepsa{\delta}],[O(v),H_a]$ or $[O(v),O_a^{\delta}(\delta H_b+Z)]$, since each of these commutators are expressed as a linear combination of monomials with at most one symmetry operator $O(u)$ (see~\eqref{eq:LongShortO-comm} for $[O(v),\Tepsa{\delta}]$). Since, moreover, $\Tepsa{\delta}$ is a 3-index symmetry, this also preserves $M'$, by what was shown in the previous paragraph. Hence, we also have $\Tepsa{\delta}[O(v),H_a]L_{\Upsilon'}(\wv')\in M'$. This finishes the proof of item~\ref{enum:AllWS4T}.

We now pass to item~\ref{enum:AllWS2T}. Without loss of generality, we suppose $a=1$; the case $a=2$ is similar. Assume $\mu$ is a weight such that $\mu(H_1)  = 1/2$ and $\mu(\dr 1) = \zeta_1^{\ell_1}$ with $\ell_1 \not\equiv 1 \modSB m_1$. First, if there is no weight $\lambda$ such that $\lambda(H_2) \in \{-\tfrac12, \tfrac12\}$, then we show that $M' = \dcover{\mathfrak T}(\wv)\oplus \dcover{\mathfrak{T}}(\sym{\ib{++-}{122}}\wv)$. Indeed, we see that $\sym{\ib{-+-}{211}} \wv \in \dcover{\TriSubAlg}( \wv)$ by \cref{prop:ActionTwoThreeSymOutsideHalf}. For $\sym{\ib{--}{12}}\wv$, we use \cref{lem:TT} as in \cref{eq:BaseCase} to express its values in terms of elements in $\dcover{\TriSubAlg}(\wv) \oplus \dcover{\TriSubAlg}(\sym{\ib{-+-}{122}}\wv)$. Then the rest follows by the proof of item~\ref{enum:AllWS4T}.

Now if there is a weight $\lambda(H_2) = \varepsilon\tfrac{1}{2}$ and $\lambda(\dr 2) = \zeta_b^{-\eps}$, for $\varepsilon\in \{+,-\}$, then, assuming without loss of generality that $\eps = -$, we see by \cref{lem:LadDontCrossAxis} that, for $\wv\in M_{\lambda}$, $\sym{\ib{-+-}{211}}\wv$ and $\df{2}\wv$ share the same $\mathfrak{t}_0$-weight values, and that $\sym{\ib{--}{12}}\wv$ and $\df{2}\sym{\ib{-+-}{122}}\wv$ share the same $\mathfrak t_0$-weight value. We will show that, in both cases, they are in fact linearly dependent. 

We suppose $\sym{\ib{-+-}{211}}\wv \neq 0$. 
We first denote $\hwv_1:=\df{2}\wv $ and $\hwv_2:=\sym{\ib{-+-}{211}}\wv$. By \cref{lem:TT}, we have $\sym{\ib{++-}{211}}\sym{\ib{-+-}{211}}\wv = D \wv$ for a certain constant $D$. We can suppose $D\neq 0$ since, if it were 0 all the time, then $\dcover{\TriSubAlg}(\sym{\ib{-+-}{211}}\wv)$ would be a submodule of $M$, since $\sym{\ib{\pm -}{12}}\sym{\ib{-+-}{211}}\wv=0$ by \cref{lem:TT}, a contradiction of the irreducibility of $M$. 

So we denote $\wv_+ = \hwv_1 + \hwv_2/\sqrt{D}$ and $\wv_- = \hwv_1 - \hwv_2/\sqrt{D}$. Then $\sym{\ib{++-}{211}}\wv_{\pm} = \pm \sqrt{D} \df{2} \wv_{\pm}$. But that would mean $M=\dcover{\TriSubAlg}(\wv_+)\oplus \dcover{\TriSubAlg}(\wv_-)$, for two nonzero submodules, contradicting the irreducibility of $M$. Hence, $\sym{\ib{++-}{211}}\wv$ is proportional to $\df{2}\wv$ and in particular $\sym{\ib{++-}{211}}\wv\in\dcover{\TriSubAlg}(\wv)$. In a similar fashion, we show that $\sym{\ib{--}{12}}\wv$ is in the same weight space as $\df{2}\sym{\ib{-+-}{122}}\wv$ using \cref{lem:TT}. This proves that $M= \dcover{\TriSubAlg}(\wv) \oplus \dcover{\TriSubAlg}(\sym{\ib{-+-}{122}}\wv)$.

The last case remaining is if there are weight $\lambda$ with $\lambda(H_a) = \varepsilon/2$ and $\lambda(\dr{a}) = \zeta_a^{\ell_a}$ with $\ell_a \equiv \varepsilon \modSB m_a$ without being in the situation of item~\ref{enum:AllWS2T}. Then, we follow the same proof as above to prove that $M=\dcover{\TriSubAlg}(\wv)$ for any nonzero $\wv$. This concludes the last instances of item~\ref{enum:AllWS1T}, finishing the proof. 
\end{proof}

\begin{remark}
It is a natural question whether we can induce a simple $\tama$-module from a $\widetilde{\TriSubAlg}$-module. \cref{prop:ActionTwoThreeSymOutsideHalf,thm:AllWeightSpace} indicate a way to go, but without having the full presentation by generators and relations of $\tama$, it is not straightforward to define the induction. In light of the recent work in the classical setting~\cite{CDMO25}, this will be pursued  in future work. 
\end{remark}

We can now use the previous theorem to describe in more detail the $\mathfrak{h}$- and $\mathfrak{t}_0$-weight spaces of finite-dimensional irreducible representations.

\begin{proposition}\label{prop:WS_1D_If_HWS_1D}
	 For a simple finite-dimensional $\tama$-module $M$, the following hold:
  \begin{enumerate}
 \item 
All $\mathfrak{t}_0$-weight spaces of $M$ are one-dimensional.
\item If for all $\mu \in \textup{Wt}_{\mathfrak{t}_0}(M)$ and $a\in\{1,2\}$, we have $\mu(H_a)\notin \{-\tfrac12,0,\frac12\}$, 
then all $\mathfrak{h}$-weight spaces of $M$ are one-dimensional.
  \end{enumerate}
\end{proposition}
\begin{proof}
Let $\hwv \in M$ be a nonzero weight vector. 
Denote by $\mathfrak{T}(\hwv)$ the subspace of $M$ obtained by applying the triangular subalgebra $\TriSubAlg$ on the vector. Because of \cref{prop:OddOppLadders,prop:EvenEvenLadders,prop:OddOddLaddersIsEven}, 
the space $\TriSubAlg(\hwv)$ will decompose into one-dimensional $\mathfrak{h}$-weight spaces, where \cref{lem:actRootsWeights} can be used to track down the action of each $L_\alpha\in\TriSubAlg$. %

Now, for $a\in\{1,2\}$, the action of $\df{a} \in \dcover{W}$ on an $H_a$-eigenvector will result in an eigenvector with opposite sign eigenvalue. 
Hence, the only possibility for the action of $\dcover{W}$ on $\TriSubAlg(\hwv)$ to result in $\mathfrak{h}$-weight spaces of dimension higher than 1 in $\dcover{\TriSubAlg}(\hwv)$ is if the eigenvalues of $H_a$ for some $a\in\{1,2\}$ on $\TriSubAlg(\hwv)$ contain both $N$  and its opposite $-N$, where $N$ must then be either an 
integer or half-integer by Lemma~\ref{lem:actRootsWeights}. 
If this is the case, then there must be a $\mu \in \textup{Wt}_{\mathfrak{t}_0}(M)$ and $a\in\{1,2\}$ such that $\mu(H_a)\in \{-\tfrac12,0,\frac12\}$. 
If the latter is not the case, then by \cref{thm:AllWeightSpace},  we have 
$M=\dcover{\mathfrak{T}}(\hwv) $ and  the action of $\dcover{\mathfrak{T}}$ suffices to describe all weight spaces of $M$. 
This proves the second item.

Now, assume the $\mathfrak{h}$-weights of $M$ contain half-integers, but not zero (we will handle this case later). 
In this case, the $\mathfrak{h}$-weight spaces of $M$ may not be one-dimensional.
Within a $\TriSubAlg$-submodule, the eigenvalues of $H_a$ for $a\in\{1,2\}$ cannot change sign because of item (ii) and (iv) of \cref{lem:LadDontCrossAxis}, so the $\mathfrak{h}$-weight spaces of a $\dcover{\mathfrak{T}}$-submodule of $M$ will be one-dimensional.
Using \cref{lem:actRootsWeights,lem:LadDontCrossAxis}, we then see 
that all $\mathfrak{t}_0$-weight spaces in each decomposition of $M$ in \cref{thm:AllWeightSpace} will be one-dimensional. 


Next, we consider the case where $\mu(H_1) = 0$ for some $\mu \in \textup{Wt}_{\mathfrak{t}_0}(M)$. The case $\mu(H_2) = 0$ is similar. We assume for now $\mu(H_2) \neq 0$, and we will handle the case where both are zero at the end of the proof. 
Let $v\in M_\mu$ be such that $H_1 v = 0$. 
We have  $H_1\df{1} v =-\df{1}H_1 v = 0$, while $H_2\df{1} v =\df{1}H_2 v=\mu( H_2 )\df{1}v $
so $v$ and 
$\df{1} v$ are in the same $\mathfrak{h}$-weight space.
We note that, for $v$ and  $\df{1} v$ to be in the same $\mathfrak{t}_0$-weight space, their $Z$-eigenvalue must be zero, and their $\dr{1} $-eigenvalue must be $-1 = \zeta_1^{m_1}$, which can only occur if the dihedral parameter $m_1$ is odd.

Now, assume that $v$ and 
$\df{1} v$ are in the same $\mathfrak{t}_0$-weight space.
We will show that they must be linearly dependent. 
We first consider the (sub)case where for every weight $\lambda \in \textup{Wt}_{\mathfrak{t}_0}(M)$, we have 
$\lambda(H_2) \not\in \{-\tfrac12,0,\tfrac12\}$. 
In this case,   $M = \dcover{\mathfrak T}(\wv)$ for any nonzero $\wv \in M$. 
However, denoting $v_\pm = v \pm  \df{1} v$, one observes that $v_\pm$ is not in  $M = \dcover{\mathfrak T}(v_\mp)$, which contradicts $M$ being irreducible. 

Next, assume $\lambda(H_2) \in \{-\tfrac12,\tfrac12\}$ for some $\lambda \in \textup{Wt}_{\mathfrak{t}_0}(M)$. 
We distinguish between two cases depending on the value of $\lambda(\dr 2)$. 
On the one hand, if $\lambda(\dr 2) = \zeta_2^{\ell_2}$ with $\ell_2 \equiv -\varepsilon \modSB m_2$ where $\eps\in\{+,-\}$ such that $\lambda(H_2) = \varepsilon\frac12$, then again $M = \dcover{\mathfrak T}(\wv)$ for any nonzero $\wv \in M$ and we can use the same argument. 
On the other hand, we can use that for a specific choice of $\wv$ the right-hand side of \eqref{eq:M=2T} in \cref{thm:AllWeightSpace} can only contain one of $ v_\pm$ but not the other, which again contradicts $M$ being irreducible. 

Finally, if $\lambda(H_2) = 0$ for some $\lambda \in \textup{Wt}_{\mathfrak{t}_0}(M)$, then, if both $m_1$ and $m_2$ are odd, we could have a weight $\mu = (0,0,0,\zeta_1^{m_1},\zeta_2^{m_2}) \in \textup{Wt}_{\mathfrak{t}_0}(M)$, which will be invariant under the actions of $\df{1}$, $\df2$ and also $\df 1\df2$. Using the system of equations in \cref{thm:ClassiLabel}, it follows that such a weight cannot occur in a finite-dimensional irreducible representation, since the chains $\{A(\mu,\alpha,K)\}_{K\in\NN}$ of \eqref{eq:A} will not terminate for the same value of $\kappa$ for all negative $\alpha$. 

To conclude, in all cases, $v$ and $\df{1} v$ must be linearly dependent, and since $M$ is irreducible and can be generated from $v$, all $\mathfrak{t}_0$-weight spaces will be one-dimensional.
\end{proof}

\begin{example}\label{ex:OneHalf}
For $u\in\{1,\ldots,\lfloor m_1/2 \rfloor -1\}$, let $(\rho_u,V_u)$ be the two-dimensional representation of $\dih_{2m_1}$ with basis $\{z^+,z^-\}\subset V_u$ where $\rho_u(s_p)(z^\pm)=\zeta_1^{2up}z^\mp$, with $\zeta_1 = e^{i\pi/m_1}$, for all $p=1,\ldots,m_1$.
    Consider the standard module $M_\kappa(\tau)$ of $H_\kappa$ where $\tau = \rho_u\otimes\textup{triv}$.  Let $M^0=\bbc\otimes V_\tau \otimes \bbs$ be the degree $0$ space inside $M_c(\tau)\otimes \bbs$. Since $M^0$ is in the kernel of $\underline{D}$, this is an eight-dimensional $\tama$-module which is also a $*$-unitary $\tama$-module. We omit the detailed computations and summarise the situation in what follows. First of all, all ladder operators $L_\alpha$ act by $0$ on $M^0$. We consider the ordered basis $\{\hwv_1,\ldots,\hwv_8\}$ explicitly given by
    \[
    \begin{aligned}
       \hwv_1&:=1\otimes z^+\otimes 1,& 
       \hwv_2&:=1\otimes z^-\otimes \bar\theta_1,&
       \hwv_3&:=1\otimes z^+\otimes \bar\theta_2,&
       \hwv_4&:=1\otimes z^-\otimes \bar\theta_1\wedge \bar\theta_2,\\
       \hwv_5&:=1\otimes z^-\otimes 1,& 
       \hwv_6&:=1\otimes z^+\otimes \bar\theta_1,& v_7&:=1\otimes z^-\otimes \bar\theta_2,&
       \hwv_8&:=1\otimes z^+\otimes \bar\theta_1\wedge \bar\theta_2.
    \end{aligned}
    \]
    
    We have two cases to consider. If $ 2u\neq m_1$, then $Q_1(2u)=0$ and the vectors $\hwv_1,\ldots,\hwv_8$ are weight vectors with corresponding $(H_1,H_2,Z,\dcover{r_1},\dcover{r_2})$-eigenvalues $\nu_1,\ldots,\nu_8$ given by 
    \[
    \begin{aligned}
       \nu_1&=(\lambda_1,\lambda_2,\Lambda,\zeta_1^{2u+1},\zeta_2),& 
       \nu_5&=(\lambda_1,\lambda_2,\Lambda,\bar\zeta_1^{2u-1},\zeta_2),\\
       \nu_2&=(-\lambda_1,\lambda_2,-\Lambda,\bar\zeta_1^{2u+1},\zeta_2), &
       \nu_6&=(-\lambda_1,\lambda_2,-\Lambda,\zeta_1^{2u-1},\zeta_2),\\
       \nu_3&=(\lambda_1,-\lambda_2,-\Lambda,\zeta_1^{2u+1},\bar\zeta_2),& 
       \nu_7&=(\lambda_1,-\lambda_2,-\Lambda,\bar\zeta_1^{2u-1},\bar\zeta_2),\\
       \nu_4&=(-\lambda_1,-\lambda_2,\Lambda,\bar\zeta_1^{2u+1},\bar\zeta_2),&  
       \nu_8&=(-\lambda_1,-\lambda_2,\Lambda,\zeta_1^{2u-1},\bar\zeta_2),
    \end{aligned}
    \]
    with $\lambda_1 = \tfrac{1}{2}$, $\lambda_2 = \tfrac{1}{2} + Q_2(0)$ and $\Lambda = -\tfrac{1}{2} - \lambda_1-\lambda_2$. Recall that $\zeta_a = e^{i\pi m_a}$ for $a \in\{ 1,2\}$. One checks that 
    \[
      \begin{aligned}
    O_{\ib{--}{12}}(\hwv_1) &= 2(1+Q_2(0))\hwv_8,&& 
    O_{\ib{+-}{12}}(\hwv_2) &= -2(1+Q_2(0))\hwv_7,\\
    O_{\ib{-+}{12}}(\hwv_3) &= 2(1+Q_2(0))\hwv_6,&&
    O_{\ib{++}{12}}(\hwv_4) &= -2(1+Q_2(0))\hwv_5.
     \end{aligned}
    \]
    Provided that $\kappa$ is such that $1+Q_2(0) \neq 0$, this module is an irreducible $\tama$-module which decomposes as $M^0 = U_1\oplus U_2$ when viewed as a $\dcover{W}$-representation, with $U_1=\textup{span}\{\hwv_1,\hwv_2,\hwv_3,\hwv_4\}$ and $U_2=\textup{span}\{\hwv_5,\hwv_6,\hwv_7,\hwv_8\}$. Note that the $\mathfrak{h}$-weight spaces are two-dimensional in this situation, while all $\mathfrak{t}_0$-weights are distinct. If $1+Q_2(0) = 0$, this module is irreducible if and only if $Q_2(2) \neq 0$.

    In the case when $2u=m_1$, as $\zeta_1$ is a $2m_1$-root of unity, the behaviour of the spinor weights change since $\zeta_1^{2u+1}=\bar\zeta_1^{2u-1}$. Furthermore, the vectors $\hwv_1,\ldots,\hwv_8$ are not necessarily weight vectors, depending on whether $Q_1(2u) = 0$ or not (so depending on the parameter $\kappa$). We thus consider the basis $\{\wv_1^+,\ldots,\wv_4^+,\wv_1^-,\ldots,\wv_4^-\}$ with $\wv_j^\pm := \hwv_j \pm \hwv_{4+j}$. The corresponding $\mathfrak{t}_0$-weights $\mu_1^\pm,\ldots,\mu_4^\pm$ are given by 
    \[
    \begin{aligned}
       \mu^+_1&=(\lambda_1^+,\lambda_2,\Lambda^+,\zeta_1^{2u+1},\zeta_2),& 
       \mu^-_1&=(\lambda_1^-,\lambda_2,\Lambda^-,\zeta_1^{2u+1},\zeta_2),\\
       \mu^+_2&=(-\lambda_1^+,\lambda_2,-\Lambda^+,\bar\zeta_1^{2u+1},\zeta_2), &
       \mu^-_2&=(-\lambda_1^-,\lambda_2,-\Lambda^-,\bar\zeta_1^{2u+1},\zeta_2),\\
       \mu^+_3&=(\lambda_1^+,-\lambda_2,-\Lambda^+,\zeta_1^{2u+1},\bar\zeta_2),& 
       \mu^-_3&=(\lambda_1^-,-\lambda_2,-\Lambda^-,\zeta_1^{2u+1},\bar\zeta_2),\\
       \mu^+_4&=(-\lambda_1^+,-\lambda_2,\Lambda^+,\bar\zeta_1^{2u+1},\bar\zeta_2),& 
       \mu^-_4&=(-\lambda_1^-,-\lambda_2,\Lambda^-,\bar\zeta_1^{2u+1},\bar\zeta_2),
    \end{aligned}
    \]
with $\lambda^\pm_1 = \tfrac{1}{2}\pm Q_1(2u)$, $\lambda_2 = \tfrac{1}{2} + Q_2(0)$ and $\Lambda^\pm = -\tfrac{1}{2} - \lambda_1^\pm -\lambda_2$. In the present case when $2u=m_1$, the module $M^0$ is reducible as a $\tama$-module and decomposes as $M^0 = U^+\oplus U^-$ where $U^\pm = \textup{span}\{u_1^\pm,\ldots,u_4^\pm\}$ are irreducible $\tama$-modules that satisfy $U^+\cong U^-$ as $\dcover{W}$-representations. Both $\tama$-irreducible summands have one-dimensional $\mathfrak{t}_0$-weights. When $1+2Q_1(2u)=0$, then $\lambda^+=-\lambda^+=0\in\bbz$ and in this situation the irreducible module $U^+$ has $\mathfrak{h}$-spaces of dimension $2$. Finally, when $Q_1(2u)=0$, we have $\lambda_1^+=\lambda^-$ (and $\Lambda^+=\Lambda^-)$, the vectors $\hwv_1,\ldots,\hwv_8$ are weight vectors and $U^+\cong U^-$ as $\tama$-modules. 
\end{example}

Recall that a finite-dimensional $\tama$-module $M$ is called $*$-unitary if it is equipped with a positive-definite sesquilinear pairing (which we agree to be conjugate-linear in the second entry) that satisfies $\langle L u,v\rangle = \langle u,L^*v\rangle$, for any $u,v\in M$ and $L\in\tama$. Unitary modules enjoy the following properties.

\begin{proposition}\label{prop:UnitModule}
    Let $M$ be a finite-dimensional $*$-unitary $\tama$-module. Then the following hold.
    \begin{enumerate}
        \item If $\mu,\lambda$ are distinct elements in $\textup{Wt}_{\mathfrak{t}_0}(M)$, then $M_\mu \perp M_\lambda$.
        \label{enum:Unit1}
         \item When restricted to a weight space  $M_\lambda$, a ladder operator $L_\alpha$ either acts invertibly or by zero.
        \label{enum:Unit2}
    \end{enumerate}
\end{proposition}

\begin{proof}
Given $T\in \mathfrak{T}_0 = \alg A(\mathfrak{t}_{0})$ and $\lambda \in \textup{Wt}_{\mathfrak{t}_0}(M)$, let $T_\lambda$ denote the scalar by which $T$ acts on $M_\lambda$. For any $T\in \mathfrak{T}_0$, $\wv\in M_\mu$, $\hwv\in M_\lambda$, with $\lambda,\mu \in \textup{Wt}_{\mathfrak{t}_0}(M)$, by~\cref{prop:StarStrucTriangAlg}, we have
\[
T_\mu\langle \wv,\hwv \rangle = 
\langle T\wv,\hwv \rangle = \langle \wv,T^*\hwv \rangle
=T_\lambda\langle \wv,\hwv \rangle
\]
and hence $\mu\neq \lambda$ implies $\langle \wv,\hwv \rangle =0$  proving~\ref{enum:Unit1}.  From  \cref{prop:OddOppLadders,prop:EvenEvenLadders}, we have that $L_{-\alpha}L_\alpha \wv = C_\lambda \wv$ for all $\wv\in M_\lambda$, for some scalar $C_\lambda$. If $C_\lambda = 0$ then, using \cref{prop:StarStrucTriangAlg}, we have $0 = \langle L_{-\alpha}L_\alpha \wv,\wv \rangle= -(-1)^{|\alpha|}\langle L_\alpha \wv, L_\alpha \wv\rangle$, which implies $L_\alpha \wv =0$ so~\ref{enum:Unit2} is proven. 
\end{proof}

\subsection{Triangle-representations}

    It is possible to refine the coarse classification of \cref{thm:ClassiLabel} by studying all candidate tuples respecting the equations~\eqref{eq:thmclassrel}.
    The classification for $W=D_{2m}\times \ZZ_2$ proceeded in this way, explicitly verifying the relations of the algebra; see~\cite{DBLROVdJ22}.
 
    Given a label for an abstract vector $\hwv$, we can define a vector space from $\dcover{\TriSubAlg}(\hwv)$ and give it an action of $\tama$ by using \cref{prop:ActionTwoThreeSymOutsideHalf} in most cases. However, without having the full presentation by generators and relations of $\tama$ we cannot guarantee that such a vector space endowed with a linear action of the generators descends to a module of $\tama$.

Nevertheless, using that $(\tama,\mathfrak{g})$ is a commuting pair of subalgebras of $H_\kappa\otimes \ca C$, in favourable situations we \emph{can} guarantee the realisation of a finite-dimensional $\tama$-module inside some $(H_\kappa\otimes \ca C)$-module.  

We focus on specific representations, those having a highest-weight vector $\hwv$ such that $\Ld \hwv = 0 = \Mcd \hwv$ and $\Nmm=\Mam$.  They will contain the motivating example of the polynomial monogenic representations and, for $\kappa$ sufficiently small, will also cover the only possibilities; see~\cref{thm:ClassiSmallKappa}. For positive $\kappa$, we will also be able to give an explicit basis for the representations, to be compared with the explicit realisation of monogenics in~\cite{DBLROVdJ23}.

Before giving the definition, we note that the relations $\Ld \hwv= 0 = \Mcd \hwv$ give rise to two extra conditions on the weights of the label of \cref{thm:ClassiLabel}: $A(\mu,-\varpi_2,0)=0=A(\mu,-\varpi_1-\varpi_2,0)$. In addition, the weight space contour will be triangular; see \cref{fig:triangle}. This inspired the name of this class of representations.

\begin{definition}\label{def:mono-type}
    We call a $\tama$-representation $\mathcal{V}$ a \emph{triangle-representation} if,  for a nonnegative integer $N\in \NN$, it has 
     a highest weight of the form $\mu=(\lambda_1,\lambda_2,\Lambda,\zeta_1^{\ell_1},\zeta_2^{\ell_2}) \in  \mathfrak{t}_0^*$ 
     where
     \begin{equation}\label{eq:triWeight}
        \lambda_1 = N+ \tfrac12 \pm Q_1(\ell_1 -2N-1), \quad \lambda_2 = \tfrac12 \pm Q_2(\ell_2-1)\quad \Lambda = -1/2 - \lambda_1 - \lambda_2,
    \end{equation}
   and if
 $\Ld \hwv = 0 = \Mcd \hwv$ for $\hwv \in \mathcal{V}_\mu$.
\end{definition}

Unitary representations have extra structure attached to them and if their highest weight respects~\eqref{eq:triWeight}, then they are a triangle-representation.

\begin{proposition}\label{prop:UnitaryWithWeightAreTri}
    A unitary representation $\mathcal{V}$ with a  highest weight $\mu = (\lambda_1,\lambda_2,\Lambda,\zeta_1^{\ell_1},\zeta_2^{\ell_2})$ given by~\eqref{eq:triWeight} for a certain integer $N$ is a triangle-representation.
\end{proposition}
\begin{proof}
 The only thing remaining to check for $\mathcal V$ to be triangle-representation is that $\Ld\hwv = 0 = \Mcd\hwv$. If $N=0$, all ladder operators act as $0$. If $N>0$, from the value of $\lambda_1$, $\lambda_2$ and $\Lambda$, we have $\Lb\Ld\hwv = 0 = \Mab\Mcd\hwv$ by the factorisations of \cref{prop:OddOppLadders,prop:EvenEvenLadders}. Then item~\ref{enum:Unit2} of \cref{prop:UnitModule}  implies that $\Ld\hwv = 0 = \Mcd\hwv$.
 \end{proof}

When $\kappa$ is in a sufficiently small neighbourhood of $0$, there are further restrictions on the possible labels that a finite-dimensional irreducible representation can have. The following definition makes precise what is meant by a sufficiently small neighbourhood of $0$ in our context.
\begin{definition}\label{def:smallkappa}
	We call the parameter function $\kappa$ \emph{small} if $|Q_a(j)|< \tfrac12$ for all $j\in \ZZ$ and $a\in \{1,2\}$.
\end{definition}

When $\kappa$ is small, the only finite-dimensional simple $\tama$-modules are the triangle-representations. 

\begin{theorem}\label{thm:ClassiSmallKappa}
	If $\kappa$ is small, then any finite-dimensional irreducible $\tama$-representation is a triangle-representation.
\end{theorem}
\begin{proof}
Let $\mathcal{V}$ be a finite-dimensional irreducible $\tama$-representation with highest $\mathfrak{t}_0$-weight $\mu =(\lambda_1,\lambda_2,\Lambda,\zeta_1^{\ell_1},\zeta_2^{\ell_2})$.
We want to solve the system $A(\mu,\alpha,0)=0$ for $\alpha\in\Phi_+$ of \cref{thm:ClassiLabel}. We focus first on  $\{\varpi_1,\varpi_2,\varpi_1+\varpi_2\}
\subset \Phi_+$. Since $0\leq |Q_1(\ell_1-1)|, |Q_2(\ell_2-1)|< 1/2$, the system of equations $A(\mu,\varpi_1,0)=0$, $A(\mu,\varpi_2,0)=0$ and $A(\mu,\varpi_1+\varpi_2,0)=0$  is equivalent to the following set of quadratic equations:
	\begin{equation}\label{eq:soluhwvgen}
		\begin{gathered}
			(\Lambda + \lambda_2)^2 = (\lambda_1+1/2)^2, \quad
			(\Lambda+\lambda_1)^2 = (\lambda_2 + 1/2)^2, \quad
			(\lambda_1+\lambda_2 + 1)^2 - (\Lambda-1/2)^2 =0.
		\end{gathered}
	\end{equation}
	The solutions to~\eqref{eq:soluhwvgen} must respect $\Lambda = -1/2 - \lambda_1 -\lambda_2$. Then the remaining equation $A(\mu,\varpi_1-\varpi_2,0)=0$ is satisfied either by $\lambda_2 = 1/2$ or $\lambda_2 = 1/2\pm Q_2(\ell_2+1)$. For both instances, $\Mbm=0=\Nmm$.

	Let $N:= \Nmp$. We turn to equation $A(\mu,-\varpi_1+\varpi_2,N)=0$. As $\Lambda = -1/2-\lambda_1-\lambda_2$,  the only way for it to hold is by having
	\begin{align*}
		\lambda_1 - N -1/2 = \pm Q_1(\ell_1+1-2N), \text{ so } \lambda_1 = 1/2+N \pm Q_1(\ell_1+1-2N),
	\end{align*}
	since $(\lambda_2 +N +1/2)^2 - Q_2(\ell_2-1+2N)^2 \neq 0$ and $(\lambda_1 - \lambda_2 - 2N -1)^2 - (\lambda_1 + \lambda_2)^2 \neq 0$. 	Then we also have $A(\mu,-\varpi_1,N)=0$, and furthermore, $A(\mu,-\varpi_1,j)\neq 0$ for $0\leq j < N$ as
	\begin{align*}
		((N - j \pm Q_1(\ell_1+1-2N))^2 - Q_1(\ell_1+1-2j)^2 &\neq 0\\
		((-1)^j(N+1/2 \mp Q_1(\ell_1+1-2N) \mp Q_2(\ell_2+1)) + 1/2 \pm Q_2(\ell_2+1))^2 &\neq 		(N-j \pm Q_1(\ell_1 +1-2N))^2,
	\end{align*}
	which comes from the fact that $Q_a(j)\leq 1/2$. This shows that $\Mam = N$.
    Finally, we show that $\Ld\hwv = 0 = \Mcd\hwv$. Suppose that $\Ld\hwv \neq 0$. Then, since $\Mad\hwv =0$, we apply  \cref{prop:OddOddLaddersIsEven} on $\hwv$ to get $\La\Ld\hwv = 4\Mad (Z-H_1 - H_2 +\tfrac{1}{2})\hwv =0$ since $\kappa$ small means $\Lambda - \lambda_1 - \lambda_2 + \tfrac{1}{2}$ is nonzero. We consider then $(\Lc\La)\Ld\hwv$. Using \cref{prop:OddOppLadders} we find
    \begin{equation}\label{eq:LcLaLdiszero}
        0=(\Lc\La)\Ld\hwv=((\lambda_1 + \tfrac12)^2 - Q_1(\ell_1-1)^2)((-\Lambda +\lambda_2)^2 - (\lambda_1 + \tfrac12)^2)\Ld\hwv.
    \end{equation}
   Replacing the values for $\Lambda$, $\lambda_1$ and $\lambda_2$ in~\eqref{eq:LcLaLdiszero} we find that $\Ld\hwv =0$ since $\kappa$ small prevents the factor to be zero.
    From $\Ld\hwv=0$ and \cref{prop:EvenEvenLadders}, we find $\Mcd\hwv=0$ because $\kappa$ is small. The representation $\mathcal{V}$ is a triangle-representation.    
\end{proof}

  In particular, when $\kappa=0$,  we find $\lambda_1 = N+1/2$, $\lambda_2 = 1/2$, and $\Lambda = -1/2-\lambda_1-\lambda_2$, and $\Mam=N = \Nmp$, $\Mbm =0=\Nmm$. If, in addition to being small, a parameter is in the open neighbourhood of $0$ in the unitary loci of all standard modules (in the sense of~\cite{ES09}), we can guarantee that there exists a triangle-representation of $\tama$ realised inside a $(H_{\kappa}\otimes \Clif)$-module using unitarity.

\begin{theorem}\label{thm:SmallKappaRealisation}
    
    Suppose $\kappa$ is small and taken in a neighbourhood of the $0$ parameter where all $H_{\kappa}$-modules $M_\kappa(\tau)$ are unitarisable. Every finite-dimensional irreducible $\tama$-representation occuring in the kernel of $\DDop$ is a triangle-representation. Such representations can thus be realised inside $M_{\kappa}(\tau)\otimes \bbs$ for the spin module $\bbs$ and a standard module $M(\tau)$.
    
\end{theorem}
\begin{proof}
    From \cite[Propositions~4.3 and 4.4]{ES09} there is a small, convex and open neighbourhood of the $0$ parameter where all $M_\kappa(\tau)$ are irreducible and unitarisable. Hence, its intersection with the locus of small parameters is a non-empty open set. For such standard modules, the unitarity will imply a completely reducible joint-decomposition of the pair $(\tama,\mathfrak{g})$. All finite-dimensional simple $\tama$-modules that occur will be triangle-representations. This proves the existence of triangle-representations realised inside a $(H_{\kappa}\otimes\Clif)$-module.
\end{proof}

Granted that a triangle-representation is realised inside some $M_\kappa(\tau)\otimes \bbs$ and has some further properties, as described in \cref{prop:MonoMaxi}, we can use the action of the $\mathfrak{T}$-algebra to construct a useful basis with explicit action. They are akin to representations of type I in the classification of~\cite{DBLROVdJ22}. The case $\tau = \mathrm{triv}$ is the monogenic case, see \cref{fig:ExKappaToZero} for a pictorial example.

We will assume $\kappa$ positive here. Most negative $\kappa$ can also be considered with minor changes, but some negative exceptional values of $\kappa$, those where the representation theory of the underlying rational Cherednik algebra changes (see~\cite{Ch06}), would modify the behaviour of the representation theory. Some of these exceptional values are considered in Example~\ref{ex:OneHalf}.  

Avoiding those exceptional values also enables the explicit realisation of the representation with the generalised symmetry methods of~\cite{DBLROVdJ23}, since those values are precisely those for which projection operators get ill-defined. 

\begin{proposition}\label{prop:MonoMaxi}
Let $N\in\NN$ and $\kappa$ be positive. Suppose that a finite-dimensional irreducible triangle-representation $\mathcal{V}$  is realised inside some $M_\kappa(\tau)\otimes\bbs$ and that its $\mathfrak{t}_0$-highest weight $\mu=(\lambda_1,\lambda_2,\Lambda,\zeta_1^{\ell_1},\zeta_2^{\ell_2})$ is given by $\lambda_1 := N+1/2 + Q_1(\ell_1 +1 -2 N)$, $\lambda_2 := 1/2 + Q_2(\ell_2+1)$ and $\Lambda := -1/2-\lambda_1-\lambda_2$ with $\ell_1 \equiv 2N+1\modSB m_1$ and $\ell_2 \equiv 1 \modSB m_2$. Then, $\mathcal{V}$ has dimension $2(N+1)(N+2)$, and a highest-weight vector $\hwv$ of weight $\mu$
    defines a basis given by
 \begin{equation}\label{eq:propMonoMaxiBasis}
 	\mathcal{B}:= \Big\{\hwv_{ij}^{\delta\eps}:= (\df2)^{\eps}(\df1)^{\delta}(\Lb)^j(\Lc)^i  \hwv  \mid  \delta,\eps\in \{0,1\},  j\leq i\in \{0,\dots, N\}\Big\}.
 \end{equation}
  \end{proposition}
   \begin{proof}
   The element $\hwv$ defined in the proposition is a highest-weight vector  and so $\La \hwv = \Lb\hwv = 0 = \Mab\hwv = \Mad\hwv$. The value of $\lambda_1$ makes it so that $A(\mu,-\varpi_1+\varpi_2,N)=0$. The values of $\ell_1$ and $\ell_2$ ensure that $\mathcal{V} = \dcover{\TriSubAlg}(\hwv)$ by \cref{thm:AllWeightSpace} item~\ref{enum:AllWS1T}. 
   From the maximality of $Q_1(\ell_1 + 1 -2N)$ and $Q_2(\ell_2+1)$, we have that $A(\mu,-\varpi_1+\varpi_2,j)\neq 0$ for $0\leq j < N$.

   Then the action of $\mathfrak{T}$ on the set~\eqref{eq:propMonoMaxiBasis} is retrieved from \cref{prop:OddOppLadders,prop:EvenEvenLadders,prop:OddOddLaddersIsEven}. The full action of $\tama$ on the $\hwv_{ij}^{\delta\eps}$ is retrieved from \cref{eq:actOijk,eq:actOij}. Indeed, the $H_1$- and $H_2$-weights of elements of $\mathcal{B}$ are given by:
   \begin{align}
   H_1 \hwv_{ij}^{\delta\eps} &= (-1)^{\delta} (\lambda_1-i)\hwv_{ij}^{\delta\eps}, & H_2 \hwv_{ij}^{\delta\eps} &= (-1)^{\eps} (\lambda_2+i)\hwv_{ij}^{\delta\eps},
   \end{align}
   which implies in particular that $(2\lambda_1^{\mathfrak{u}}+\delta)\neq 0$ and $(2\lambda_2^{\mathfrak{u}}+\eps)\neq 0$ for all $u\in \mathcal{B}$, since $|(\lambda_1-i)| > 1/2$ and $|(\lambda_2+i)| >1/2$ for all $0\leq i\leq N$. So the conditions of the proposition are fulfilled.

   Finally, all the vectors of $\mathcal{B}$ have distinct pairs of  $H_1$- and $H_2$-eigenva\-lues, so they are all linearly independent. It is thus a basis of $2(N+1)(N+2)$ elements, proving the proposition.
   \end{proof}

\newcommand{\DIM}{3}
\begin{figure}[h]
    \centering
    \begin{tikzpicture}[scale=.8]
 \foreach \i in {2,...,\DIM}
    {
    \foreach \j in {\i,...,\DIM}
        {
    	\foreach \x/\y in {1/1, 1/-1, -1/1, -1/-1}
             {
            \linkbetween{{\x*(\DIM-\j+1-.5)}}{{\y*(\i-.5}}{{\x*(\DIM-\j+1-.5)}}{{\y*(\i-1-.5)}}{orange}
            \linkbetween{{\x*(\i-.5)}}{{\y*(1-.5)}}{{\x*(\i-1-.5)}}{{\y*(1-.5)}}{teal} }
        }
    }
    \foreach \i in {1,...,\DIM}
        {
        \foreach \j in {\i,...,\DIM}
        {
        \foreach \x/\y in {1/1, 1/-1, -1/1, -1/-1}
        {
    \draw[color = black,fill=gray, opacity=.7]
    ({\x*(\i-.5)},{\y*(\DIM-\j+1-.5)}) circle(3pt);
    }
    }}
    \draw (3-.5,1) node {$\hwv_{00}^{++}$};
    \draw ({-(3-.5)},1) node {$\hwv_{00}^{-+}$};
    \draw ({-(3-.5)},-1) node {$\hwv_{00}^{--}$};
    \draw (3-.5,-1) node {$\hwv_{00}^{+-}$};
    \draw (1,3) node {$\hwv_{22}^{++}$};
    \draw[dotted,->] (-\DIM,0) -- (\DIM,0) node[right] {$\varpi_1$};
    \draw[dotted,->] (0,-\DIM) -- (0,\DIM) node[above] {$\varpi_2$} ;
\end{tikzpicture}
    \caption{An example of a triangle-representation basis with $N=2$ reconstructed in \cref{prop:MonoMaxi}. The horizontal teal lines are applications of $\Lc$, and the vertical orange lines, of $\Lb$, with the appropriate reflections. A node at $(i,j)$ is the weight space of $\mathfrak h$-weight $(\tfrac12+i+Q_1(\ell_1+1-2N),\tfrac12+j+Q_2(\ell_2+1))$.}
    \label{fig:triangle}
\end{figure}

\section*{Acknowledgements}
Part of the results of this paper appeared in the doctoral thesis of ALR. The authors, and ALR in particular, wish to thank Kieran Calvert and Martina Balagovi\'c for many insightful comments on the research presented here given in the course of the thesis evaluation. Finally, we thank the referees for their comments.

\section*{Declarations}
\begingroup
\singlespacing
\subsection*{Ethical Approval} Not applicable.
\subsection*{Competing interests} The authors declare no conflict of interest related to this work.
\subsection*{Authors' contribution} MDM, ALR, and RO took an active role in all parts of the research. MDM, ALR and RO wrote and reviewed the whole manuscript.
\subsection*{Funding}  ALR was supported in part by the EOS Research Project [grant number 30889451] and by scholarships from the Fonds de recherche du Qu\'ebec -- Nature et technologies [grant number 270527 and 326641]. ALR acknowledges the support given under Federal Ministry of Education and Research of Germany and by Sächsische Staatsministerium für Wissenschaft, Kultur und Tourismus 
(project identification number ScaDS.AI) while he worked at ScaDS.AI and the hospitality of the MPI-MiS Leipzig during his stay, and has received funding through his Hausdorff postdoc, which is funded by Deutsche Forschungsgemeinschaft (DFG, German Research Foundation) under Germany's Excellence Strategy -- GZ 2047/1, Projekt-ID 390685813. 
RO was supported by a postdoctoral fellowship, fundamental research, of the Research Foundation -- Flanders (FWO), number 12Z9920N.
MDM gladly acknowledges the support of the special research fund (BOF) from Ghent University [BOF20/PDO/058].
\subsection*{Availability of data and software}
Data sharing is not applicable to this article as no datasets
were generated or analysed during the study.
\endgroup
\printbibliography
\end{document}